%% file: arxiv_v9bis.tex
\documentclass[10pt,reqno,a4paper]{amsart}

\DeclareMathSizes{10}{9}{6}{5} 
\calclayout%

\input{packages}  

\input{notations}  
\input{bibliography_setup.tex}
\usepackage{nameref,zref-xr}
\zxrsetup{toltxlabel}
\zexternaldocument*[cplx-]{cpam_final}
\bibliography{bibclt}

\numberwithin{equation}{section}
\newtheorem{theorem}{Theorem}[section]
\newtheorem{assumption}[theorem]{Assumption}
\newtheorem{lemma}[theorem]{Lemma}
\newtheorem{proposition}[theorem]{Proposition}
\newtheorem{definition}[theorem]{Definition}

\newtheorem{remark}[theorem]{Remark}

\newtheorem{convention}[theorem]{Convention}

\date{\today}
\author{Giorgio Cipolloni\(^{\dagger}\) \and L\'{a}szl\'{o} Erd\H{o}s} 
\address{IST Austria, Am Campus 1, 3400 Klosterneuburg, Austria}
\author{Dominik Schr\"oder\(^{\ast}\)}
\address{Institute for Theoretical Studies, ETH Zurich, Clausiusstr.\ 47, 8092 Zurich, Switzerland}
\email{giorgio.cipolloni@ist.ac.at} 
\email{lerdos@ist.ac.at}
\email{dschroeder@ethz.ch}
\thanks{\(^\dagger\)This project has received funding from the European Union's Horizon 2020 research and innovation programme under the Marie Sk\l odowska-Curie Grant Agreement No.\ 665385.}
\thanks{\(^\ast\)Supported by Dr.\ Max R\"ossler, the Walter Haefner Foundation and the ETH Z\"urich Foundation}
\subjclass[2010]{60B20, 15B52} 
\keywords{Dyson Brownian Motion, Local Law, Girko's Formula, Linear Statistics, Central Limit Theorem}
\title[Fluctuation around the circular law]{Fluctuation around the circular law for  random matrices with real entries}
\date{\today}
\begin{document}
\thispagestyle{empty}
\begin{abstract}
    We extend our recent result~\cite{1912.04100}  on the central limit theorem for the linear eigenvalue statistics
    of non-Hermitian  matrices \(X\) with independent, identically distributed \emph{complex}
    entries  to the \emph{real} symmetry class. We find that
    the expectation and variance  substantially differ from their complex counterparts, reflecting (i) the
    special spectral symmetry of real matrices onto the real axis; and (ii) the fact that real i.i.d.\ matrices have many real
    eigenvalues. %
    Our result generalizes the previously known special cases where either the test function is  analytic~\cite{MR3540493} or
    the first four moments of the matrix elements match the real Gaussian~\cite{MR3306005,1510.02987}. The key element of the proof is the analysis of several weakly dependent Dyson Brownian motions (DBMs). The conceptual novelty of the real case compared with~\cite{1912.04100} is that the correlation structure of the stochastic differentials in each individual DBM is non-trivial, potentially even jeopardising its well-posedness.
\end{abstract}

\maketitle

\section{Introduction}\label{sec:INT}
We consider an  ensemble of \(n\times n\) random matrices \(X\) with  \emph{real} i.i.d.\ entries of zero mean and variance \(1/n\); the corresponding model with \emph{complex} entries has been studied in~\cite{1912.04100}. According to the \emph{circular law}~\cite{MR1428519,MR2409368,MR773436} (see also~\cite{MR2908617}),
the density of the eigenvalues \(\{ \sigma_i\}_{i=1}^n\) of \(X\)
converges to the uniform distribution on the unit disk.
Our main result is that the fluctuation of their linear statistics is Gaussian, i.e.
\begin{equation}\label{lin stat def}
    L_n(f):= \sum_{i=1}^n f(\sigma_i)  - \E \sum_{i=1}^n f(\sigma_i)  \sim \cN (0, V_f)
\end{equation}
converges, as \(n\to \infty\),   to a centred normal distribution  for  regular test functions \(f\) with at least \(2+\delta\) derivatives.
We compute the variance  \(V_f\)  and the next-order deviation of the expectation \(\E \sum_{i=1}^n f(\sigma_i) \)
from the value \(\frac{n}{\pi}\int_{\abs{z}\le 1} f(z)\) given by the circular law. As in the complex case,
both quantities depend on the fourth cumulant of the single entry distribution of \(X\), but in the real case
they also incorporate the spectral symmetry of \(X\) onto the real axis. Moreover, the expectation carries additional
terms, some of them are concentrated around the real axis; a by-product of the approximately \(\sqrt{n}\)  real eigenvalues of \(X\).
For the Ginibre  (Gaussian) case they may be computed from the explicit density~\cite{MR1231689, MR1437734}, but for general distributions they were
not known before.
As expected, the spectral symmetry essentially enhances  \(V_f\) by a factor of two compared with the complex case
but this effect is modified by an additional term involving the fourth cumulant.  Previous works considered
either the  case of analytic test functions \(f\)~\cite{MR2738319, MR3540493} or the (approximately) Gaussian case, i.e.\ when \(X\) is the real Ginibre
ensemble or at least the first four moments of the matrix elements of \(X\) match the Ginibre ensemble~\cite{MR3306005,1510.02987}.
In both cases some terms in the unified formulas  for the expectation and the variance
vanish and thus the combined effect of the spectral symmetry, the eigenvalues
on the real axis, and the role of the fourth cumulant was not detectable in these works.
We remark that  a CLT for polynomial statistics of only the real eigenvalues for real Ginibre matrices
was proven in~\cite{MR3612267}.

In~\cite{MR2346510} the limiting random field \(L(f):=\lim_{n\to\infty}L_n(f)\) for complex Ginibre matrices has been identified as a projection of the \emph{Gaussian free field (GFF)}~\cite{MR2322706}. We extended this interpretation~\cite{1912.04100} to general complex i.i.d.\ matrices with non-negative fourth cumulant and obtained a rank-one perturbation of the projected GFF\@. As a consequence of the CLT in the present paper, we find that in the real case the limiting random field is a version of the same GFF, symmetrised with respect to the real axis, reflecting the fact that complex eigenvalues of real matrices come in pairs of complex conjugates.

In general, proving CLTs for the real symmetry class is considerably harder than for the complex one.
The techniques based upon the first four moment matching~\cite{MR3306005,1510.02987} are insensitive  to the symmetry class, hence these results are obtained
in parallel for both real and complex ensembles.
Beyond this method, however, most results on CLT for non-Hermitian matrices
were restricted to the complex case~\cite{MR4125967, MR1687948, MR3161483, MR2095933, MR2294978,MR2361453}, see the introduction of~\cite{1912.04100} for a detailed history, as well as for references to the analogous CLT problem for Hermitian ensembles and
log-gases.   The special role that the real axis plays in the spectrum of the
real case substantially complicates even the explicit formulas for the Ginibre ensemble both for the density~\cite{MR1437734}
as well as for the \(k\)-point correlation functions~\cite{MR173726, MR2530159,MR2185860}.
Besides the complexity of the explicit formulas, there are several conceptual reasons why the real case is more involved.
We now explain them since they directly motivated the new ideas in this paper compared with~\cite{1912.04100}.

In~\cite{1912.04100} we
started with Girko's formula~\cite{MR773436} in the form given in~\cite{MR3306005}
that relates the eigenvalues of \(X\) with
resolvents of a family of \(2n\times 2n\) Hermitian matrices
\begin{equation}\label{eq:herm}
    H^z:= \begin{pmatrix}
        0                   & X-z \\
        X^\ast-\overline{z} & 0
    \end{pmatrix}
\end{equation}
parametrized by \(z\in \C  \). For any smooth, compactly supported test function \(f\) we have
\begin{equation}\label{girko}
    \sum_{i=1}^n f(\sigma_i) = -\frac{1}{4\pi} \int_{\C } \Delta f(z)\int_0^\infty \Im \Tr  G^z(i\eta)\dif\eta \dif^2 z,
\end{equation}
where \(G^z(w):=  (H^z-w)^{-1}\) is the resolvent of \(H^z\). We therefore needed to understand the resolvent  \(G^z(i\eta)\)
along the imaginary axis on all scales \(\eta\in (0,\infty)\).

The main contribution to~\eqref{girko} comes from the  \(\eta\sim 1\) \emph{macroscopic} regime, which is handled
by proving a multi-dimensional CLT for resolvents with several \(z\) and \(\eta\) parameters and computing their expectation and covariance
by cumulant expansion. The local laws along the imaginary axis from~\cite{MR3770875, 1907.13631} serve as a basic input (in the current work, however, we need
to extend them for spectral parameters \(w\) away from the imaginary axis).
The core of the argument  in the real case is similar to the complex case in~\cite{1912.04100},    however several additional terms have to
be computed due to the difference between the real and complex cumulants. By explicit calculations, these additional terms break the rotational
symmetry in the \(z\) parameter and, unlike in the complex case,  the answer is not a function of \(\abs{z}\) any more.
The \emph{mesoscopic} regime \(n^{-1}\ll \eta\ll 1\) is treated together with the macroscopic one; the fact that only  the \(\eta\sim 1\) regime contributes to~\eqref{girko} is revealed \emph{a posteriori} after these calculations.

The scale \(\eta\lesssim n^{-1}\) in~\eqref{girko} requires a very different treatment since local laws are not applicable any more and individual eigenvalues
\(0\le \lambda_1^z\le \lambda_2^z \ldots\)
of \(H^z\)    near zero substantially influence the fluctuation of \(G^z(i\eta)\)
(since \(H^z\) has a symmetric spectrum, we consider only positive eigenvalues).
The main insight of~\cite{1912.04100} was that it is sufficient to establish  that the small eigenvalues, say,
\(\lambda_1^z\) and \(\lambda_1^{z'}\), are asymptotically independent if \(z\) and \(z'\) are relatively
far away, say \(\abs{z-z'}\ge n^{-1/100}\). This was achieved by exploiting the fast local equilibration
mechanism of the \emph{Dyson Brownian motion (DBM)}, which is the stochastic flow of eigenvalues  \({\bm \lambda}^z(t):= \{ \lambda_i^z(t)\}\)
generated by adding a  time-dependent Gaussian (Ginibre) component. The initial condition of this flow was chosen carefully
to almost reproduce \(X\) after a properly tuned short time.  We needed to follow the evolution
of \({\bm \lambda}^z(t)\) for different \(z\) parameters  simultaneously. These flows
are correlated since they are driven by the same random source.  We thus needed
to study a family of DBMs, parametrized by \(z\), with correlated driving Brownian motions.
The correlation structure is given by the \emph{overlap} of the eigenfunctions
of \(H^z\) and \(H^{z'}\).  We could show that this overlap is small, hence the Brownian
motions are essentially independent, if \(z\) and \(z'\) are far away. This step required
to develop a new type of local law for \emph{products} of resolvent,
e.g.\ for  \(\Tr G^z( \ii\eta) G^{z'}( \ii\eta')\) with \(\eta, \eta'\sim  n^{-1+\epsilon}\).
Finally, we trailed the joint evolution  of \({\bm \lambda}^z(t)\)
and \({\bm \lambda}^{z'}(t)\) by their independent Ginibre counterparts, showing that
they themselves are asymptotically independent.

We follow the same strategy in the current paper for the real case, but  we  immediately face with the basic
question: how do the low lying eigenvalues of \(H^z\), equivalently the small singular values of \(X-z\),
behave?   We do not need to compute their joint distribution,
but we need to approximate them with an appropriate Ginibre ensemble.
For \emph{complex} \(X\) in~\cite{1912.04100} the approximating Ginibre ensemble was naturally complex. For \emph{real} \(X\)
there seem to be two possibilities. The  key insight of our current analysis  is  that the small
singular values of \(X-z\) behave as those of  a \emph{complex} Ginibre matrix even though \(X\) is
\emph{real}, as long as \(z\) is genuinely complex (Theorem~\ref{theo:un}). In particular, we prove that the least singular value of \(X-z\) belongs to the complex universality class. Moreover, we prove that the small singular values of \(X-z_1\) and the ones of \(X-z_2\) are asymptotically independent as long as \(z_1\) and \(z_2\) are far from each other.

To explain the origin of this apparent mismatch, we will derive the DBM
\begin{equation}
    \label{eq:DDBM}
    \dif \lambda_i^z=\frac{\dif b_i^z}{\sqrt{n}}+\frac{1}{2n}\sum_{j\ne i} \frac{1+\Lambda_{ij}^z}{\lambda_i^z-\lambda_j^z}\dif t  +\ldots
\end{equation}
for \({\bm \lambda}^z(t)\), ignoring some additional terms with negative indices coming from the spectral symmetry of \(H^z\) (see~\eqref{eq:impnewDBM} and~\eqref{eq:newdbmrc}
for the precise equation).
The correlations of the driving Brownian motions are given by
\begin{equation}
    \label{eq:bbc}
    \E \dif b_i^z \dif b_j^{z'} = \frac{1}{2} \bigl[ \Theta_{ij}^{z,z'}+ \Theta_{ij}^{z,\overline{z}'}\bigr]\dif t
\end{equation}
with overlaps \(\Theta,\Lambda\) defined as
\begin{equation}
    \label{eq:ovint}
    \Theta_{ij}^{z, z'}: = 4\Re\bigl[\braket{ {\bm u}^{z'}_j, {\bm u}_i^z } \braket{  {\bm v}^z_i,  {\bm v}^{z'}_j}\bigr], \qquad
    \Lambda_{ij}^z: = \Theta_{ij}^{z,\overline{z}},
\end{equation}
where \(({\bm u}_i^z,  {\bm v}_i^z)\in \C^{2n}\) is
the (normalized) eigenvector of \(H^z\) corresponding to the eigenvalue \(\lambda^z_{ i}\). Note that \(\Theta^{z,z}_{ij} =\delta_{i,j}\),
and for \(j \ne  i \) we have that \(\Lambda^z_{ij}\approx 0\).
Moreover,  if \(z\) is very close to the real axis, then
the eigenvectors of \(H^z\) are essentially real and \(\Lambda_{ii}^z=\Theta_{ii}^{z,\overline{z}}\approx \Theta_{ii}^{z,z}=1\).
With \(z=z'\), this leads to~\eqref{eq:DDBM} being essentially a \emph{real} DBM with \(\beta=1\).
(We recall that  the  parameter \(\beta=1,2\), customarily indicating the real or complex
symmetry class of a random matrix,  also expresses  the ratio of the coefficient of the repulsion to
the strength of the diffusion   in the DBM setup.)
However, if  \(z\) and \(\bar z\) are far away, i.e.\ \(z\) is away from
the real axis, then we can show that the overlap \(\Lambda^z=\Theta^{z,\bar z}\) is small, hence \(\Lambda_{ij}^z \approx 0\) for all \(i,j\), including \(i=j\). Thus the variance of the
driving Brownian motions in~\eqref{eq:bbc} with \(z=z'\) is reduced by a factor of two, rendering~\eqref{eq:DDBM} essentially
a \emph{complex} DBM  with \(\beta=2\).

The appearance of \(\Lambda^z\) in~\eqref{eq:DDBM} and the second term \(\Theta^{z,\overline{z}'}\) in~\eqref{eq:bbc} is specific to the real symmetry class; they were not present
in the complex case~\cite{1912.04100}. They have three main effects for our analysis. First, they change the symmetry class of the
DBM~\eqref{eq:DDBM} as we just explained. Second, due to the symmetry relation \(\lambda_{-1}^z = -\lambda_1^z\) and \(b_{-1}^z=-b_{1}^z\),
the strength of the level repulsion  between \(\lambda_1^{z}\) and \(\lambda_{-1}^z\)  in~\eqref{eq:DDBM} is already critically small
even for \(\Lambda^z=0\), see e.g.~\cite[Appendix A]{MR3916329},
hence the well-posedness of~\eqref{eq:DDBM} does not follow from standard results on DBM\@.
Third,  \(\Theta^{z,\overline{z}}\) renders the driving Brownian motions \(\bm{b}^z=\{ b^z_i \}\) correlated for different
indices \(i\) even for the \emph{same} \(z\), since \(\Lambda^z_{ij}\) in general is nonzero.
In fact,  the vector \({\bm b}^z\) is even not Gaussian, hence strictly speaking it is only a multidimensional martingale but not a Brownian motion
in general.
In contrast, \(\Theta^{z,z}_{ij} = \delta_{i,j}\) and only the  overlaps
\smash{\(\Theta_{ij}^{z, z'}\)} for \emph{different} \(z\ne z'\) are nontrivial.
Thus in the complex case~\cite{1912.04100}, lacking the term \(\Theta^{z,\overline{z}}\) in~\eqref{eq:bbc}, the DBM~\eqref{eq:DDBM} for any fixed \(z\) was the conventional  DBM with independent
Brownian motions and parameter \(\beta=2\) (c.f.~\cite[Eq.~\eqref{cplx-eq:DBMeA}]{1912.04100})
and only the DBMs for \emph{different} \(z\)'s were  mildly correlated. In the real case
the correlations are already present within~\eqref{eq:DDBM} for the \emph{same} \(z\) due to \(\Lambda^z=\Theta^{z,\overline{z}}\ne 0\).

We note that Dyson Brownian motions with  nontrivial coefficients in the repulsion term have already been investigated
in~\cite{MR4009717}  (see also~\cite{1908.04060}) in the context of spectral universality of addition of random matrices twisted by Haar unitaries,
however the driving Brownian motions were independent. The issue of well-posedness, nevertheless, has already emerged  in~\cite{MR4009717} when
the more critical orthogonal group (\(\beta=1\)) was considered.  The corresponding part of our analysis partly relies on techniques
developed in~\cite{MR4009717}. We have already treated the dependence of Brownian motions for different \(z\)'s in~\cite{1912.04100}
for the complex case;  but the more general dependence structure characteristic to
the real case is  a new challenge that the current work resolves.

\subsection*{Acknowledgement} We are grateful to Peter Forrester for pointing out a missing term in \eqref{eq:expv} in the original manuscript.

\subsection*{Notations and conventions}
We introduce some notations we use throughout the paper. For integers \(k\in\N\) we use \([k]:= \set{1,\dots, k}\). We write \(\HC \) for the upper half-plane \(\HC := \set{z\in\C \given \Im z>0}\), \(\DD\subset\C\) for the open unit disk, and we use the notation \(\dif^2 z:= 2^{-1} \ii(\dif z\wedge \dif \overline{z})\) for the two dimensional volume form on \(\C\).
For positive quantities \(f,g\) we write \(f\lesssim g\) and \(f\sim g\) if \(f \le C g\) and \(c g\le f\le Cg\), respectively, for some constants \(c,C>0\) which depend only on the \emph{model parameters} appearing in~\eqref{eq:hmb}.
For any two positive real numbers \(\omega_*,\omega^\ast\in  \R_+\), by \(\omega_*\ll\omega^\ast\) we denote that \(\omega_*\le c\omega^\ast\) for some sufficiently small constant \(0<c\le 1/1000\).
We denote vectors by bold-faced lower case Roman letters \(\vx, \vy,\dots,\in\C^k\), for some \(k\in\N\), and use the notation \(\dif {\bm x}:=\dif x_1 \dots \dif x_k\). Vector and matrix norms, \(\norm{\vx}\) and \(\norm{A}\), indicate the usual Euclidean norm and the corresponding induced matrix norm. For any \(k\times k\) matrix \(A\) we set \(\braket{A} := k^{-1}\Tr  A\) to denote the normalized trace of \(A\).
Moreover, for vectors \(\vx,\vy\in\C^k\) and matrices \(A,B\in \C^{k\times k}\) we define
\[ \braket{\vx,\vy} := \sum \ov{x_i} y_i, \qquad \braket{ A,B}:= \braket{A^\ast B}=\frac{1}{k}\Tr A^\ast B.\]
We will use the concept of ``event with very high probability'' meaning that for any fixed \(D>0\) the probability of the event is bigger than \(1-n^{-D}\) if \(n\ge n_0(D)\). Moreover, we use the convention that \(\xi>0\) denotes an arbitrary small exponent which is independent of \(n\).

\section{Main results}\label{sec:MR}
We consider \emph{real i.i.d.\ matrices} \(X\), i.e.\ \(n\times n\) matrices whose entries are independent and identically distributed as \smash{\(x_{ab}\stackrel{\mathsf{d}}{=} n^{-1/2}\chi\)} for some real random variable \(\chi\), satisfying the following:
\begin{assumption}\label{ass:1}
    We assume that \(\E \chi=0\) and \(\E \chi^2=1\). In addition we assume the existence of high moments, i.e.\ that there exist constants \(C_p>0\), for any \(p\in\N \), such that
    \begin{equation}
        \label{eq:hmb}
        \E \abs{\chi}^p\le C_p.
    \end{equation}
\end{assumption}

The \emph{circular law}~\cite{MR1428519, MR863545, MR2908617, MR2663633, MR3813992, MR866352, MR773436, MR2575411, MR2409368} asserts that the empirical distribution of eigenvalues \({\{\sigma_i\}}_{i=1}^n\) of a complex i.i.d.\ matrix \(X\) converges to the uniform distribution on the unit disk \(\DD\), i.e.
\begin{equation}
    \label{eq:circlaw}
    \lim_{n\to \infty}\frac{1}{n}\sum_{i=1}^n f(\sigma_i)=\frac{1}{\pi}\int_\DD f(z)\dif^2z,
\end{equation}
with very high probability for any continuous bounded function \(f\). Our main result is a central limit theorem for the centred \emph{linear statistics}
\begin{equation}
    \label{eq:linstatmainr}
    L_n(f):= \sum_{i=1}^n f(\sigma_i)-\E \sum_{i=1}^n f(\sigma_i)
\end{equation}
for general real i.i.d.\ matrices and generic test functions \(f\), complementing the recent central limit theorem~\cite{1912.04100} for the linear statistics of \emph{complex i.i.d.\ matrices}. This CLT, formulated in Theorem~\ref{theo:CLT}, and its proof have two corollaries of independent interest that are formulated in Section~\ref{sec:GFF} and Section~\ref{sec:UNZ}.

In order to state the result we introduce some notations. For any function \(h\) defined on the boundary of the unit disk \(\partial\DD \) we define its Fourier transform as
\begin{equation}
    \label{eq:furtra}
    \wh{h}(k)=\frac{1}{2\pi}\int_0^{2\pi} h(e^{\ii\theta}) e^{-\ii \theta k}\dif\theta, \qquad k\in\Z .
\end{equation}
For \(f,g\in H^{2+\delta}(\Omega)\) for some domain \(\Omega\supset \ov\DD\) we define
\begin{equation}
    \label{eq:h12norm}
    \begin{split}
        \braket{ g,f}_{\dot{H}^{1/2}(\partial\DD )}&:= \sum_{k\in\Z }\abs{k} \ov{\wh{g}(k)}\wh{f}(k) , \qquad \norm{f}^2_{\dot{H}^{1/2}(\partial\DD )}:= \braket{f,f}_{\dot{H}^{1/2}(\partial\DD )}, \\
        \braket{g,f}_{{H}_0^{1}(\DD)}&:= \braket{\nabla g,\nabla f}_{L^2(\DD)}, \qquad \norm{f}_{{H}_0^{1}(\DD )}^2:= \braket{ f,f}_{{H}_0^{1}(\DD )},
    \end{split}
\end{equation}
where, in a slight abuse of notation, we identified \(f\) and \(g\) with their restrictions to \(\partial \DD \). We use the convention that \(f\) is extended to \(\mathbf{C}\) by setting it equal to zero on \(\Omega^c\). Finally, we introduce the projection
\begin{equation}\label{Psym def}
    (P_\sym f)(z) := \frac{f(z)+f(\ov z)}{2}.
\end{equation}
which maps functions on the complex plane to their symmetrisation with respect to the real axis.

\begin{theorem}[Central Limit Theorem for linear statistics]\label{theo:CLT}
    Let \(X\) be a real \(n\times n\) i.i.d.\ matrix satisfying Assumption~\ref{ass:1} with eigenvalues \(\{\sigma_i\}_{i=1}^n\), and denote the fourth \emph{cumulant}\footnote{Note that in the real case the fourth cumulant is given by \(\kappa_4=\kappa(\chi,\chi,\chi,\chi)=\E\chi^4-3\), while in the complex case~\cite{1912.04100} the relevant fourth cumulant was given by \(\kappa(\chi,\chi,\ov\chi,\ov\chi)=\E\abs{\chi}^4-2\).} of \(\chi\) by \(\kappa_4:= \E\chi^4-3\). Fix \(\delta>0\), let \(\Omega\subset \C \) be open and such that \(\overline{\DD }\subset \Omega\). Then, for complex-valued test functions \(f\in H^{2+\delta}(\Omega)\), the centred linear statistics \(L_n(f)\), defined in~\eqref{eq:linstatmainr}, converge
    \[ L_n(f) \Longrightarrow L(f), \]
    to complex Gaussian random variables \(L(f)\) with expectation \(\E L(f)=0\) and variance \(\E \abs{L(f)}^2=C(f,f)=: V_f\) and \(\E L(f)^2 = C(\ov f, f)\), where
    \begin{equation}\label{eq:cov}
        \begin{split}
            C(g,f)&:= \frac{1}{2\pi}\braket{ \nabla P_\sym g,\nabla P_\sym f}_{L^2(\DD )}+\braket{P_\sym g,P_\sym f}_{\dot{H}^{1/2}(\partial\DD )} \\
            &\quad + \kappa_4 \left(\frac{1}{\pi}\int_\DD \overline{g(z)}\dif^2 z- \frac{1}{2\pi}\int_0^{2\pi} \overline{g(e^{\ii\theta})}\dif \theta\right)\left(\frac{1}{\pi}\int_\DD f(z)\dif^2 z-\frac{1}{2\pi}\int_0^{2\pi} f(e^{\ii\theta})\dif\theta\right).
        \end{split}
    \end{equation}
    For the \(k\)-th moments we have an effective convergence rate of
    \[ \E L_n(f)^k \ov{L_n(f)}^l = \E L(f)^k \ov{L(f)}^l + \landauO*{n^{-c(k+l)}}\]
    for some constant \(c(k+l)>0\). Moreover, the expectation in~\eqref{eq:linstatmainr} is given by
    \begin{equation}
        \label{eq:expv}
        \begin{split}
            \E \sum_{i=1}^n f(\sigma_i)={}& E( f) + \landauO*{n^{-c}} \\
            E(f) :&={} \frac{n}{\pi}\int_\DD f(z) \dif^2 z  + \frac{1}{4\pi}\int_{\DD } \frac{f(\Re z)-f(z)}{(\Im z)^2}\dif^2z -\frac{\kappa_4}{\pi}\int_\DD f(z) (2\abs{z}^2-1)\dif^2 z \\
            &\quad +\frac{1}{8\pi}\int_\DD\Delta f(z)\,\dif^2z -\frac{1}{2\pi}\int_0^{2\pi}  f(e^{\ii\theta}) \dif \theta + \frac{1}{2\pi}\int_{-1}^{1} \frac{f(x)}{\sqrt{1-x^2}}\dif x
            +\frac{f(1)+f(-1)}{4}
        \end{split}
    \end{equation}
    for some small constant \(c>0\).
\end{theorem}
\begin{remark}\leavevmode
    \begin{enumerate}[label= (\roman*)]
        \item Both expectation \(E(f)\) and covariance \(C(g,f)\) only depend on the symmetrised functions \(P_\sym f\) and \(P_\sym g\). Indeed, \(E(f)=E(P_\sym f)\), and the coefficient of \(\kappa_4\) in~\eqref{eq:cov} can also be written as an integral over \(P_\sym f\) and \(P_\sym g\).
        \item By polarisation, a multivariate central limit theorem as in~\cite[Corollary~\ref{cplx-cor:multCLT}]{1912.04100} follows immediately and any mixed \(k\)-th moments have an effective convergence rate of order \(n^{-c(k)}\).
        \item The variance \(V_f=\E\abs{L(f)}^2\) in Theorem~\ref{theo:CLT} is strictly positive whenever \(f\) is not constant on the unit disk (see~\cite[Remark~\ref{cplx-remark Vf pos}]{1912.04100}).
    \end{enumerate}
\end{remark}
\begin{remark}[Comparison with~\cite{1510.02987} and~\cite{MR3540493}]\leavevmode
    \begin{enumerate}[label= (\roman*)]
        \item The central limit theorem~\cite[Theorem 2]{1510.02987} is a special case of Theorem~\ref{theo:CLT}. Indeed,~\cite[Theorem 2]{1510.02987} implies that for real i.i.d.\ matrices with entries matching the real Ginibre ensemble to the fourth moment, and real-valued smooth test functions \(f\) compactly supported within the upper half of the unit disk \(L_n(f)\) converge to a real Gaussian of variance
              \begin{equation}
                  \label{eq:symnosym}
                  \frac{1}{4\pi} \braket{\nabla f,\nabla f}_{L^2(\DD)} = \frac{1}{2\pi} \braket{\nabla P_\sym f,\nabla P_\sym f}_{L^2(\DD)},
              \end{equation}
              where we used that \(z\mapsto f(z)\) and \(z\mapsto f(\ov z)\) are assumed to have disjoint support. Due to the moment matching assumption, \(\kappa_4=0\) in the setting of~\cite{1510.02987}.
        \item The central limit theorem~\cite[Corollary~2.6]{MR3540493} is also a special case of Theorem~\ref{theo:CLT}. Indeed,~\cite[Corollary~2.6]{MR3540493} implies that for real i.i.d.\ matrices and test functions \(f\) which are analytic in a neighbourhood of the unit disk and satisfy \(P_\sym f\colon\ov\DD\to\R\) the linear statistics \(L_n(f)\) converge to a Gaussian of variance
              \[
                  \begin{split}
                      \frac{1}{\pi} \int_\DD \abs{\partial_z f(z)}^2 \dif^2 z&= \frac{1}{4\pi} \braket{\nabla f,\nabla f}_{L^2(\DD)} + \frac{1}{2}\braket{f,f}_{\dot H^{1/2}(\partial\DD)}\\
                      &=\frac{1}{2\pi} \braket{\nabla P_\sym f,\nabla P_\sym f}_{L^2(\DD)} + \braket{P_\sym f,P_\sym f}_{\dot H^{1/2}(\partial\DD)}.
                  \end{split}\]
              Here in the first step we used the analyticity of \(f\) (see~\cite[Eq.~\eqref{cplx-eq:chan}]{1912.04100}), and in the second step we used that \(\braket{(\nabla f)(z),(\nabla f(\ov\cdot))(z)}=0\) and that \(\wh f(k)=0\) for \(k<0\) while \(\wh{f(\ov\cdot)}(k)=0\) for \(k>0\) by analyticity. We thus arrived at~\eqref{eq:cov}, since the coefficient of \(\kappa_4\) in~\eqref{eq:cov} vanishes also by analyticity of \(f\) in the setting of~\cite{MR3540493}.
    \end{enumerate}
\end{remark}

\begin{remark}[Comparison with the complex case]
    We remark that the limiting variance in the case of complex i.i.d.\ matrices, as studied in~\cite{1912.04100}, is generally different from the real case. In the complex case \(L_n(f)\) converges to a complex Gaussian with variance
    \[\begin{split}
            V_f^{(\C)} ={}& V_f^{(\C,1)} + \kappa_4 V_f^{(\C,2)},\\
            V_f^{(\C,1)}:={}& \frac{1}{4\pi} \norm{\nabla f}^2_{L^2(\DD)}+\frac{1}{2}\norm{f}^2_{\dot H^{1/2}(\partial(\DD))},\quad V_f^{(\C,2)}:= \abs{\braket{f}_{\DD}-\braket{f}_{\partial\DD}}^2,
        \end{split}\]
    where \(\braket{\cdot}_{\mathbf{D}}\) denotes the averaging over \(\mathbf{D}\) as in~\eqref{eq:cov}. In contrast, in the real case the limiting variance is given by
    \[ V_f^{(\R)} = 2 V_{P_\sym f}^{(\C,1)} + \kappa_4 V_{f}^{(\C,2)}.\]
    Thus the variances agree exactly in the case of analytic test functions by~\eqref{eq:symnosym} and \(V_{f}^{(\C,2)}=0\), while e.g.\ in the case of symmetric test functions, \(f=P_\sym f\) and vanishing fourth cumulant \(\kappa_4=0\) the real variance is twice as big as the complex one, \(V_f^{(\R)}=2 V_f^{(\C)}\).
\end{remark}

\begin{remark}[Real correction to the expected circular law]
    In~\cite[Theorem 6.2]{MR1437734} Edelman computed the density of genuinely complex eigenvalues of the real Ginibre ensemble to be
    \begin{equation}\label{eq:edelmanbulk}
        \rho_n(x+\ii y) := \sqrt{\frac{2n}{\pi}}\abs{y}e^{2n y^2} \erfc(\sqrt{2n}\abs{y}) \frac{\Gamma(n-1,n(x^2+y^2))}{\Gamma(n-1)}
    \end{equation}
    in terms of the \emph{upper incomplete Gamma function \(\Gamma(s,x)\)}.
    Using the large \(n\) asymptotics uniform in \(z=x+\ii y\) for the incomplete Gamma function~\cite[Eq.~(2.2)]{MR1421474} we obtain
    \[\rho_n(z) \approx  \sqrt{\frac{2n}{\pi}}\abs{\Im z}e^{2n (\Im z)^2} \erfc(\sqrt{2n}\abs{\Im z}) \erfc\Bigl(\sgn(\abs{z}-1)\sqrt{n(\abs{z}^2-1-2\log\abs{z})}\Bigr), \]
    which, using asymptotics of the error function for any fixed \(\abs{z}<1\),
    \[\sqrt{\frac{2n}{\pi}}\abs{\Im z}e^{2n (\Im z)^2} \erfc(\sqrt{2n}\abs{\Im z}) \approx \frac{1}{2\pi} - \frac{1}{8n\pi(\Im z)^2},\]
    gives that
    \[\rho_n(z) = \frac{1}{\pi} -\frac{1}{4\pi n}\frac{1}{(\Im z)^2} + \landauo{n^{-1}},\]
    in agreement with the second term in the rhs.\ of~\eqref{eq:expv} accounting for the \(n^{-1}\)-correction to the circular law away from the real axis.

    The situation very close to the real axis is much more subtle. The density of the real Ginibre eigenvalues is explicitly known~\cite[Corollary 4.3]{MR1231689} and it is asymptotically uniform on \([-1,1]\), see~\cite[Corollary 4.5]{MR1231689}, giving a singular correction of mass of order \(n^{-1/2}\) to the circular law. However, the abundance of real eigenvalues is balanced by the sparsity of genuinely complex eigenvalues in a narrow strip around the real axis --- a consequence of the factor \(\abs{y}\) in~\eqref{eq:edelmanbulk}. Since these two effects of order \(n^{-1/2}\)  cancel each other on the scale of our test functions \(f\), they are not directly visible in~\eqref{eq:expv}. Instead we obtain a smaller order correction of order \(n^{-1}\) specific to the real axis, in form of the second, the penultimate and the ultimate term in~\eqref{eq:expv}.
\end{remark}

\begin{remark}[Special case: Polynomial test functions]
    We remark that in~\cite{Forrester_2009,MR2515561} exact \(n\)-dependent formulae for \(\E \Tr X^k=\E \sum_i \sigma_i^k\) and real Ginibre \(X\) have been obtained. Translated into our scaling it follows from~\cite[Corollary 4]{Forrester_2009} that
    \begin{equation}\label{eq forrester}
        \E \Tr X^k = \begin{cases}
            1, & k\text{ even}, \\ 0, & k\text{ odd},
        \end{cases}
        + \landauok{1}
    \end{equation}
    for integers \(k\ge 1\), as \(n\to\infty\) (note that the trace is unnormalised). The asymptotics~\eqref{eq forrester} are consistent with~\eqref{eq:expv} since
    \[ \int_\DD \Delta z^k\,\dif^2 z=0, \quad\int_\DD z^k \dif^2 z = 0, \quad  \int_{-1}^1 (e^{\ii\theta})^k \dif \theta = 0,\quad \frac{1^k+(-1)^k}{4}=\begin{cases}
            \frac{1}{2}, & k\text{ even}, \\ 0,&k \text{ odd},
        \end{cases}\]
    and
    \[ \frac{1}{4\pi}\int_\DD \frac{(\Re z)^k-z^k}{(\Im z)^2}\dif^2 z =  \begin{cases}
            \frac{1}{2}-2^{-k}\binom{k-1}{k/2}, & k\text{ even}, \\ 0,&k \text{ odd},
        \end{cases},\quad \frac{1}{2\pi}\int_{-1}^1 \frac{x^k}{\sqrt{1-x^2}}\dif  x =  \begin{cases}
            2^{-k}\binom{k-1}{k/2}, & k\text{ even}, \\ 0,&k \text{ odd}.
        \end{cases} \]
\end{remark}

\subsection{Connection to the Gaussian free field}\label{sec GFF}\label{sec:GFF}

It has been observed in~\cite{MR2346510} that for complex Ginibre matrices the limiting random field \(L(f)\) can be viewed as a projection of the \emph{Gaussian free field (GFF)}~\cite{MR2322706}. In~\cite[Section~\ref{cplx-sec GFF}]{1912.04100} we extended this interpretation to general complex i.i.d.\ matrices with \(\kappa_4\ge 0\) and provided an interpretation as a rank-one perturbation of the projected GFF\@. The real case yields the symmetrised version of the same GFF with respect to the real axis, reflecting the fact that the complex eigenvalues of real matrices come in pairs of complex conjugates. We keep the explanation brief due to the similarity to~\cite[Section~\ref{cplx-sec GFF}]{1912.04100}.

The Gaussian free field on \(\C\) is a \emph{Gaussian Hilbert space} of random variables \(h(f)\) indexed by functions in the Sobolev space \(f\in H_0^1(\C)\) such that the map \(f\mapsto h(f)\) is linear and
\begin{equation}\label{eq GFF def}
    \E h(f) = 0 ,\quad \E \ov{h(f)} h(g) = \braket{f,g}_{H_0^1(\C)} = \braket{\nabla f,\nabla g}_{L^2(\C)}.
\end{equation}
The Sobolev space \(H^1_0(\C)=\ov{C_0^\infty(\C)}^{\norm{\cdot}_{H^1_0(\C)}}\) can be orthogonally decomposed into
\[ H_{0}^1(\DD) \oplus  H_{0}^1(\ov\DD^c)  \oplus  H_0^1(\DD\cup\ov\DD^c)^\perp, \]
i.e.\ the \( H^{1}_0\)-closure of smooth functions which are compactly supported in \(\DD\) or \(\ov\DD^c\), and their orthogonal complement \smash{\( H_0^1((\partial\DD)^c)^\perp\)}, the closed subspace of functions analytic outside of \(\partial\DD\) (see e.g.~\cite[Thm.~2.17]{MR2322706}). With the orthogonal projection \(P\) onto the first and third of these subspaces,
\[ P:= P_{ H_{0}^1(\DD)} + P_{ H_0^1((\partial\DD)^c)^\perp},\]
we have (see~\cite[Eq.~\eqref{cplx-eq f decomp}]{1912.04100})
\begin{equation}\label{eq f decomp}
    \begin{split}
        \norm*{Pf}_{ H_0^1(\C)}^2 &= \norm{ f }^2_{ H_0^1(\DD)} + 2\pi \norm{f }_{\dot H^{1/2}(\partial\DD)}^2.
    \end{split}
\end{equation}
If \(\kappa_4\ge0\), then \(L\) can be interpreted as
\begin{equation}\label{eq GFF interp}
    L = \frac{1}{\sqrt{2\pi}} P P_\sym h + \sqrt{\kappa_4} \Bigl(\braket{\cdot}_\DD-\braket{\cdot}_{\partial\DD}\Bigr)\Xi,
\end{equation}
where \(\Xi\) is a standard real Gaussian, independent of \(h\), and the projection of \(h\) is to be interpreted by duality, i.e.\ \((PP_\sym h)(f):= h(PP_\sym f)\), cf.~\cite[Eq.~\eqref{cplx-eq L kappa4 pos}]{1912.04100}. Indeed,
\[
    \E \abs*{\frac{1}{\sqrt{2\pi}}h(PP_\sym f) + \sqrt{\kappa_4}(\braket{f}_\DD-\braket{f}_{\partial\DD})\Xi}^2 = C(f,f),
\]
as a consequence of~\eqref{eq GFF def} and~\eqref{eq f decomp}.

\subsection{Universality of the local singular value statistics of \texorpdfstring{\(X-z\)}{X-z} close to zero}\label{sec:UNZ}

As a by-product of our analysis we obtain the universality of the small singular values of \(X-z\), and prove that (up to a rescaling) their distribution asymptotically agrees with the singular value distribution of a \emph{complex} Ginibre matrix \(\widetilde{X}\) if \(z\notin \mathbf{R}\), even though \(X\) is a \emph{real} i.i.d.\ matrix. In the following by \(\{\lambda_i^z\}_{i\in [n]}\) we denote the singular values of \(X-z\) in increasing order.

It is natural to express universality in terms of the \(k\)-point correlation functions \(p_{k,z}^{(n)}\) which are defined implicitly by
\begin{equation}
    \label{eq:defcorrfunc}
    \E \binom{n}{k}^{-1} \sum_{\{i_1,\dots, i_{k}\}\subset [n]} f( \lambda_{i_1}^z, \dots, \lambda_{i_k}^z) =\int_{\R^k} f({\bm x}) p_{k,z}^{(n)}({\bm x})\, \dif {\bm x},
\end{equation}
for test functions \(f\). The summation in~\eqref{eq:defcorrfunc} is over all the subsets of \(k\) distinct integers from \([n]\). Denote by \(p_k^{(\infty,\C)}\) the scaling limit of the \(k\)-point correlation function \(p_k^{(n,\C)}\) of the singular values of a complex \(n\times n\) Ginibre matrix \(\widetilde{X}\). See e.g.~\cite[Eqs.~(2.3)--(2.4)]{MR1236195} or~\cite[Eq.~(1.3)]{MR2162782} for the explicit expression of \(p_k^{(\infty,\C)}\).

\begin{theorem}[Universality of small singular values of \(X-z\)]\label{theo:un}
    Fix \(z\in \C\) with \(\abs{\Im z}\sim 1\), and \(\abs{z}\le 1-\epsilon\), for some small fixed \(\epsilon>0\). Let \(X\) be an i.i.d.\ matrix with real entries satisfying Assumption~\ref{ass:1}, and denote by \(\rho^z\) the self consistent density of states of the singular values of \(X-z\) (see~\eqref{eq:scdos} later). Then for any \( k\in\mathbf{N}\), and for any compactly supported test function \(F\in C_c^1(\R^k)\), it holds
    \begin{equation}
        \int_{\R^k}F({\bm x}) \left[ \rho^z(0)^{-k}p_{k,z}^{(n)}\left( \frac{{\bm x}}{n\rho^z(0)}\right) -p_k^{(\infty,\C)}({\bm x})\right]\dif {\bm x}=\mathcal{O}\left( n^{-c(k)}\right),
    \end{equation}
    where \(c(k)>0\) is a small constant only depending on \(k\). The implicit constant in \(\mathcal{O}(\cdot)\) may depend on \(k\), \(\norm{ F}_{C^1}\), and \(C_p\) from~\eqref{eq:hmb}.
\end{theorem}

\begin{remark}
    Theorem~\ref{theo:un} states that the local statistics of the singular values of \(X-z\) close to zero, for \(\abs{\Im z}\sim 1\), asymptotically agree with the ones of a complex Ginibre matrix \(\widetilde{X}\), even if the entries of \(X\) are real i.i.d.\ random variables. It is expected that the same result holds for all (possibly \(n\)-dependent) \(z\) as long as \(\abs{\Im z}\gg n^{-1/2}\), while in the opposite regime \(\abs{\Im z}\ll n^{-1/2}\) the local statistics of the \emph{real} Ginibre prevails with an interpolating family of new statistics which emerges for \(\abs{\Im z}\sim n^{-1/2}\).
\end{remark}

Besides the universality of small singular values of \(X-z\), our methods also allow us to conclude the asymptotic independence of the small singular values of \(X-z_1\) and those of \(X-z_2\) for generic \(z_1,z_2\). More precisely, similarly to~\eqref{eq:defcorrfunc}, we define the correlation function \(p_{k_1,z_1;k_2,z_2}^{(n)}\) for the singular values of \(X-z_1\) and \(X-z_2\) implicitly by
\begin{equation}
    \label{eq:newcorrf}
    \E \binom{n}{k_1}^{-1}\binom{n}{k_2}^{-1} \sum_{\substack{\{i_1,\dots, i_{k_1}\}\subset [n] \\ \{j_1,\dots, j_{k_2}\}\subset [n]}} f({\bm \lambda}_{\bm i}^{z_1}, {\bm \lambda}_{\bm j}^{z_2}) =\int_{\R^{k_1}}\dif {\bm x}_1\int_{\R^{k_2}} \dif {\bm x}_2\, f({\bm x}_1,{\bm x}_2) p_{k_1,z_1; k_2,z_2}^{(n)}({\bm x}_1,{\bm x}_2),
\end{equation}
for any test function \(f\), and any \(k_1,k_2\in \mathbf{N}\), where we used the notations \({\bm \lambda}_{\bm i}^{z_1}:=(\lambda_{i_1}^{z_1},\dots, \lambda_{i_{k_1}}^{z_1})\) and \({\bm \lambda}_{\bm j}^{z_2}:= (\lambda_{j_1}^{z_2},\dots, \lambda_{j_{k_2}}^{z_2})\).

\begin{theorem}[Asymptotic independence of small singular values of \(X-z_1, X-z_2\)]\label{theo:indun}
    Let \(z_1,z_2\in \C\) be as \(z\) in Theorem~\ref{theo:un}, and assume that \(\abs{z_1-z_2}, \abs{z_1-\overline{z}_2}\sim 1\). Let \(X\) be an i.i.d.\ matrix with real entries satisfying Assumption~\ref{ass:1}, then for any \( k_1, k_2\in\mathbf{N}\), and for any compactly supported test function \(F\in C_c^1(\R^k)\), with \(k=k_1+k_2\), using the notation \({\bm x}=({\bm x}_1,{\bm x}_2)\), with \({\bm x}_l\in \mathbf{R}^{k_l}\), it holds
    \begin{equation}
        \int_{\R^k}F({\bm x}) \left[ \frac{1}{(\rho^{z_1})^{k_1}(\rho^{z_2})^{k_2}}p_{k_1,z_1;k_2, z_2}^{(n)}\left( \frac{{\bm x}_1}{n\rho^{z_1}},\frac{{\bm x}_2}{n\rho^{z_2}}\right) -p_{k_1}^{(\infty,\C)}({\bm x}_1)p_{k_2}^{(\infty,\C)}({\bm x}_2)\right]\dif {\bm x}=\mathcal{O}\left( n^{-c(k)}\right),
    \end{equation}
    where \(\rho^{z_l}=\rho^{z_l}(0)\), and \(c(k)>0\) is a small constant only depending on \(k\). The implicit constant in \(\mathcal{O}(\cdot)\) may depend on \(k\), \(\norm{ F}_{C^1}\), and \(C_p\) from~\eqref{eq:hmb}.
\end{theorem}

\begin{remark}
    We stated Theorem~\ref{theo:un} for two different \(z_1,z_2\) for notational simplicity. The analogous result holds for any finitely many \(z_1,\dots, z_q\) such that \(|z_l-z_m|, |z_l-\overline{z}_m|\sim 1\), with \(l,m\in [q]\).
\end{remark}

\section{Proof strategy}

The proof of Theorem~\ref{theo:CLT} follows a similar strategy as the proof of~\cite[Thm.~\ref{cplx-theo:CLT}]{1912.04100} with several major changes. We use Girko's formula to relate the eigenvalues of \(X\) to the resolvent of the \(2n \times 2n\) matrix
\begin{equation}
    \label{eq:herher}
    H^z := \begin{pmatrix}
        0          & X-z \\
        (X-z)^\ast & 0
    \end{pmatrix},
\end{equation}
the so called \emph{Hermitisation} of~\(X-z\). We denote the eigenvalues of \(H^z\), which come in pairs symmetric with respect to zero, by \(\{\lambda^z_{\pm i}\}_{i\in[n]}\). The \emph{local law}, see Theorem~\ref{theo:Gll} below, asserts that the resolvent \(G(w)=G^z(w):= (H^z-w)^{-1}\) of \(H^z\) with \(\eta=\Im w\ne 0\) becomes approximately deterministic, as \(n\to \infty\). Its limit is expressed via the unique solution of the scalar equation
\begin{equation}
    \label{eq m}
    - \frac{1}{m^z} = w + m^z -\frac{\abs{z}^2}{w + m^z}, \quad \eta\Im m^z(w) >0,\quad \eta=\Im w\ne 0,
\end{equation}
which is a special case of the \emph{matrix Dyson equation} (MDE), see e.g.~\cite{MR3916109} and~\eqref{eq MDE} later. Note that on the imaginary axis \(m^z(\ii\eta)=\ii\Im m^z(\ii\eta)\).  We define the \emph{self-consistent density of states} of \(H^z\) and its extension to the upper half-plane by
\begin{equation}
    \label{eq:scdos}
    \rho^z(E):=\rho^z(E+\ii 0), \qquad \rho^z(w):= \frac{1}{\pi}\Im m^z(w).
\end{equation}
In terms of \(m^z\) the deterministic approximation to \(G^z\) is given by the \(2n \times 2n\) block matrix
\begin{equation}
    \label{eq M}
    M^z(w) := \begin{pmatrix}
        m^z(w)         & -z u^z(w) \\
        - \ov z u^z(w) & m^z(w)
    \end{pmatrix}, \quad u^z(w):=  \frac{m^z(w)}{w+ m^z(w)},
\end{equation}
where each block is understood to be a scalar multiple of the \(n\times n\) identity matrix. We note that \(m,u,M\) are uniformly bounded in \(z,w\), i.e.
\begin{equation}
    \label{eq M bound}
    \norm{M^z(w)}+\abs{m^z(w)}\lesssim 1,\quad \abs{u^z(w)}\le \abs{m^z(w)}^2 + \abs{u^z(w)}^2 \abs{z}^2 < 1,
\end{equation}
see e.g.~\cite[Eqs.~\eqref{cplx-eq M bound}--\eqref{cplx-mubound}]{1912.04100}.

The \emph{local law} for \(G^z(w)\)
in its full averaged and isotropic form  has been obtained for \(w\in\ii\R\) in~\cite{MR3770875} for the bulk regime \(\abs{1-\abs{z}}\ge \epsilon\) and in~\cite{1907.13631} for the edge regime \(\abs{1-\abs{z}}<\epsilon\). In fact, in
the companion paper~\cite{1912.04100} on the complex CLT the local law
for \(w\) on the imaginary axis was sufficient. For the real CLT, however, we need its extension to general spectral parameters \(w\) in the bulk \(\abs{1-\abs{z}}\ge\epsilon\) case that we state below.
We remark that tracial  and entry-wise form of   the local law in Theorem~\ref{theo:Gll}
has already been established in~\cite[Theorem 3.4]{MR3230002}.

\begin{theorem}[Optimal local law for \(G\)]\label{theo:Gll}
    For any \(\epsilon> 0\) and \(z\in\C\) with \(\abs{1-\abs{z}}\ge\epsilon\) the resolvent \(G^z\) at \(w\in \mathbf{H}\) with \(\eta=\Im w\) is very well approximated by the deterministic matrix \(M^z\) in the sense that
    \begin{equation}\label{single local law}
        \begin{split}
            \abs{\braket{(G^z(w)-M^z(w))A}} &\le \frac{ C_\epsilon\norm{A} n^\xi }{n\eta}, \\
            \abs{\braket{\vx,(G^z(w)-M^z(w))\vy}}&\le C_\epsilon\norm{\vx}\norm{\vy}n^\xi  \Bigl(\frac{1}{\sqrt{n\eta}}+\frac{1}{n\eta}\Bigr),
        \end{split}
    \end{equation}
    with very high probability for some \(C_\epsilon \le \epsilon^{-100}\), uniformly for \(\eta \ge n^{-100}\), \(\abs{1-\abs{z}}\ge \epsilon\), and for any deterministic matrices \(A\) and vectors \(\vx,\vy\), and \(\xi>0\).
\end{theorem}
\begin{remark}[Cusp fluctuation averaging]
    For \(w\in\ii\R\) we may choose \(C_\epsilon=1\) by~\cite[Theorem~5.2]{1907.13631} which takes into account the \emph{cusp fluctuation averaging effect}. Since it is not necessary for the present work we refrain from adapting this technique for general \(w\) and rather present a conceptually simpler proof resulting in the \(\epsilon\)-dependent bounds~\eqref{single local law}.
\end{remark}

As in~\cite{1912.04100} we express the linear statistics~\eqref{lin stat def} of eigenvalues \(\sigma_i\) of \(X\) through the resolvent \(G^z\) via Girko's Hermitisation formula~\eqref{girko}
\begin{equation}
    \label{eq:GirkosplitA}
    \begin{split}
        L_n(f)&=\frac{1}{4\pi} \int_\C  \Delta f(z) \Big[\log\abs{\det (H^z-\ii T)}-\E  \log \abs{\det (H^z-\ii T)}\Big]\dif^2 z \\
        &\quad -\frac{n}{2\pi \ii} \int_\C  \Delta f(z)\left[\left(\int_0^{\eta_0}+\int_{\eta_0}^{\eta_c}+\int_{\eta_c}^T \right) \bigl[\braket{ G^z(\ii\eta)-\E G^z(\ii\eta)}\bigr]\dif \eta\right] \dif^2z \\
        &=:  J_T+I_0^{\eta_0}+I_{\eta_0}^{\eta_c}+I_{\eta_c}^T,
    \end{split}
\end{equation}
for \(\eta_0=n^{-1-\delta_0}\), \(\eta_c=n^{-1+\delta_1}\), and \(T=n^{100}\), where \(J_T\) in~\eqref{eq:GirkosplitA} corresponds to the rhs.\ of the first line in~\eqref{eq:GirkosplitA} whilst \(I_0^{\eta_0},I_{\eta_0}^{\eta_c},I_{\eta_c}^T\) correspond to the three different \(\eta\)-integrals in the second line of~\eqref{eq:GirkosplitA}. Here we used that by spectral symmetry of \(H^z\) it follows that \(\braket{G^z(\ii\eta)}\in\ii\R\) and therefore \(\Im\braket{G^z(\ii\eta)}=\braket{G^z(\ii\eta)}/\ii\) in order to obtain~\eqref{eq:GirkosplitA} from~\eqref{girko}. The regime \(J_T\) can be trivially estimated by~\cite[Lemma~\ref{cplx-lem:aprior}]{1912.04100}, while the regime \(I_0^{\eta_0}\) can be controlled using~\cite[Thm.\ 3.2]{MR2684367} as in~\cite[Lemma~\ref{cplx-lem:bbexp}]{1912.04100} (see~\cite[Remark~\ref{cplx-rem:altern}]{1912.04100} for an alternative proof). Both contributions are negligible. For the main term \(I_{\eta_c}^T\) we prove the following \emph{resolvent CLT}\@.
\begin{proposition}[CLT for resolvents]\label{prop clt resolvent}
    Let \(\epsilon>0\), \(\eta_1,\dots,\eta_p>0\), and \(z_1,\dots,z_p\in\C\) be such that for any \(i\ne j\), \(\min\{\eta_i,\eta_j\}\ge n^{\epsilon-1}\abs{z_i-z_j}^{-2}\). Then for any \(\xi>0\) the traces of the resolvents \(G_i=G^{z_i}(\ii \eta_i)\) satisfy an asymptotic Wick theorem
    \begin{equation}\label{eq CLT resovlent}
        \begin{split}
            \E\prod_{i\in[p]} \braket{G_i-\E G_i} &= \sum_{P\in\mathrm{Pairings}([p])}\prod_{\{i,j\}\in P} \E  \braket{G_i-\E G_i}\braket{G_j-\E G_j}  + \mathcal{O}\left(\Psi\right) \\
            &= \frac{1}{n^p}\sum_{P\in\mathrm{Pairings}([p])}\prod_{\{i,j\}\in P} \frac{\wh V_{i,j}+\kappa_4 U_i U_j}{2}+ \landauO{\Psi},
        \end{split}
    \end{equation}
    where
    \begin{equation}\label{eq psi error}
        \Psi:= \frac{n^\xi}{(n\eta_\ast)^{1/2}}\frac{1}{\min_{i\ne j}\abs{z_i-z_j}^4}\prod_{i\in[p]}\Bigl(\frac{1}{\abs{1-\abs{z_i}}}+\frac{1}{(\Im z_i)^2}\Bigr)\frac{1}{n\eta_i}, \qquad \eta_\ast:= \min_i\eta_i,
    \end{equation}
    and \(\wh V_{i,j}=\wh V(z_i,z_j,\eta_i,\eta_j)\) and \(U_i=U(z_i,\eta_i)\) are defined as
    \begin{equation}
        \label{eq:exder}
        \begin{split}
            \wh V(z_i,z_j,\eta_i,\eta_j)&:= V(z_i,z_j,\eta_i,\eta_j) + V(z_i,\ov{z_j},\eta_i,\eta_j)\\
            V(z_i,z_j,\eta_i,\eta_j)&:= \frac{1}{2}\partial_{\eta_i}\partial_{\eta_j} \log \bigl[ 1+(u_i u_j\abs{z_i}\abs{z_j})^2-m_i^2 m_j^2-2u_i u_j\Re z_i\overline{z_j}\bigr], \\
            U(z_i,\eta_i)&:= \frac{\ii}{\sqrt{2}}\partial_{\eta_i} m_i^2,
        \end{split}
    \end{equation}
    with \(m_i=m^{z_i}(\ii\eta_i)\) and \(u_i=u^{z_i}(\ii\eta_i)\) from~\eqref{eq m}--\eqref{eq M}.

    Moreover, the expectation of the normalised trace of \(G=G_i\) is given by
    \begin{equation}\label{EG exp}
        \E\braket{G} = \braket{M} + \cE + \landauO*{\Bigl(\frac{1}{\abs{1-\abs{z}}}+\frac{1}{\abs{\Im z}^2}\Bigr)\Bigl(\frac{1}{n^{3/2}(1+\eta)}+\frac{1}{(n\eta)^2}\Bigr)},
    \end{equation}
    where
    \begin{equation}\label{cE def}
        \cE := - \frac{\ii\kappa_4}{4n}\partial_\eta(m^4) + \frac{\ii}{4n}\partial_\eta \log\Bigl(1-u^2+2u^3\abs{z}^2-u^2(z^2+\ov z^2)\Bigr).
    \end{equation}
\end{proposition}
Proposition~\ref{prop clt resolvent} is the real analogue of~\cite[Prop.~\ref{cplx-prop:CLTresm}]{1912.04100}. The main differences are that (i) the \(V\)-term for the variance appears in a symmetrised form with \(z_j\) and \(\ov{z_j}\), (ii) the error term~\eqref{eq psi error} deteriorates as \(\Im z_i\approx 0\), and (iii) the expectation~\eqref{EG exp} has an additional subleading term which is even present in case \(\kappa_4=0\) (second term in~\eqref{cE def}).

Finally, in order to show that \(I_{\eta_0}^{\eta_c}\) in~\eqref{eq:GirkosplitA} is negligible, we prove that \(\braket{ G^{z_1}(\ii\eta_1)}\) and \(\braket{  G^{z_2}(\ii\eta_2)}\) are asymptotically independent if \(z_1\), \(z_2\) and \(z_1\), \(\overline{z}_2\) are far enough from each other, they are far away from the real axis, they are well inside \(\DD\), and \(\eta_0 \le \eta_1, \eta_2 \le \eta_c\). These regimes of the parameters \(z_1,z_2\) represent the overwhelming part of the \(\dif^2 z_1\dif^2 z_2\) integration in the calculation of \(\E \abs{I_{\eta_0}^{\eta_c}}^2\). The following proposition is the direct analogue of~\cite[Prop.~\ref{cplx-prop:indmr}]{1912.04100}.
\begin{proposition}[Independence of resolvents with small imaginary part]\label{prop:indmr}
    Fix \(p\in \N\). For any sufficiently small \(\omega_h,\omega_d>0\)
    there exist \(\omega_*,\delta_0,\delta_1\) with \(\omega_h\ll \delta_m\ll \omega_*\ll 1\), for \(m=0,1\), such that for any choice of \(z_1,\dots, z_p\) with
    \[
        \abs{z_l}\le 1-n^{-\omega_h}, \,\abs{z_l-z_m}\ge n^{-\omega_d}, \, \abs{z_l-\overline{z}_m}\ge n^{-\omega_d}, \, \abs{z_l-\overline{z}_l}\ge n^{-\omega_d},
    \]
    with \(l,m \in [p]\), \(l\ne m\), it follows that
    \begin{equation}
        \label{eq:indtrlm}
        \E \prod_{l=1}^p \braket{G^{z_l}(\ii\eta_l)}=\prod_{l=1}^p\E  \braket{G^{z_l}(\ii\eta_l)}+\landauO*{\frac{n^{p(\omega_h+\delta_0)+\delta_1}}{n^{\omega_*}}},
    \end{equation}
    for any \(\eta_1,\dots,\eta_p\in [n^{-1-\delta_0},n^{-1+\delta_1}]\).
\end{proposition}

As in the complex case~\cite{1912.04100},
one key ingredient for both Propositions~\ref{prop clt resolvent} and~\ref{prop:indmr} is a local law for products of resolvents
\(G_1,G_2\) for \(G_i=G^{z_i}(w_i)\). We remark that local laws for products of resolvents have also been derived for (generalized)
Wigner matrices~\cite{MR3805203,2001.08725} and for sample covariance matrices~\cite{MR4119592}, as well as for the addition of random matrices~\cite{10.1093/imrn/rnaa210}.

Note that the deterministic approximation to \(G_1G_2\) is
not given simply by \(M_1M_2\) where \(M_i:=M^{z_i}(w_i)\) from~\eqref{eq M}.
To describe the correct approximation, as in~\cite[Section~\ref{cplx-sec local law G2}]{1912.04100}, we define the \emph{stability operator}
\begin{equation}
    \label{eq:stabop12}
    \wh\cB=\wh\cB_{12}=\wh\cB(z_1,z_2, w_1,w_2):= 1-M_1 \SS[\cdot]M_2,
\end{equation}
acting on the space of \(2n \times 2n\) matrices. Here the linear \emph{covariance} or \emph{self-energy operator} \(\SS\colon\C^{2n\times 2n}\to\C^{2n\times 2n}\) is defined as
\begin{equation}\label{S def}
    \SS\biggl[\begin{pmatrix}
            A & B \\ C & D
        \end{pmatrix}\biggr] :=  \wt\E \wt W \begin{pmatrix}
        A & B \\ C & D
    \end{pmatrix} \wt W= \begin{pmatrix}
        \braket{D} & 0 \\ 0 & \braket{A}
    \end{pmatrix} ,\quad \wt W=\begin{pmatrix}
        0 & \wt X \\ \wt X^\ast & 0
    \end{pmatrix}, \quad \wt X\sim\mathrm{Gin}_\C,
\end{equation}
i.e.\ it averages the diagonal blocks and swaps them. Here $\widetilde{\mathbf{E}}$ denotes the expectation with respect to $\widetilde{X}$, \(\braket{ A}=n^{-1}\text{Tr}A\) and \(\mathrm{Gin}_\C\) stands for the standard complex Ginibre ensemble. The ultimate equality in~\eqref{S def} follows directly from \(\E \wt x_{ab}^2=0\), \(\E\abs{\wt x_{ab}}^2=n^{-1}\). Note that as a matter of choice we define the stability operator~\eqref{eq:stabop12} with the covariance operator \(\SS\) corresponding to the complex rather than the real Ginibre ensemble. However, to leading order there is no difference between the two and the present choice is more consistent with the companion paper~\cite{1912.04100}. The effect of this discrepancy will be estimated in a new error term (see~\eqref{eq real cumulant expansion} later).

For any deterministic matrix \(B\) we define
\begin{equation}
    \label{eq M12 def}
    M_B^{z_1,z_2}(w_1,w_2):= \wh \cB_{12}^{-1}[M^{z_1}( w_1)B M^{z_2}(w_2)],
\end{equation}
which turns out to be the deterministic approximation to \(G_1 B G_2\). Indeed, from the local law for \(G_1,G_2\), Theorem~\ref{theo:Gll}, and~\cite[Thm.~\ref{cplx-thm local law G2}]{1912.04100} we immediately conclude the following theorem.
\begin{theorem}[Local law for \(G^{z_1}BG^{z_2}\)]\label{thm local law G2}
    Fix \(z_1,z_2\in\C\) with \(\abs{1-\abs{z_i}}\ge\epsilon\), for some \(\epsilon> 0\) and \(w_1,w_2\in\C\) with \(\abs{\eta_i}:=\abs{\Im w_i}\ge n^{-1}\) such that
    \[ \eta_\ast:= \min\{\abs{\eta_1},\abs{\eta_2}\}\ge n^{-1+\epsilon_*}\abs{\wh\beta_\ast}^{-1},\]
    for some small \(\epsilon_*>0\), where \(\wh\beta_\ast\) is the, in absolute value, smallest eigenvalue of \(\wh\cB_{12}\) defined in~\eqref{eq:stabop12}. Then, for any bounded deterministic matrix \(B\), \(\norm{B}\lesssim 1\), the product of resolvents \( G^{z_1}B G^{z_2}=G^{z_1}(w_1)BG^{z_2}(w_2)\) is well approximated by \(M_B^{z_1,z_2} =M_B^{z_1,z_2}(w_1,w_2)\) defined in~\eqref{eq M12 def}
    in the sense that
    \begin{equation}\label{final local law}
        \begin{split}
            \abs{\braket{A (G^{z_1} BG^{z_2}-M_B^{z_1,z_2})}}&\le \frac{ C_\epsilon\norm{A}n^\xi}{n\eta_*|\eta_1\eta_2|^{1/2}\abs{\wh\beta_\ast}}\Bigl(\eta_\ast^{1/12}+\frac{\eta_\ast^{1/4}}{\abs{\wh\beta_\ast}} +\frac{1}{\sqrt{n\eta_\ast}}+\frac{1}{(\abs{\wh\beta_\ast} n\eta_\ast)^{1/4}}\Bigr), \\
            \abs{\braket{\vx,(G^{z_1}BG^{z_2}-M_B^{z_1,z_2})\vy}}&\le \frac{ C_\epsilon\norm{\vx} \norm{\vy}n^\xi}{(n\eta_*)^{1/2}|\eta_1\eta_2|^{1/2}\abs{\wh\beta_\ast}}
        \end{split}
    \end{equation}
    for some \(C_\epsilon\) with very high probability for any deterministic \(A,\vx,\vy\) and \(\xi>0\). If \(w_1,w_2\in\ii\R\) we may choose \(C_\epsilon=1\), otherwise we can choose \(C_\epsilon\le\epsilon^{-100}\).
\end{theorem}

An effective lower bound on \(\Re\widehat{\beta}_*\), hence on \(\abs{\widehat{\beta}_*}\), will be given in Lemma~\ref{lem:lowbeta} later.

The paper is organised as follows: In Section~\ref{section CLT for linstats} we prove Theorem~\ref{theo:CLT} by combining Propositions~\ref{prop clt resolvent} and~\ref{prop:indmr}. In Section~\ref{section G local law} we prove the local law for \(G\) away from the imaginary axis, Theorem~\ref{theo:Gll}. In Section~\ref{sec:CLTres} we prove Proposition~\ref{prop clt resolvent}, the Central Limit Theorem for resolvents using Theorem~\ref{thm local law G2}. In Section~\ref{sec:IND} we prove Proposition~\ref{prop:indmr} again using Theorem~\ref{thm local law G2}, and conclude Theorem~\ref{theo:un}.

Note that Theorem~\ref{thm local law G2}, the local law for \(G^{z_1}BG^{z_2}\), is used in two different contexts. Traces of \(AG^{z_1}BG^{z_2}\), for some deterministic matrices \(A,B\in\C^{2n \times 2n}\), naturally arise along the cumulant expansion for \(\prod_i\braket{ G_i-\E G_i}\) in Proposition~\ref{prop clt resolvent}. The proof of Proposition~\ref{prop:indmr} is an analysis of weakly correlated DBMs, where the correlations are given by eigenvector overlaps~\eqref{eq:ovint}, whose estimate is reduced to an upper bound on \(\braket{\Im G^{z_1}\Im G^{z_2}}\).

\section{Central limit theorem for linear statistics: Proof of Theorem~\ref{theo:CLT}}\label{section CLT for linstats}

From Propositions~\ref{prop clt resolvent} and~\ref{prop:indmr} we conclude Theorem~\ref{theo:CLT} analogously to~\cite[Section~\ref{cplx-sec:PCLT}]{1912.04100}, we only describe the few minor modifications.

\begin{proof}[Proof of Theorem~\ref{theo:CLT}]
    We explain the four modifications compared with the proof of~\cite[Theorem~\ref{cplx-theo:CLT}]{1912.04100}. First, there are two additional terms in in the variance~\eqref{eq:exder} and expectation~\eqref{cE def} of the resolvent CLT, compared to~\cite[Eqs.~\eqref{cplx-eq:exder}--\eqref{cplx-prop clt exp}]{1912.04100}. These additional terms result in additional explicit terms in~\eqref{eq:expv} and~\eqref{eq:cov}. For the expectation in~\eqref{eq:expv} we have
    \begin{align}\label{correction exp eq}
         & -\frac{1}{2\pi\ii}\int_\C\Delta f(z) \frac{\ii}{4n}\int_0^\infty\partial_\eta \log\Bigl(1-u^2+2u^3\abs{z}^2-u^2(z^2+\ov z^2)\Bigr) \dif \eta \dif^2z                                                              \\\nonumber
         & \;=\frac{1}{4\pi}\int_{\DD } \frac{f(\Re z)-f(z)}{(\Im z)^2}\dif^2z -\frac{1}{2\pi}\int_0^{2\pi}  f(e^{\ii\theta}) \dif \theta + \frac{1}{2\pi}\int_{-1}^{1} \frac{f(x)}{\sqrt{1-x^2}}\dif x+\frac{f(1)+f(-1)}{4}
    \end{align}
    and for the variance in~\eqref{eq:cov} we have
    \begin{equation}\label{correction var eq}
        \begin{split}
            &-\frac{1}{8\pi^2}\int_\C\dif^2 z_1\int_\C\dif^2 z_2\Delta f(z_1)\ov{\Delta g(z_2)} \int_0^\infty\dif\eta_1 \int_0^\infty\dif\eta_2 V(z_1,\ov{z_2},\eta_1,\eta_2)\\
            &\qquad= \frac{1}{4\pi}\braket{\nabla g(\ov{\cdot}),\nabla f}_{L^2(\DD)} + \frac{1}{2}\braket{g(\ov{\cdot}),f}_{\dot H^{1/2}(\partial\DD)},
        \end{split}
    \end{equation}
    so that together with contribution from \(V(z_1,z_2,\eta_1,\eta_2)\) in~\eqref{eq:exder} we have
    \[\begin{split}
            &\frac{1}{4\pi}\braket{\nabla g +\nabla g(\ov{\cdot}),\nabla f}_{L^2(\DD)} + \frac{1}{2}\braket{g +g(\ov{\cdot}),f}_{\dot H^{1/2}(\partial\DD)} \\
            &\quad = \frac{1}{2\pi} \braket{\nabla P_\sym g,\nabla P_\sym f}_{L^2(\DD)} + \braket{P_\sym g,P_\sym f}_{\dot H^{1/2}(\partial\DD)}.
        \end{split}.\]
    The identities~\eqref{correction exp eq}--\eqref{correction var eq} will be proven separately below. The other two modifications concern the error terms in~\eqref{eq psi error} and~\eqref{EG exp}. Namely, there is an
    additional factor including \((\Im z_l)^{-2}\) (cf.~\cite[Eqs.~\eqref{cplx-eq psi error},~\eqref{cplx-prop clt exp}]{1912.04100}), and, finally,~\eqref{eq:indtrlm} holds under the additional assumption that \(\abs{z_l-\overline{z}_m} \ge n^{-\omega_d}\), and \(\abs{z_l-\overline{z}_l}\ge n^{-\omega_d}\) (cf.~\cite[Prop.~\ref{cplx-prop:indmr}]{1912.04100}). Both these issues can be handled in the same way as the constraints on \(\abs{z_l-z_m}\) have been treated in~\cite[Section~\ref{cplx-sec:PCLT}]{1912.04100} (see e.g.~\cite[Eq.~\eqref{cplx-eq:bound3}]{1912.04100}). This means that we additionally exclude the regimes of negligible volume \(\abs{z_l-\overline{z}_m}< n^{-\omega_d}\) or \(\abs{z_l-\overline{z}_l}< n^{-\omega_d}\) from the \(\dif z_1\dots \dif z_p\)-integral in~\cite[Eqs.~\eqref{cplx-eq:bound2},~\eqref{cplx-eq:wickprod2}]{1912.04100} using the almost optimal \emph{a priori} bound from~\cite[Lemma~\ref{cplx-lem:aprior}]{1912.04100}. Finally, the last modification concerns the application of~\cite[Lemma~\ref{cplx-lem:firststepmason}]{1912.04100} to analyze the regime $\eta\in [\eta_0,\eta_c]$ in the computation of $\E\sum_i f(\sigma_i)$. In the current case, we need to further assume that $|\Im z|$ is not too small to rely on Lemma~\ref{lem:firststepmason} exactly as in the proof of \cite[Lemma~\ref{cplx-lem:compe}]{1912.04100}. The contribution of the regime $|\Im z|\ll 1$ is again negligible by its small volume.
\end{proof}

\begin{proof}[Proof of~\eqref{correction exp eq}]
    With the short-hand notation \(z=x+\ii y\), we compute
    \begin{equation}
        \label{eq:etaint}
        \begin{split}
            &\int_0^\infty \frac{\ii}{4n}\partial_\eta \log\Bigl(1-u^2+2u^3\abs{z}^2-u^2(z^2+\ov z^2)\Bigr) \dif\eta \\
            &\quad = -\frac{\ii}{4n}\begin{cases}
                \log 4+2\log \abs{y},                                       & \abs{z}\le 1, \\
                \log\abs*{(x^2+y^2)^2+1-2(x^2-y^2)}-\log\abs*{(x^2+y^2)^2}, & \abs{z}>1,
            \end{cases}
        \end{split}
    \end{equation}
    using that $u=1+\mathcal{O}(\eta)$ for $|z|\le 1$ and $u=|z|^{-2}+\mathcal{O}(\eta)$ for $|z|> 1$, so that for~\eqref{correction exp eq} we need to compute
    \begin{equation}
        \label{eq:intbp}
        \frac{1}{4}\int_\C \Delta f(z) \bigl[(\log 4+2\log\abs{y})\bm1(\abs{z}\le 1)+(\log\abs{z-1}^2+\log\abs{z+1}^2-2\log\abs{z}^2)\bm1(\abs{z}\ge 1) \bigr] \dif^2z.
    \end{equation}
    We may assume that \(f\) is symmetric with respect to the real axis, i.e.\ \(f=P_\sym f\) with \(P_\sym\) as in~\eqref{Psym def} since \(L_n(f-P_\sym f)=0\) by symmetry of the spectrum and therefore \(L_n(f)=L_n(P_\sym f)\). Since the functions in~\eqref{eq:intbp} are singular we introduce an \(\epsilon\)-regularisation which enables us to perform integration by parts. In particular, the integral in~\eqref{eq:intbp} is equal to the \(\epsilon\to 0\) limit of
    \begin{equation}
        \label{eq:igoy}
        \begin{split}
            &\int_\C \partial_z\partial_{\ov{z}} f(z) \bigl[(\log 4+2\log\abs{y})\bm1(\abs{z}\le 1, \abs{y}\ge \epsilon) \\
                &\qquad+(\log\abs{z-1}^2+\log\abs{z+1}^2-2\log\abs{z}^2)\bm1(\abs{z}\ge 1, \abs{z\pm1}\ge\epsilon) \bigr] \dif^2z,
        \end{split}
    \end{equation}
    where \(\abs{z\pm1}\ge \epsilon\) denotes that \(\abs{z-1}\ge \epsilon\) and \(\abs{z+1}\ge \epsilon\), and we used that the contribution from the regimes \(\abs{y}\le \epsilon\) and \(\abs{z\pm1}\le \epsilon\) are negligible as \(\epsilon\to 0\). In the following equalities should be understood in the \(\epsilon\to 0\) limit.

    Since
    \[ \log\abs{z-1}^2+\log \abs{z+1}^2-2\log\abs{z}^2=\log 4+2\log \abs{y}\]
    for \(\abs{z}=1\), when integrating by parts in~\eqref{eq:igoy}, the terms where either \(\bm1(\abs{z}\le 1)\) or \(\bm1(\abs{z}>1)\) are differentiated are equal to zero, using that
    \begin{equation}
        \label{eq:intbpchf}
        \partial_z \bm1(\abs{z}\ge 1) \dif^2 z=\frac{\ii}{2} \bm1(\abs{z}=1) \dif \overline{z}.
    \end{equation}
    We remark that~\eqref{eq:intbpchf} is understood in the sense of distributions, i.e.\ the equality holds when tested against compactly supported test functions \(f\):
    \[
        -\int_\C\partial_z f(z)\bm1(\abs{z}\ge 1) \dif^2 z= \frac{\ii}{2} \int_{\abs{z}=1} f(z)\dif \overline{z}.
    \]
    Moreover, with a slightly abuse of notation in~\eqref{eq:intbpchf} by \(\bm1(\abs{z}=1) \dif \overline{z}\) we denote the clock-wise contour integral over the unit circle. This notation is used in the remainder of this section.

    Then, performing integration by parts with respect to \(\partial_{\overline{z}}\), we conclude that~\eqref{eq:igoy} is equal to
    \begin{equation}
        \label{eq:firststep}
        -\int_\C \partial_z f(z) \left[\frac{\ii}{y} \bm1(\abs{z}\le 1, \abs{y}\ge \epsilon)+\left(\frac{1}{\overline{z}-1}+\frac{1}{\overline{z}+1}-\frac{2}{\overline{z}} \right)\bm1(\abs{z}\ge 1, \abs{z\pm1}\ge \epsilon)\right]\dif^2 z.
    \end{equation}
    In order to get~\eqref{eq:firststep} we used that
    \[ \abs*{\partial_z f(x+\ii \epsilon)-\partial_z f(x-\ii \epsilon)}\cdot \abs*{\log \epsilon} \lesssim \epsilon^{\delta'},\]
    for some small fixed \(\delta'>0\), by \(f\in H^{2+\delta}\), and similarly all the other \(\epsilon\)-boundary terms tend to zero. This implies that when the \(\partial_{\overline{z}}\) derivative hits the \(\epsilon\)-boundary terms then these give a negligible contribution as \(\epsilon\to 0\). We now consider the two terms in~\eqref{eq:firststep} separately.

    Since the integral of \(y^{-1}\) over \(\DD \) is zero we can rewrite the first term in~\eqref{eq:firststep} as
    \[ -\int_\C \partial_z (f(z)-f(x)) \frac{\ii}{y} \bm1(\abs{z}\le 1, \abs{y}\ge \epsilon)\dif^2 z.\]
    Then performing integration by parts we conclude that the first term in~\eqref{eq:firststep} is equal to
    \begin{equation}
        \label{int eq:1}
        \begin{split}
            &-\frac{1}{2}\int_{\DD } \frac{f(x+\ii y)-f(x)}{y^2} \dif x\dif y-\frac{\ii}{2}\int_0^{2\pi} \frac{f(e^{\ii\theta})-f(\cos \theta)}{\sin \theta} e^{-\ii\theta} \dif \theta\\
        \end{split}
    \end{equation}
    where we used that
    \[ \abs*{\frac{f(x,\epsilon)-2f(x,0)+f(x,-\epsilon)}{\epsilon}}\lesssim \epsilon^{\delta'}, \]
    to show that the terms when the \(\partial_z\) derivative hits the \(\epsilon\)-boundary terms go to zero as \(\epsilon\to 0\). Note that the integrals in~\eqref{int eq:1} are absolutely convergent since \(f\) is symmetric with respect to the real axis. For the second term in~\eqref{int eq:1} we further compute
    \begin{equation}
        \begin{split}
            \int_0^{2\pi} \frac{f(e^{\ii\theta})-f(\cos \theta)}{\sin \theta} e^{-\ii\theta} \dif \theta
            &= \int_0^{2\pi} \frac{f(e^{\ii\theta})-f(\cos \theta)}{\sin \theta} (\cos\theta-\ii\sin\theta) \dif \theta\\
            & =- \ii \int_0^{2\pi} \Bigl(f(e^{\ii\theta})-f(\cos \theta)\Bigr) \dif \theta
        \end{split}
    \end{equation}
    where we used that the term with $\cos \theta/\sin \theta$ is zero by symmetry.

    With defining the domain
    \[ \Omega_\epsilon:= \{\abs{z}\ge 1\}\cap  \{\abs{z\pm1}\ge \epsilon\},\]
    the second term in~\eqref{eq:firststep} is equal to
    \begin{equation}
        \label{eq:omeps}
        -\int_{\Omega_\epsilon} \partial_z f(z) \left(\frac{1}{\overline{z}-1}+\frac{1}{\overline{z}+1}-\frac{2}{\overline{z}} \right)\dif^2 z.
    \end{equation}
    Since
    \[ \frac{1}{\overline{z}-1}+\frac{1}{\overline{z}+1}-\frac{2}{\overline{z}}\]
    is anti-holomorphic on \(\Omega_\epsilon\), performing integration by parts with respect to \(\partial_z\) in~\eqref{eq:omeps}, we obtain
    \begin{equation}
        \label{eq:2}
        -\int_{\Omega_\epsilon} \partial_z f(z) \left(\frac{1}{\overline{z}-1}+\frac{1}{\overline{z}+1}-\frac{2}{\overline{z}} \right)\dif^2 z=\frac{\ii}{2}\int_{\partial\Omega_\epsilon}  f(z) \left(\frac{1}{\overline{z}-1}+\frac{1}{\overline{z}+1}-\frac{2}{\overline{z}} \right)\dif \overline{z}.
    \end{equation}
    Taking the limit \(\epsilon\to 0\) in the r.h.s.\ of~\eqref{eq:2} we conclude
    \begin{equation}
        \label{Omega eq:6}
        \begin{split}
            \lim_{\epsilon\to 0}\frac{\ii}{2}\int_{\partial\Omega_\epsilon}  f(z) \left(\frac{1}{\overline{z}-1}+\frac{1}{\overline{z}+1}-\frac{2}{\overline{z}} \right)\dif \overline{z}&=\frac{\pi}{2} \bigl[f(1)+f(-1)\bigr]- \int_0^{2\pi} f(e^{\ii\theta}) \dif \theta \\
            &\quad +\lim_{\epsilon\to 0} \left(\int_\epsilon^{\pi-\epsilon}+\int_{\pi+\epsilon}^{2\pi-\epsilon} \right) f(e^{\ii\theta})\frac{e^{-2\ii\theta}}{e^{-2\ii\theta}-1} \dif \theta.
        \end{split}
    \end{equation}
    The last term in~\eqref{Omega eq:6} simplifies to
    \begin{equation}
        \label{eq:simpl}
        \begin{split}
            \lim_{\epsilon\to 0} \left(\int_\epsilon^{\pi-\epsilon}+\int_{\pi+\epsilon}^{2\pi-\epsilon} \right) f(e^{\ii\theta})\frac{e^{-2\ii\theta}}{e^{-2\ii\theta}-1} \dif \theta
            &=  \lim_{\epsilon\to 0} \left(\int_\epsilon^{\pi-\epsilon}+\int_{\pi+\epsilon}^{2\pi-\epsilon} \right) f(e^{\ii\theta})\Big[ \frac{\ii}{2} \frac{\cos\theta}{\sin\theta}
                +\frac{1}{2}\Big] \dif \theta \\
            &= \frac{1}{2} \int_0^{2\pi} f(e^{\ii\theta}) \dif \theta,
        \end{split}
    \end{equation}
    by symmetry. By combining~\eqref{int eq:1}--\eqref{eq:simpl} we conclude~\eqref{correction exp eq}.
\end{proof}

\begin{proof}[Proof of~\eqref{correction var eq}]
    By change of variables \(z_2\to\ov{z_2}\) we can then write
    \begin{equation}\label{eq symmetrised variance}
        \begin{split}
            &\int_\C \dif^2 z_1 \int_\C \dif^2 z_2 \int_0^\infty \dif\eta_1 \int_0^\infty\dif \eta_2  \Delta f(z_1) \Delta \ov{g(z_2)}V(z_1,\ov{z_2},\eta_1,\eta_2) \\
            &\quad = \int_\C \dif^2 z_1 \int_\C \dif^2 z_2 \int_0^\infty \dif\eta_1 \int_0^\infty\dif \eta_2 \Delta f(z_1) \Delta \ov{g(\ov{z_2})} V(z_1,z_2,\eta_1,\eta_2)
        \end{split}
    \end{equation}
    such that~\cite[Lemma~\ref{cplx-lem:vi}]{1912.04100} is applicable and~\eqref{correction var eq} follows.
\end{proof}

\section{Local law away from the imaginary axis: Proof of Theorem~\ref{theo:Gll}}\label{section G local law}
The goal of this section is to prove a local law for \(G=G^z(w)\) for \(z\) in the bulk, as stated in Theorem~\ref{theo:Gll}. We do not follow the precise \(\epsilon\)-dependence in the proof explicitly but it can be checked from the arguments below that \(C_\epsilon=\epsilon^{-100}\) clearly suffices. We denote the unique solution to the deterministic matrix equation (see e.g.~\cite{MR3916109})
\begin{equation}\label{eq MDE}
    -1 = \SS[M] M + Z M + wM,\quad Z:= \begin{pmatrix}
        0 & z \\ \ov{z}&0
    \end{pmatrix}, \quad \Im M>0, \quad \Im w>0
\end{equation}
by \(M=M^z(w)\), where we recall the definition of \(\SS\) from~\eqref{S def}. The solution to~\eqref{eq MDE} is given by~\eqref{eq M}. To keep notations compact, we first introduce a commonly used (see, e.g.~\cite{MR3068390}) notion of high-probability bound.
\begin{definition}[Stochastic Domination]\label{def:stochDom}
    If \[X=\tuple*{ X^{(n)}(u) \given n\in\N, u\in U^{(n)} }\quad\text{and}\quad Y=\tuple*{ Y^{(n)}(u) \given n\in\N, u\in U^{(n)} }\] are families of non-negative random variables indexed by \(n\), and possibly some parameter \(u\) in a set \(U^{(n)}\), then we say that \(X\) is stochastically dominated by \(Y\), if for all \(\epsilon, D>0\) we have \[\sup_{u\in U^{(n)}} \Prob\left[X^{(n)}(u)>n^\epsilon  Y^{(n)}(u)\right]\leq n^{-D}\] for large enough \(n\geq n_0(\epsilon,D)\). In this case we use the notation \(X\prec Y\). Moreover, if we have \(\abs{X}\prec Y\) for families of random variables \(X, Y\), we also write \(X=\mathcal{O}_\prec(Y)\).
\end{definition}
Let us assume that some a-priori bounds
\begin{equation}\label{eq init bound}
    \abs{\braket{\vx,(G-M)\vy}}\prec \Lambda, \qquad \abs{\braket{A(G-M)}}\prec \xi
\end{equation}
for some deterministic control functions \(\Lambda\) and \(\xi\) depending on \(w,z\) have already been established, uniformly in \(\vx, \vy, A\) under the constraint \(\norm{\vx}, \norm{\vy}, \norm{A}\le1\). From the resolvent equation \(1=(W-Z-w)G\) we obtain
\begin{equation}\label{eq res eq}
    -1=-WG+ZG+wG = \SS[G]G + ZG + wG -\un{WG},
\end{equation}
where we introduced the \emph{self-renormalisation}, denoted by underlining, of a random
variable of the form \(W f(W)\) for some regular function \(f\) as
\begin{equation}\label{self renorm}
    \un{W f(W)} := W f(W)- \wt\E \wt W (\partial_{\wt W} f)(W), \qquad \wt W = \begin{pmatrix}
        0 & \wt X \\ \wt X^\ast & 0
    \end{pmatrix},\quad \wt X \sim \mathrm{Gin}_\C,
\end{equation}
with \(\wt X\) independent of \(X\). The choice of defining the self-renormalisation in terms of the complex rather than real Ginibre ensemble has the consequence that an additional error term needs to be estimated. For real Ginibre we have
\[ \E WG = -\E\SS[G]G - \E\cT[G]G, \quad \cT\Biggl[\begin{pmatrix}
            a & b \\c&d
        \end{pmatrix}\Biggr] = \frac{1}{n} \begin{pmatrix}
        0 & c^t \\b^t&0
    \end{pmatrix}, \]
but the renormalisation comprises only the \(\SS[G]\) term, i.e.\
\[\underline{WG} = WG + \E \SS[G]G,\]
thus the \(\cT\)-term needs to be estimated. By the Ward identity \(GG^\ast=G^\ast G=\eta^{-1}\Im G\) it follows that
\begin{equation}\label{eq T bound}
    \begin{split}
        \abs{\braket{\vx,\cT[G]G\vy}}\le \frac{1}{n}\sqrt{\braket{\vx,G G^\ast\vx}}\sqrt{\braket{\vy,G^\ast G\vy}}=\frac{1}{n\eta}\sqrt{\braket{\vx,\Im G\vx}}\sqrt{\braket{\vy,\Im G\vy}} \prec \frac{\Lambda+\rho}{n\eta},
    \end{split}
\end{equation}
where \(\rho:=\pi^{-1}\Im m\) from~\eqref{eq:scdos}. By~\cite[Theorem~4.1]{MR3941370} it follows that
\[ \abs{\braket{\vx,(WG+\SS[G]G+\cT[G]G)\vy}}\prec \sqrt{\frac{\rho+\Lambda}{n\eta}}, \quad \abs{\braket{A(WG+\SS[G]G+\cT[G]G) }}\prec\frac{\rho+\Lambda}{n\eta} \]
and therefore, together with the bound~\eqref{eq T bound} on the \(\cT\)-term we obtain
\begin{equation}\label{fluctuation av bound}
    \abs{\braket{\vx,\un{WG}\vy}}\prec \sqrt{\frac{\rho+\Lambda}{n\eta}},\quad \abs{\braket{A\un{WG}}}\prec \frac{\rho+\Lambda}{n\eta}.
\end{equation}

We now consider the \emph{stability operator} \(\cB:=1-M\SS[\cdot]M\) which expresses the stability of~\eqref{eq MDE} against small perturbations. Since \(\SS\) only depends on the four block traces of the input matrix, and \(M\) is a multiple of the identity matrix in each block, the operator \(\cB\) can be understood as an operator acting on \(2\times 2\) matrices after taking a \emph{partial trace}. Henceforth for all practical purposes we may identify \(\cB\) with this four dimensional operator. Written as a \(4\times 4\) matrix, it is given by
\begin{equation}\label{cal B}
    \cB = \begin{pmatrix}
        B_1 & 0 \\ B_2 & 1
    \end{pmatrix}, \quad B_1 = \begin{pmatrix}
        1-u^2\abs{z}^2 & -m^2 \\ -m^2 & 1-u^2\abs{z}^2
    \end{pmatrix}, \quad B_2 = \begin{pmatrix}
        m u z & mu z \\ mu\ov z & mu\ov z
    \end{pmatrix},
\end{equation}
with \(m,u\) defined in~\eqref{eq m}--\eqref{eq M}. Here the rows and columns of \(\cB\) are ordered in such a way that \(2\times 2\) matrices are mapped to vectors as in
\[ \begin{pmatrix}
        a & b \\c&d
    \end{pmatrix}\Rightarrow  \begin{pmatrix}
        a \\d\\b\\c
    \end{pmatrix}.\]
We first record some spectral properties of \(\cB\) in the following lemma, the proof of which we defer to the end of the section. Note that \(\mathcal{B}^*\) refers to the adjoint of \(\mathcal{B}\) with respect to the scalar product \(\braket{ A,B}=(2n)^{-1}\text{Tr} A^*B\), for any deterministic matrices \(A,B\in\C^{2n\times 2n}\).
\begin{lemma}\label{lemma B analysis}
    Let \(w\in\HC\), \(z\in\C\) be bounded spectral parameters, \(\abs{w}+\abs{z}\lesssim1\). Then the operator \(\cB\) has the trivial eigenvalues \(1\) with multiplicity \(2\), and furthermore has two non-trivial eigenvalues, and left and right eigenvectors
    \[ \begin{split}
            \cB[E_-] &=(1+m^2-u^2\abs{z}^2)E_-\quad \cB^\ast[E_-]=\ov{(1+m^2-u^2\abs{z}^2)} E_-, \\
            \cB[V_r] &= (1-m^2-u^2\abs{z}^2)V_r, \quad \cB^\ast[V_l]=\ov{(1-m^2-u^2\abs{z}^2)}V_l,
        \end{split}  \]
    where \(E_-:= (E_1-E_2)/\sqrt{2}\) and
    \begin{equation}
        \label{e1e2}
        E_1:=\begin{pmatrix}
            1 & 0 \\0&0
        \end{pmatrix}, \quad E_2:= \begin{pmatrix}
            0 & 0 \\0&1
        \end{pmatrix},\quad V_r := \begin{pmatrix}
            m^2+u^2 \abs{z}^2 & -2muz \\ -2mu\ov{z} &  m^2+u^2 \abs{z}^2
        \end{pmatrix}, \quad V_l:= \frac{1}{\ov{\braket{V_r}}}.
    \end{equation}
    Moreover, for the second non-trivial eigenvalue we have the lower bound
    \begin{equation}\label{eq beta lower bound lemma}
        \abs{1-m^2-u^2\abs{z}^2}\gtrsim \begin{cases}
            \Im m,     & \abs{1-\abs{z}}\ge\epsilon, \\
            (\Im m)^2, & \abs{1-\abs{z}}<\epsilon.
        \end{cases}
    \end{equation}
\end{lemma}
Corresponding to the two non-trivial eigenvalues of \(\cB\) we define the \emph{spectral projections}
\[\cP_\ast:=\braket{E_-,\cdot}E_-,\quad \cP:=\braket{V_l,\cdot}V_r,\quad \cQ_\ast:=1-\cP_\ast,\quad \cQ:=1-\cP_\ast-\cP.\]
From~\eqref{eq MDE} and~\eqref{eq res eq} it follows that
\begin{equation}\label{B[G-M] eq}
    \cB[G-M] = M\SS[G-M](G-M) - M\un{WG}.
\end{equation}
We now distinguish the two cases \(\rho\sim1\) and \(\rho\ll1\). In the former we obtain
\begin{equation}\label{BQast}
    \norm{\cQ_\ast\cB^{-1}}_{\norm{\cdot}\to\norm{\cdot}}\lesssim \frac{1}{\abs{1-m^2-u^2\abs{z}^2}} \lesssim 1
\end{equation}
by~\eqref{eq beta lower bound lemma}. Since \(\braket{E_-,G}=\braket{E_-,M}=0\) by block symmetry, it follows that
\[G-M=\cQ_\ast[G-M]=\cQ_\ast\cB^{-1}\cB[G-M]\]
and thus
\begin{subequations}\label{eq bootstrap ast}
    \begin{equation}
        \begin{split}
            \braket{\vx,(G-M)\vy} &=  \Tr \left[(\cQ_\ast \cB^{-1})^\ast[\vx\vy^\ast]\right]^\ast \cB[G-M] \\
            &= \sum_{i=1}^4 \braket{\vx_i, (M\SS[G-M](G-M)-M\un{WG}) \vy_i}\\
            & =\landauOprec*{\xi\Lambda+\sqrt{\frac{\rho+\Lambda}{n\eta}}},
        \end{split}
    \end{equation}
    where we used that the image of \(\vx\vy^\ast\) under \((\cQ_\ast\cB^{-1})^\ast\) is of rank at most \(4\), hence it can be written as \(\sum_{i=1}^4 {\bm x}_i{\bm y}_i^\ast\) with vectors of bounded norm. Similarly, for general matrices \(A\) we find
    \begin{equation}
        \begin{split}
            \braket{A(G-M)} & = \braket{\left[(\cQ \cB^{-1})^\ast[A^\ast]\right]^\ast \cB[G-M]}\\
            & = \braket{\left[(\cQ_\ast \cB^{-1})^\ast[A^\ast]\right]^\ast (M\SS[G-M](G-M)-M\un{WG})}\\
            &=\landauOprec*{\frac{\rho+\Lambda}{n\eta}+\xi^2}.
        \end{split}
    \end{equation}
\end{subequations}

In the complementary case \(\rho\ll 1\) we similarly decompose
\begin{equation}\label{eq G-M decomp}
    \begin{split}
        G-M ={}& \cP[G-M] +\cP_\ast[G-M] + \cQ[G-M] =  \theta V_r + \cQ[G-M],\quad \theta:=\braket{V_l,G-M}.
    \end{split}
\end{equation}
Now we apply \(\cB\) to both sides of~\eqref{eq G-M decomp} and take the inner product with \(V_l\) to obtain
\begin{equation}\label{eq theta indirect}
    \braket{V_l,\cB[G-M]} = (1-m^2-u^2\abs{z}^2) \theta + \braket{V_l,\cB\cQ[G-M]}
\end{equation}
from~\eqref{B[G-M] eq}. For the spectral projection \(\cQ\) we find
\begin{equation}
    \begin{split}
        \cB^{-1}\cQ &= \cQ \cB^{-1} = \begin{pmatrix}
            0 & 0 \\ B_3 & 1
        \end{pmatrix}, \quad B_3 = \frac{mu}{m^2+u^2\abs{z}^2}\begin{pmatrix}
            z & z \\ \ov z & \ov z
        \end{pmatrix}.
    \end{split}
\end{equation}
Thus it follows that
\begin{equation}
    \norm{\cB^{-1}\cQ}_{\norm{\cdot}\to\norm{\cdot}} \lesssim \frac{\abs{m u z}}{\abs{m^2 + u^2\abs{z}^2}} \lesssim 1
\end{equation}
since in the regime \(\rho\ll1\) we have \(\abs{1-m^2-u^2\abs{z}^2}\ll1\) due to \(\abs{\Im u^2}\ll 1\) which follows by a simple calculation.

By using~\eqref{B[G-M] eq} in~\eqref{eq theta indirect} it follows that
\begin{equation}\label{theta eq}
    \abs{\theta} \prec \frac{1}{\rho}\Bigl(\frac{\rho+\Lambda}{n\eta}+\xi^2\Bigr)
\end{equation}
from~\eqref{eq init bound},~\eqref{fluctuation av bound} since, due to~\(\abs{\abs{z}-1}\gtrsim \epsilon\), we have \(\abs{1-m^2-u^2\abs{z}^2}\ge\rho\) according to~\eqref{eq beta lower bound lemma}. For general vectors \(\vx,\vy\) it follows from~\eqref{eq G-M decomp},~\eqref{theta eq} and inserting \(1=\cB^{-1}\cB\) similarly to~\eqref{eq bootstrap ast} that
\begin{subequations}\label{eq bootstrap}
    \begin{equation}
        \begin{split}
            \braket{\vx,(G-M)\vy} &= \landauOprec*{\frac{\rho+\Lambda}{\rho n\eta}+\frac{\xi^2}{\rho}} + \braket{\left[(\cQ \cB^{-1})^\ast[\vx\vy^\ast]\right]^\ast \cB[G-M]} \\
            &= \landauOprec*{\frac{\rho+\Lambda}{\rho n\eta}+\frac{\xi^2}{\rho}} +\sum_{i=1}^4 \braket{\vx_i, (M\SS[G-M](G-M)-M\un{WG}) \vy_i}\\
            & =\landauOprec*{\frac{\rho+\Lambda}{\rho n\eta}+\frac{\xi^2}{\rho}+\xi\Lambda+\sqrt{\frac{\rho+\Lambda}{n\eta}}},
        \end{split}
    \end{equation}
    and
    \begin{equation}
        \begin{split}
            \braket{A(G-M)} & =\landauOprec*{\frac{\rho+\Lambda}{\rho n\eta}+\frac{\xi^2}{\rho}} + \braket{\left[(\cQ \cB^{-1})^\ast[A^\ast]\right]^\ast \cB[G-M]}\\
            & =\landauOprec*{\frac{\rho+\Lambda}{\rho n\eta}+\frac{\xi^2}{\rho}} + \braket{\left[(\cQ \cB^{-1})^\ast[A^\ast]\right]^\ast (M\SS[G-M](G-M)-M\un{WG})}\\
            &=\landauOprec*{\frac{\rho+\Lambda}{\rho n\eta}+\frac{\xi^2}{\rho}}.
        \end{split}
    \end{equation}
\end{subequations}

By using the bounds in~\eqref{eq bootstrap ast} and~\eqref{eq bootstrap} in the two complementary regimes we improve the input bound in~\eqref{eq init bound}. We can iterate this procedure and obtain
\begin{equation}
    \abs{\braket{\vx,(G-M)\vy}}\prec \frac{1}{n\eta}+\sqrt{\frac{\rho}{n\eta}}, \qquad \abs{\braket{A(G-M)}} \prec\frac{1}{n\eta}.
\end{equation}
In order to make sure the iteration yields an improvement one needs an priori bound on \(\xi\) of the form \(\xi\ll 1\) since otherwise \(\xi^2\) is difficult to control. For large \(\eta\) such an a priori bound is trivially available which can then be iteratively bootstrapped by monotonicity down to the optimal \(\eta\gg n^{-1}\). For details on this standard argument the reader is referred to e.g.~\cite[Section~3.3]{MR4089499}. Then the local law for any \(\eta>0\) readily follows by exactly the same argument as in~\cite[Appendix A]{Cipolloni2020_2}. This completes the proof of Theorem~\ref{theo:Gll}.\qed%

\begin{proof}[Proof of Lemma~\ref{lemma B analysis}]
    The fact that \(\cB\) has the eigenvalue \(1\) with multiplicity 2, and the claimed form of the remaining two eigenvalues and corresponding eigenvectors can be checked by direct computations.
    Taking the imaginary part of~\eqref{eq m} we have
    \begin{equation}\label{eq im m eq}
        (1-\abs{m}^2 - \abs{u}^2 \abs{z}^2) \Im m = (\abs{m}^2 + \abs{u}^2\abs{z}^2) \Im w,
    \end{equation}
    which implies
    \begin{equation}\label{mubound}
        \abs{m}^2 + \abs{u}^2\abs{z}^2< 1, \qquad \lim_{\Im w\to 0} (\abs{m}^2+\abs{u}^2\abs{z}^2) =1, \quad \Re w\in\overline{\supp\rho}
    \end{equation}
    as \(\Im m\) and \(\Im w\) have the same sign. Here \(\supp\rho\) should be understood as the support of the self-consistent density of states, as defined in~\eqref{eq:scdos}, restricted to the real axis.
    The second bound in~\eqref{eq beta lower bound lemma} then follows from~\eqref{mubound} and
    \begin{equation}\label{beta lower bound rho}
        \abs{1-m^2-u^2\abs{z}^2}\ge \Re(1-m^2-u^2\abs{z}^2) = 1-(\Re m)^2 + (\Im m)^2 -\Re(u^2)\abs{z}^2\gtrsim (\Im m)^2.
    \end{equation}

    The bound~\eqref{beta lower bound rho} can be improved in the case \(\rho\ll1\) if \(w\) is near a \emph{regular edge} of \(\rho\), i.e.\ where \(\rho\) locally vanishes as a square-root. According to~\cite[Eq.~(15b)]{Cipolloni2020} the density \(\rho\) has two regular edges \(\pm\sqrt{\ed_+}\) if \(\abs{z}\le 1-\epsilon\), and four regular edges in \(\pm\sqrt{\ed_+},\pm\sqrt{\ed_-}\) for \(\abs{z}\ge 1+\epsilon\), where
    \[ \ed_\pm := \frac{8(1-\abs{z}^2)^2 \pm (1+8\abs{z}^2)^{3/2}-36(1-\abs{z}^2)+27}{8\abs{z}^2}\gtrsim 1.\]
    By the explicit form of \(\ed_\pm\) it follows that \(\ed_\pm\gtrsim 1\) whenever \(\abs{1-\abs{z}}\ge \epsilon\).
    In contrast, if \(\abs{z}=1\), then \(\rho\) has a \emph{cusp} singularity in \(0\) where it locally vanishes like a cubic root. Near a regular edge we have \(\Im m \lesssim \sqrt{\Im w}\), and therefore from~\eqref{eq im m eq}
    \[ (1- \abs{m}^2 - \abs{u}^2 \abs{z}^2 ) \gtrsim \sqrt{ \Im w } \gtrsim \Im m \]
    and it follows that
    \[ \abs{1- m^2 - u^2 \abs{z}^2} \gtrsim \Im m,\]
    proving also the first inequality in~\eqref{eq beta lower bound lemma}.
\end{proof}

\section{CLT for resolvents: Proof of Proposition~\ref{prop clt resolvent}}\label{sec:CLTres}
The goal of this section is to prove the CLT for resolvents, as stated in Proposition~\ref{prop clt resolvent}. The proof is very similar to~\cite[Section~\ref{cplx-sec:CLTres}]{1912.04100} and we focus on the differences specific to the real case. Within this section we consider resolvents \(G_1,\dots,G_p\) with \(G_i=G^{z_i}(\ii\eta_i)\) and \(\eta_i\ge n^{-1}\). As a first step we recall the leading-order approximation of  \(G=G_i\)
\begin{equation}\label{eq G-M replacement}
    \braket{G-M}=-\braket{\un{WG}A} + \landauOprec*{\frac{1}{\abs{\beta}(n\eta)^2}}, \quad A:= (\cB^\ast)^{-1}[1]^\ast M
\end{equation}
from~\cite[Eq.~\eqref{cplx-eq G-M exp}]{1912.04100}, where the stability operator \(\cB\) has been defined in~\eqref{cal B}. Here \(\beta\) is the eigenvalue of \(\cB\) with eigenvector \((1,1,0,0)\) and is bounded by (see~\cite[Eq.~\eqref{cplx-beta bound}]{1912.04100})
\begin{equation}\label{beta bound}
    \abs{\beta} \gtrsim \abs{1-\abs{z}}+\eta^{2/3}.
\end{equation}
One important input for the proof of Proposition~\ref{prop clt resolvent} is a lower bound on the eigenvalues of the stability operator \(\wh\cB\), defined in~\eqref{eq:stabop12}, the proof of which we defer to the end of the section. Note that the two-body stability operator \(\wh\cB\) and its eigenvalues \(\wh\beta,\wh\beta_\ast\) are consistently decorated by hats (\(\wh{\cdot}\)) to distinguish them from their one-body analogues \(\cB,\beta\). We will consistently equip \(\cB, \wh\cB\) and their eigenvalues, \(\beta, \wh\beta, \wh\beta_\ast\) with indices when instead of \(M\) they are defined with the help of \(M_i= M^{z_i} (w_i)\); e.g.\ \(\wh\beta_*^{1i}\) is
the lowest eigenvalue of \(\wh\cB_{1i} = \wh\cB(z_1, z_i, w_1, w_i)\) defined analogously to~\eqref{eq:stabop12}.
\begin{lemma}\label{lem:lowbeta}
    For $z_1,z_2\in\mathbf{C}$, $w_1,w_2\in \mathbf{C}\setminus\mathbf{R}$ such that \(\abs{z_i},\abs{w_i}\lesssim 1\) the two non-trivial eigenvalues \(\wh\beta,\wh\beta_\ast\) of \(\wh\cB\) satisfy
    \begin{equation}\label{eq beta lower bound}
        \min\{\Re{\wh\beta},\Re{\wh\beta_\ast} \}\gtrsim \abs{z_1-z_2}^2+ \min\{ \abs{w_1+\ov{w_2}},\abs{w_1-\ov{w_2}}\}^2 + \abs{\Im w_1} + \abs{\Im w_2}
    \end{equation}
\end{lemma}
\begin{proof}[Proof of Proposition~\ref{prop clt resolvent}]
    The proof of Proposition~\ref{prop clt resolvent} goes in two steps. First, we use~\eqref{eq G-M replacement} and a cumulant expansion in order to prove the asymptotic representation of the expectation in~\eqref{EG exp}. In the second step we then turn to the computation of higher moments and establish an asymptotic Wick theorem in the form of~\eqref{eq CLT resovlent}.

    We use the notation $\Delta^{ab}$ for the matrix $(\Delta^{ab})_{cd}=\delta_{ac}\delta_{bd}$ and decompose \(W=\sum_{ab}w_{ab}\Delta^{ab}\). For each \(a,b\) we then perform a cumulant expansion and obtain
    \begin{equation}\label{eq real cumulant expansion}
        \E\braket{\un{WG}A}= -\frac{1}{n}\sump_{ab} \E \braket{\Delta^{ab}G\Delta^{ab}GA} + \sum_{k\ge 2} \sum_{ab}\sum_{\bm\alpha\in\{ab,ba\}^k}\frac{\kappa(ab,\bm\alpha)}{k!}\E\partial_{\bm\alpha}\braket{\Delta^{ab} GA},
    \end{equation}
    which has an additional term compared to the complex case~\cite[Eq.~\eqref{cplx-eq single WGA exp}]{1912.04100} since the self-renormalisation~\eqref{self renorm} was chosen such that it only takes the \(\kappa(ab,ba)=1\) and not the \(\kappa(ab,ab)=1\) cumulant into account. Here \(\kappa(ab,cd,ef,\ldots)\) denotes the joint cumulant of the random variables \(w_{ab},w_{cd},w_{ef},\ldots\), and we denote partial derivatives by \(\partial_{\bm\alpha}:=\partial_{w_{\alpha_1}}\cdots \partial_{w_{\alpha_k}}\) for tuples \(\bm\alpha=(\alpha_1,\ldots,\alpha_k)\), with \(\alpha_i\in [n]\times [n]\). In~\eqref{eq real cumulant expansion} we introduced the notation
    \[ \sump_{ab}:= \sum_{a\le n}\sum_{b>n} + \sum_{a>n}\sum_{b\le n}.\]
    We note that by Assumption~\ref{ass:1} the cumulants \(\kappa(\alpha_1,\ldots,\alpha_k)\) satisfy the scaling
    \begin{equation}\label{kappa scaling}
        \abs*{\kappa(\alpha_1,\dots,\alpha_k)} \lesssim n^{-k/2}.
    \end{equation}
    For the second term in~\eqref{eq real cumulant expansion} we find exactly as in~\cite[Eqs.~\eqref{cplx-g-m single ref}--\eqref{cplx-MA eq}]{1912.04100} that
    \begin{equation}\label{exp kge2}
        \sum_{k\ge 2} \sum_{ab}\sum_{\bm\alpha\in\{ab,ba\}^k}\frac{\kappa(ab,\bm\alpha)}{k!}\partial_{\bm\alpha}\braket{\Delta^{ab} GA} =\frac{\ii \kappa_4}{4n} \partial_\eta(m^4) + \landauOprec*{\frac{1}{\abs{\beta}}\Bigl(\frac{1}{n^{3/2}(1+\eta)}+\frac{1}{(n\eta)^2}\Bigr)}.
    \end{equation}
    For the first term in~\eqref{eq real cumulant expansion}, which is new compared to~\cite[Eq.~\eqref{cplx-eq single WGA exp}]{1912.04100}, we rewrite
    \begin{equation*}
        \frac{1}{n}\sump_{ab} \braket{\Delta^{ab}G\Delta^{ab}GA} = \frac{1}{n} \braket{GA E G^t E' }=\frac{1}{n} \braket{G^z A E G^{\ov z} E' },
    \end{equation*}
    where we used that \((G^z)^t=G^{\ov z}\), and the convention that formulas containing \((E,E')\) are understood so that the matrices \(E,E'\) are summed over the assignments \((E,E')=(E_1,E_2)\) and \((E,E')=(E_2,E_1)\) with
    \[ E_1:=\begin{pmatrix}
            1 & 0 \\0&0
        \end{pmatrix}\qquad E_2:=\begin{pmatrix}
            0 & 0 \\0&1
        \end{pmatrix}.\]
    From the local law~\cite[Theorem~\ref{cplx-thm local law G2}]{1912.04100} for products of resolvents and the bound on \(\abs{\wh\beta_\ast}\) from Lemma~\ref{lem:lowbeta} we can thus conclude
    \begin{equation}\label{eq additional exp term}
        \begin{split}
            \frac{1}{n}\sump_{ab} \braket{\Delta^{ab}G\Delta^{ab}GA} &= \frac{1}{n}\braket{M_{AE}^{z,\ov z}E'} + \landauOprec*{\frac{1}{\abs{z-\ov z}^2} \frac{1}{(n\eta)^2}}\\
            &= \frac{m}{n}\frac{ m^4  + m^2 u^2 \abs{z}^2 -2u^4 \abs{z}^4 +  2 u^2 (x^2-y^2)  }{(1-m^2- u^2\abs{z}^2 )(1+u^4\abs{z}^4-m^4 - 2 u^2(x^2-y^2)  )}\\
            &\qquad+ \landauOprec*{\frac{1}{\abs{z-\ov z}^2} \frac{1}{(n\eta)^2}},
        \end{split}
    \end{equation}
    where \(z=x+\ii y\), and the second step follows by explicitly computing the inverse
    \[ M_{AE}^{z,\ov z} = (1-M^z\SS[\cdot]M^{\ov z})^{-1}[M^z A E M^{\ov z}]\]
    in terms of the entries of \(M\), noting that \(m^{z}=m^{\ov z}\) and \(u^{z}=u^{\ov z}\). Then, using the definition \(v:=-\ii m>0\) and that
    \[ \abs{z}^2u^2+v^2=u,\qquad u'=-\frac{2uv}{1+u-\abs{z}^2u^2}, \qquad v^2=u(1-\abs{z}^2u)\]
    we obtain
    \begin{equation}\label{exp deriv derivation}
        \begin{split}
            &m\frac{ m^4  + m^2 u^2 \abs{z}^2 -2u^4 \abs{z}^4 +  2 u^2 (x^2-y^2)  }{(1-m^2- u^2\abs{z}^2 )(1+u^4\abs{z}^4-m^4 - 2 u^2(x^2-y^2)  )} \\
            &\qquad = -\frac{\ii u'}{2} \frac{u-3\abs{z}^2u^2+2u(x^2-y^2)}{1-u^2+2u^3\abs{z}^2-2u^2(x^2-y^2)}.
        \end{split}
    \end{equation}
    Now~\eqref{EG exp} follows from combining~\eqref{eq G-M replacement} and~\eqref{eq real cumulant expansion}--\eqref{exp deriv derivation}.

    We now turn to the computation of higher moments for which we recall from~\eqref{eq G-M replacement} and~\eqref{EG exp} that
    \begin{equation}\label{eq prod G-EG exp}
        \begin{split}
            \prod_{i\in[p]} \braket{G_i-\E G_i} &= \prod_{i\in[p]} \braket{G_i-M_i-\cE_i} + \landauOprec*{\frac{\psi}{n\eta}}\\
            & = \prod_{i\in[p]} \braket{-\un{WG_i}A_i-\cE_i} + \landauOprec*{\frac{\psi}{n\eta}}
        \end{split}
    \end{equation}
    with \(A_i\) as in~\eqref{eq G-M replacement} and \(\cE_i\) as in~\eqref{cE def}, and
    \begin{equation}
        \label{eq:smallpsi}
        \psi:= \prod_i \Bigl(\frac{1}{\abs{\beta_i}}+\frac{1}{(\Im z_i)^2}\Bigr)\frac{1}{n\eta_i}\le \prod_i \Bigl(\frac{1}{\abs{1-\abs{z_i}}}+\frac{1}{(\Im z_i)^2}\Bigr)\frac{1}{n\eta_i}
    \end{equation}
    with the bound on \(\beta_i\) from~\eqref{beta bound}.
    We begin with the cumulant expansion of \(\un{W G_1}\) to obtain
    \begin{equation}\label{eq G-M first exp}
        \begin{split}
            &\E \prod_{i\in[p]}\braket{-\un{W G_i}A_i-\cE_i} \\
            &\quad= \E\biggl(\frac{1}{n}\sump_{ab}\braket{\Delta^{ab}G_1 \Delta^{ab}G_1A_1}-\braket{\cE_1}\biggr)\prod_{i\ne 1}\braket{-\un{W G_i}A_i-\cE_i} \\
            &\qquad+ \sum_{i\ne1}\E\wh\E \braket{\wh W G_1 A_1}\braket{\wh W G_i A_i-\un{WG_i \wh W G_i} A_i}\prod_{j\ne 1,i}\braket{-\un{W G_j}A_j-\cE_j}\\
            & \qquad +\sum_{k\ge 2}\sum_{ab}\sum_{\bm\alpha\in\{ab,ba\}^k} \frac{\kappa(ba,\bm \alpha)}{k!} \E \partial_{\bm\alpha}\Bigl[\braket{-\Delta^{ba}G_1A_1}\prod_{i\ne 1}\braket{-\un{W G_i}A_i-\cE_i} \Bigr],
        \end{split}
    \end{equation}
    where, compared to~\cite[Eq.~\eqref{cplx-eq g-m 2 first exp}]{1912.04100}, the first line on the rhs.\ has an additional term specific to the real case, and \(\wh W\), as opposed to \(\wt W\) in~\eqref{self renorm}, is the Hermitisation of an independent real Ginibre matrix \(\wh X\) with expectation \(\wh\E\). The expansion of the third line on the rhs.\ of~\eqref{eq G-M first exp} is completely analogous to~\cite{1912.04100} since for cumulants of degree at least three nothing specific to the complex case was used. Therefore we obtain, from combining\footnote{Note that the definition of \(\cE\) in~\cite[Eq.~\eqref{cplx-cE def}]{1912.04100} differs from~\eqref{cE def} in the present paper.}~\cite[Eqs.~\eqref{cplx-k=2 est},~\eqref{cplx-eq kappa ge 3 conclusion}]{1912.04100}, that
    \begin{equation}\label{higher cum recall}
        \begin{split}
            &\sum_{k\ge 2}\sum_{ab}\sum_{\bm\alpha\in\{ab,ba\}^k} \frac{\kappa(ba,\bm \alpha)}{k!} \E \partial_{\bm\alpha}\Bigl[\braket{-\Delta^{ba}G_1A_1}\prod_{i\ne 1}\braket{-\un{W G_i}A_i-\cE_i} \Bigr] \\
            & = -\frac{\ii\kappa_4}{4n}\partial_{\eta_1}(m_1^4) \E \prod_{i\ne1} \braket{-\un{WG_i}A_i-\cE_i}+\sum_{i\ne1} \frac{\kappa_4 U_1 U_i}{2n^2} \E \prod_{j\ne1,i}\braket{-\un{WG_j}A_j-\cE_j}+\landauO*{\frac{n^\xi\psi}{\sqrt{n\eta_\ast}}},
        \end{split}
    \end{equation}
    where
    \[U_i := -\sqrt{2}\braket{M_i}\braket{M_i A_i}=\frac{\ii}{\sqrt{2}}\partial_{\eta_i}m_i^2.\]
    Recall the definition of \(\cE_i\) in~\eqref{cE def}, then using~\eqref{eq additional exp term}--\eqref{exp deriv derivation} and~\eqref{higher cum recall} in~\eqref{eq G-M first exp} we thus have
    \begin{equation}\label{eq wick reduction}
        \begin{split}
            &\E \prod_{i\in[p]}\braket{-\un{W G_i}A_i-\cE_i} \\
            &= \sum_{i\ne1}\E\Bigl(\frac{\kappa_4 U_1 U_i}{2n^2}+\wh\E \braket{\wh W G_1 A_1}\braket{\wh W G_i A_i-\un{WG_i \wh W G_i} A_i}\Bigr)\prod_{j\ne 1,i}\braket{-\un{W G_j}A_j-\cE_j}+\landauO*{\frac{n^\xi\psi}{\sqrt{n\eta_\ast}}}.
        \end{split}
    \end{equation}

    It remains to consider the variance term in~\eqref{eq wick reduction} for which we use the identity
    \begin{equation}\label{eq WA WB real comp}
        \wh\E \braket{\wh W A} \braket{\wh W B} = \frac{1}{2n^2} \braket{A E (B + B^t) E'} = \frac{\braket{A E_1 (B+B^t) E_2} + \braket{A E_2 (B+B^t) E_1}}{2n^2}
    \end{equation}
    in order to compute
    \begin{equation}\label{eq variance step 1}
        \begin{split}
            &\wh\E \braket{\wh W G_1 A_1}\braket{\wh W G_i A_i-\un{WG_i \wh W G_i} A_i} \\
            &\quad = \frac{1}{2n^2}\braket{G_1 A_1 E (G_i A_i + A_i^t G_i^t)E'-G_1 A_1 E (\un{G_i A_i W G_i}+\un{G_i^t W A_i^t G_i^t})E'},
        \end{split}
    \end{equation}
    where, compared to~\cite[Eqs.~\eqref{cplx-eq tr W tr W}--\eqref{cplx-eq Wick Gauss term}]{1912.04100}, there is an additional term with transposition. Here the self-renormalisation e.g.\ in \(\un{G_i A_i W G_i}\) is defined analogously to~\eqref{self renorm} with the derivative acting on both \(G_i\)'s. For the second term in~\eqref{eq variance step 1} we identify the leading order contribution using the fact that \(G^z(w)^t=G^{\ov z}(w)\) and denoting \(G_{\bar i}=G^{\ov{z_i}}(\ii\eta_i)\) as
    \begin{equation}\label{eq three G self renorm}
        \begin{split}
            \braket{G_1 A_1 E (\un{G_i A_i W G_i}+\un{G_i^t W A_i^t G_{\bar i}^t})E'}&=- \braket{G_1\SS[G_1 A_1 E G_i A_i] G_i E'+G_1 \SS[G_1 A_1 E G_{\bar i}] A_i^t G_{\bar i} E'} \\
            &\quad +\braket{\un{G_1 A_1 EG_i A_i W G_i E'} + \un{ G_1 A_1 EG_{\bar i} W A_i^t G_{\bar i} E'}}
        \end{split}
    \end{equation}
    for which we use the local law from Theorem~\ref{thm local law G2} to conclude that the main terms in~\eqref{eq variance step 1} are
    \begin{equation}\label{eq local law appl}
        \begin{split}
            &\braket{G_1 A_1 E (G_i A_i + A_i^t G_{\bar i})E'+G_1\SS[G_1 A_1 E G_i A_i] G_i E'+G_1 \SS[G_1 A_1 E G_{\bar i}] A_i^t G_{\bar i} E'} \\
            &\quad= \wh V_{1,i} + \landauOprec*{\frac{1}{n\abs{\wh\beta_\ast^{1i}}^2\eta_\ast^{1i}|\eta_1\eta_i|^{1/2}}+\frac{1}{n^2\abs{\wh\beta_\ast^{1i}}^2(\eta_\ast^{1i})^2|\eta_1\eta_i|}} \\
            &\wh V_{1,i}:= \braket{M_{A_1 E}^{z_1,z_i} A_i E' + M_{A_1 E A_i^t}^{z_1,\ov{z_i}}  E' + \SS[M_{A_1 E}^{z_1,z_i}A_i]M_{E'}^{z_i,z_1}+ \SS[M_{A_1 E}^{z_1,\ov{z_i}}]A_i^t M_{E'}^{\ov{z_i},z_1} },
        \end{split}
    \end{equation}
    where \(\abs{\wh\beta_\ast^{1i}}\gtrsim |z_1-z_i|^2\) from Lemma~\ref{lem:lowbeta}, and $\eta_*^{1i}:=\min\{\eta_1,\eta_i\}$. By an explicit computation similarly to~\cite[Eq.~\eqref{cplx-eq:vw}]{1912.04100} it follows that
    \begin{equation}
        \wh V_{1,i} = V(z_1,z_i,\eta_1,\eta_i) + V(z_1,\ov{z_i},\eta_1,\eta_i)
    \end{equation}
    with \(V\) being exactly as in the complex case, i.e.\ as in~\eqref{eq:exder}. For the error term in~\eqref{eq three G self renorm} we claim that
    \begin{equation}\label{eq three G self-renorm claim}
        \E\abs{ \braket{\un{G_1 A_1 EG_i A_i W G_i E'}}}^2+ \E\abs{\braket{\un{ G_1 A_1 EG_{\bar i} W A_i^t G_{\bar i} E'}}}^2 \lesssim \Bigl(\frac{1}{n\eta_1\eta_i\eta_\ast^{1i}}\Bigr)^2.
    \end{equation}
    The CLT for resolvents, as stated in~\eqref{eq CLT resovlent} follows from inserting~\eqref{eq variance step 1}--\eqref{eq three G self-renorm claim} into~\eqref{eq wick reduction}, and iteration of~\eqref{eq wick reduction} for the remaining product.

    In order to conclude the proof of Proposition~\ref{prop clt resolvent} it remains to prove~\eqref{eq three G self-renorm claim}. Introduce the shorthand notation \(G_{i1i}\) for generic finite sums of products of \(G_i, G_1, G_i\) (or \(G_{\bar i}\) in place of \(G_i\)) with arbitrary bounded deterministic matrices, e.g.\ \(G_i E' G_1 A_1 EG_i A_i\) appearing in the first term in~\eqref{eq three G self-renorm claim}. We will prove the more general claim
    \begin{equation}\label{Gi1i claim}
        \E\abs{\braket{\un{WG_{i1i}}}}^2 \lesssim \Bigl(\frac{1}{n\eta_1\eta_i\eta_\ast^{1i}}\Bigr)^2.
    \end{equation}
    The proof is similar to~\cite[Eq.~\eqref{cplx-eq G1Gii claim}]{1912.04100}. Therefore we focus on the differences. In the cumulant expansion of~\eqref{Gi1i claim} there is an additional term compared to~\cite[Eq.~\eqref{cplx-G1Gii first}]{1912.04100} given by
    \begin{equation}\label{eq GGG additonal term}
        \begin{split}
            &\frac{1}{n}\sump_{ab} \E\braket{\Delta^{ab} G_i \Delta^{ab}G_{i1i}+\Delta^{ab} G_{i1} \Delta^{ab}G_{1i}+\Delta^{ab} G_{i1i} \Delta^{ab}G_{i} }\braket{\un{WG_{i1i}}} \\
            &\quad=\frac{1}{n} \E\braket{G_{1iii}+G_{i1i1}}\braket{\un{WG_{i1i}}},
        \end{split}
    \end{equation}
    where we combined two terms of type \(G_{1iii}\) into one since in our convention \(G_{1iii}\) is a short-hand notation for generic sums of products. We now perform another cumulant expansion of~\eqref{eq GGG additonal term} to obtain
    \begin{equation}\label{eq GGG additonal term 2}
        \begin{split}
            &\frac{1}{n}\E\braket{G_{1iii}+G_{i1i1}}\braket{\un{WG_{i1i}}} \\
            &= \frac{1}{n^2}\E\braket{G_{1iii}+G_{i1i1}}^2 \\
            &\quad+ \frac{1}{n}\E\wt\E \braket{\wt W (G_{1iii1}+G_{iii1i}+G_{ii1ii}+G_{i1iii}+G_{1i1i1}+G_{i1i1i})}\braket{\wt W G_{i1i}}\\
            &\quad + \sum_{k\ge 2} \landauO*{\frac{1}{n^{(k+3)/2}}} \sump_{ab} \sum_{\bm\alpha\in\{ab,ba\}^k} \E \partial_{\bm\alpha} \Bigl[\braket{G_{1iii}+G_{i1i1}}\braket{\Delta^{ab}G_{i1i}}\Bigr],
        \end{split}
    \end{equation}
    where the first line on the rhs.\ corresponds to the term where the remaining \(W\) acts on \(G_{i1i}\) within its own trace as in~\eqref{eq GGG additonal term}, and in the last line we used the scaling bound~\eqref{kappa scaling} for \(\kappa\). In order to estimate~\eqref{eq GGG additonal term} we recall~\cite[Lemma~\ref{cplx-lemma general products}]{1912.04100}.
    \begin{lemma}\leavevmode\label{lemma general products}
        Let $w_1,w_2,\dots$, $z_1,z_2,\dots$, denote arbitrary spectral parameters with $\eta_i=\Im w_i>0$. Let $G_j=G^{z_j}(w_j)$, then with $G_{j_1\dots j_k}$ we denote generic products of resolvents $G_{j_1},\dots,G_{j_k}$, or their adjoints/transpositions (in that order, each $G_{j_i}$ appears exactly once) with bounded deterministic matrices in between, e.g. $G_{1i1}=A_1G_1A_2G_i A_3G_1A_4$.
        \begin{enumerate}[label=(\roman*)]
            \begin{subequations}
                \item For \(j_1,\dots j_k\) we have the isotropic bound
                \begin{equation}\label{eq general iso bound}
                    \abs{\braket{\vx,G_{j_1\dots j_k}\vy}} \prec \norm{\vx}\norm{\vy}\sqrt{\eta_{j_1}\eta_{j_k}}\Bigl(\prod_{n=1}^k \eta_{j_n}\Bigr)^{-1}.
                \end{equation}
                \item For \(j_1,\dots,j_k\) and any \(1\le s< t\le k\) we have the averaged bound
                \begin{equation}\label{eq general av bound}
                    \abs{\braket{G_{j_1\dots j_k}}} \prec \sqrt{\eta_{j_{s}}\eta_{j_{t}}}\Bigl(\prod_{n=1}^k \eta_{j_n}\Bigr)^{-1}.
                \end{equation}
            \end{subequations}
        \end{enumerate}
    \end{lemma}
    Since only $\eta_1, \eta_i$ play a role within the proof of~\eqref{eq three G self-renorm claim}, we drop the indices from $\eta_*^{1i}$ and use the notation $\eta_*=\eta_*^{1i}$. For the first term in~\eqref{eq GGG additonal term 2} we use~\eqref{eq general av bound} to obtain
    \begin{equation}\label{eq new new term}
        \frac{1}{n^2}\abs{\braket{G_{1iii}+G_{i1i1}}}^2 \prec \frac{1}{n^2\eta_1^2\eta_i^2\eta_\ast^2}.
    \end{equation}
    Similarly for the second term we use~\eqref{eq WA WB real comp} and again~\eqref{eq general av bound} to bound it by
    \begin{equation}
        \begin{split}
            &\frac{1}{n}\abs{\wt\E \braket{\wt W (G_{1iii1}+G_{iii1i}+G_{ii1ii}+G_{i1iii}+G_{1i1i1}+G_{i1i1i})}\braket{\wt W G_{i1i}}}\\
            &\qquad \prec\frac{1}{n^3\eta_1^2\eta_i^2\eta_\ast^3}\le \frac{1}{n^2\eta_1^2\eta_i^2\eta_\ast^2}
        \end{split}
    \end{equation}
    since \(\eta_\ast\ge 1/n\). Finally, for the last term of~\eqref{eq GGG additonal term 2} we estimate
    \begin{equation}\label{new kge2 term}
        \abs*{\landauO*{\frac{1}{n^{(k+7)/2}}} \sump_{ab}\sum_c \sum_{\bm\alpha} \partial_{\bm\alpha} \Bigl[(G_{1iii}+G_{i1i1})_{cc} (G_{i1i})_{ba}\Bigr]}\prec \frac{1}{n^{2}\eta_1^2\eta_i^2\eta_\ast^2}
    \end{equation}
    for any \(k\ge 2\). Indeed, for \(k\ge 3\) the claim~\eqref{new kge2 term} follows trivially from~\eqref{eq general iso bound} and the observation that the bound~\eqref{eq general iso bound} remains invariant under the action of derivatives. Indeed, differentiating a term like \((G_{i1i})_{ab}\) gives rise to the terms \((G_{i})_{aa}(G_{i1i})_{bb},(G_{i1})_{ab}(G_{1i})_{ab},\dots\) for all of which~\eqref{eq general iso bound} gives the same estimate as for \((G_{i1i})_{ab}\) since the presence of an additional factor of \(G_1\) or \(G_i\) is compensated by the fact that the same type of \(G\) appears two additional times as the first or last factor in some product. For the \(k=2\) case we observe that by parity at least one factor will be off-diagonal in the sense that it has two distinct summation indices from \(\{a,b,c\}\) giving rise to an additional factor of \((n\eta_\ast)^{-1/2}\) by summing up one of the indices with the Ward identity. For example, for the term with \((G_{1iii})_{cc}(G_{i1})_{bb}(G_{1i})_{aa}(G_i)_{ba}\) we estimate
    \[ \begin{split}
            n^{-9/2}\abs*{\sump_{ab}\sum_c (G_{1iii})_{cc}(G_{i1})_{bb}(G_{1i})_{aa}(G_i)_{ba}}&\prec n^{-9/2} \frac{n}{\eta_1^{3/2}\eta_i^{7/2}}\sump_{ab} \abs{(G_i)_{ba}}\\
            &\le n^{-3} \frac{1}{\eta_1^{3/2}\eta_i^{7/2}} \sum_b \sqrt{\sum_a \abs{(G_i)_{ba}}^2}\\
            &=n^{-3} \frac{1}{\eta_1^{3/2}\eta_i^{4}} \sum_b \sqrt{(\Im G_i)_{bb}} \prec \frac{1}{n^2\eta_1^{3/2}\eta_i^{4}}.
        \end{split} \]
    Thus, in general we obtain a bound of
    \[ \frac{1}{n^{3/2}}\Bigl(\frac{1}{\eta_1^{3/2}\eta_i^{7/2}}+\frac{1}{\eta_1^{5/2}\eta_i^{5/2}}\Bigr)\frac{1}{\sqrt{n\eta_\ast}}\lesssim \frac{1}{n^{2}\eta_1^2\eta_i^2\eta_\ast^2}.\]
    By combining~\eqref{eq new new term}--\eqref{new kge2 term} we obtain a bound of \((n\eta_1\eta_i\eta_\ast)^{-2}\) on the additional term~\eqref{eq GGG additonal term}. The remaining terms can be estimated as in~\cite[Eq.~\eqref{cplx-eq G1Gii claim}]{1912.04100} and we conclude the proof of~\eqref{Gi1i claim} and thereby Proposition~\ref{prop clt resolvent}.
\end{proof}
\begin{proof}[Proof of Lemma~\ref{lem:lowbeta}]
    The claim~\eqref{eq beta lower bound} is equivalent to the claim
    \begin{equation}\label{t ineq effective}
        \max\{\Re \tau,\Re \tau_\ast\} \le 1-c\bigl[\abs{z_1-z_2}^2+ \min\{ \abs{w_1+\ov{w_2}}^2,\abs{w_1-\ov{w_2}}^2\} + \abs{\Im w_1} + \abs{\Im w_2} \bigr], \quad c>0,
    \end{equation}
    where \(\tau,\tau_\ast\) are the eigenvalues of the matrix
    \begin{equation}
        R := \begin{pmatrix}
            z_1 \ov{z_2} u_1 u_2 & m_1 m_2 \\ m_1 m_2 & \ov{z_1} z_2 u_1 u_2
        \end{pmatrix},
    \end{equation}
    thus $\widehat{\beta}=1-\tau$, $\widehat{\beta}_*=1-\tau_*$.
    We first check that~\eqref{t ineq effective} holds true ineffectively, i.e.\ with \(c=0\). We claim that
    \begin{equation}\label{Re spec ineq}
        \max\Re\Spec(A) \le \lambda_{\max }\Bigl(\frac{A+A^\ast}{2}\Bigr) := \max\Spec\Bigl(\frac{A+A^\ast}{2}\Bigr)
    \end{equation}
    holds for any square matrix \(A\). Indeed, suppose that \(A\vx=\lambda \vx\), \(\norm{\vx}=1\) and \((A+A^\ast)/2\le M\) in the sense of quadratic forms. We then compute
    \[0\ge \braket*{\vx,\Bigl(\frac{A+A^\ast}{2}-M\Bigr)\vx}=\frac{\braket{\vx,A\vx}+\braket{A\vx,\vx}}{2}-M = \Re\lambda-M, \]
    from which~\eqref{Re spec ineq} follows by choosing \(M\) to be the largest eigenvalue of \((A+A^\ast)/2\).

    Since \(R\) is such that its entrywise real part is given by \(\Re R=(R+R^\ast)/2\), from~\eqref{Re spec ineq} we conclude the chain of inequalities
    \begin{subequations}
        \begin{align}
             & \max\{\Re \tau,\Re \tau_\ast\}\le\lambda_{\max }\begin{pmatrix}
                                                                   \Re(z_1 \ov{z_2} u_1 u_2) & \Re(m_1 m_2) \\ \Re(m_1 m_2) & \Re(\ov{z_1} z_2 u_1 u_2)
                                                               \end{pmatrix} \label{eq:0}                                                 \\
             & =(\Re u_1u_2)(\Re z_1\ov{z_2})+\sqrt{\bigl(\abs{\Im u_1 u_2}\abs{\Im z_1\ov{z_2}}+\abs{\Re m_1 m_2}\bigr)^2-2\abs{\Im u_1 u_2}\abs{\Im z_1\ov{z_2}}\abs{\Re m_1 m_2}} \label{eq:35} \\
             & \le(\Re u_1u_2)(\Re z_1\ov{z_2}) + \abs{\Im u_1 u_2}\abs{\Im z_1\ov{z_2}} + \abs{\Re m_1 m_2} \label{eq:4}                                                                          \\
             & \le\abs[\Big]{(\Re u_1u_2)(\Re z_1\ov{z_2}) + \abs{\Im u_1 u_2}\abs{\Im z_1\ov{z_2}}  }+ \abs{\Re m_1 m_2}\label{eq:45}                                                             \\
             & =\sqrt{\abs{z_1 z_2 u_1 u_2}^2- \bigl(\Re u_1 u_2 \abs{\Im z_1\ov{z_2}}-\Re z_1 \ov{z_2}\abs{\Im u_1 u_2}\bigr)^2 } + \sqrt{\abs{m_1 m_2}^2-[\Im m_1 m_2]^2} \label{eq:5}           \\
             & \le\abs{z_1 z_2 u_1 u_2} + \abs{m_1 m_2} \label{eq:6}                                                                                                                               \\
             & =\sqrt{(\abs{u_1 z_1}^2 + \abs{m_1}^2)(\abs{u_2 z_2}^2 + \abs{m_2}^2)-(\abs{u_1 z_1 m_2}-\abs{u_2 z_2 m_1})^2} \label{eq:7}                                                         \\
             & \le\sqrt{(\abs{m_1}^2+\abs{z_1 u_1}^2)(\abs{m_2}^2+\abs{z_2 u_2}^2)} \label{eq:8}                                                                                                   \\
             & \le 1,\label{eq:9}
        \end{align}
    \end{subequations}
    where in the last step we used~\eqref{mubound}.

    We now assume that for some \(0\le\epsilon\ll 1\) we have
    \begin{equation}\label{assumption}
        \max\{\Re \tau,\Re \tau_\ast\} \ge 1-\epsilon^2,
    \end{equation}
    i.e.\ that all inequalities in~\eqref{eq:0}--\eqref{eq:9} are in fact equalities up to an \(\epsilon^2\) error. The assertion~\eqref{t ineq effective} is then equivalent to
    \begin{equation}\label{eq z1 z2 claim}
        \abs{z_1-z_2} + \min\{ \abs{w_1+\ov{w_2}},\abs{w_1-\ov{w_2}}\} + \sqrt{\abs{\Im w_1}} + \sqrt{\abs{\Im w_2}} \lesssim \epsilon,
    \end{equation}
    the proof of which we present now.

    The fact that~\eqref{eq:8}--\eqref{eq:9} is \(\epsilon^2\)-saturated implies the saturation
    \begin{equation}\label{m z u almost id}
        \abs{m_i}^2+\abs{z_i u_i}^2 = 1 + \landauO{\epsilon^2},
    \end{equation}
    and, consequently,
    \begin{equation}\label{eq u sim 1}
        \abs{u_i}\sim 1.
    \end{equation}
    Indeed, suppose that \(\abs{u_i}\ll 1\), then on the one hand since \(u_i = u_i^2\abs{z_i}^2 - m_i^{2}\), it follows that \(\abs{m_i}\ll 1\), while on the other hand \(\abs{1-\abs{m_i}^2}\ll1\) from~\eqref{m z u almost id} which would be a contradiction. From~\eqref{eq im m eq} it follows that
    \[ \abs{m_i}^2  + \abs{u_i}^2\abs{z_i}^2 \le 1 - c\,\Im w_i,\]
    from which we conclude \(\abs{\Im w_1}+\abs{\Im w_2}\lesssim \epsilon^2\), i.e.\ the bound on the last two terms in~\eqref{eq z1 z2 claim}.
    The \(\epsilon^2\)-saturation of~\eqref{eq:7}--\eqref{eq:8} implies that
    \[\begin{split}
            \landauO{\epsilon} = \abs{u_1 z_1 m_2}-\abs{u_2 z_2 m_1} &= \sqrt{1-\abs{m_1}^2}\abs{m_2}-\sqrt{1-\abs{m_2}^2}\abs{m_1} +\landauO{\epsilon^2}\\
            &= \sqrt{1-\abs{u_1 z_1}^2}\abs{u_2 z_2}-\sqrt{1-\abs{u_2 z_2 }^2}\abs{u_1 z_1} +\landauO{\epsilon^2}.
        \end{split}\]
    Thus it follows that
    \begin{equation}\label{eq m1 m2}
        \abs{m_1}=\abs{m_2}+\landauO{\epsilon}, \qquad \abs{z_1 u_1}= \abs{z_2 u_2} + \landauO{\epsilon}.
    \end{equation}

    In the remainder of the proof we distinguish the cases
    \begin{enumerate}[label= (C\arabic*)]
        \item\label{case1} \(\epsilon\ll\abs{z_1}\) and \(\abs{m_1}\sim 1\),
        \item\label{case2} \(\abs{z_1}\lesssim\epsilon\),
        \item\label{case4} \(\abs{m_1}\lesssim\sqrt{\epsilon}\) and \(\abs{z_1}\sim 1\),
        \item\label{case5} \(\sqrt{\epsilon}\ll\abs{m_1}\ll1\) and \(\abs{z_1}\sim 1\),
    \end{enumerate}
    where we note that this list is exhaustive since \(\abs{z_1}\ll1\) implies \(\abs{m_1}\sim 1\) from~\eqref{m z u almost id}.

    In case~\ref{case1} we have \(\abs{z_2}\sim\abs{z_1}\) and \(\abs{m_1}\sim\abs{m_2}\sim1\) from~\eqref{eq u sim 1}--\eqref{eq m1 m2}. By the near-saturation of~\eqref{eq:5}--\eqref{eq:6} it follows that \(\Im m_1 m_2=\landauO{\epsilon}\) and therefore with~\eqref{eq m1 m2} that
    \begin{equation}\label{m1 pm m2}
        m_1=\pm\ov{m_2}+\landauO{\epsilon},
    \end{equation}
    hence \(\abs{\Re m_1 m_2}\sim 1\). From the \(\epsilon^2\)-saturation of~\eqref{eq:35}--\eqref{eq:4} and~\eqref{eq:5}--\eqref{eq:6} it then follows that
    \begin{equation}\label{eq z u relations zsmall}
        \abs{\Im u_1 u_2}\abs*{\Im\frac{ z_1\ov{z_2}}{\abs{z_1 z_2}}}=\landauO*{\frac{\epsilon^2}{\abs{z_1}^2}},\quad (\Re u_1 u_2)\abs*{\Im \frac{z_1\ov{z_2}}{\abs{z_1 z_2}}} = \Bigl(\Re \frac{z_1\ov{z_2}}{\abs{z_1 z_2}}\Bigr)\abs{\Im u_1u_2} + \landauO*{\frac{\epsilon}{\abs{z_1}}},
    \end{equation}
    and~\eqref{eq z u relations zsmall} implies
    \begin{equation}\label{eq Im u z eps}
        \abs{\Im u_1u_2}+\abs*{\Im\frac{ z_1\ov{z_2}}{\abs{z_1 z_2}}}\lesssim \frac{\epsilon}{\abs{z_1}}.
    \end{equation}
    Indeed, the first equality in~\eqref{eq z u relations zsmall} implies that at least one of the two factors is at most of size \(\epsilon/\abs{z_1}\ll 1\) in which case the second equality implies that the other factor satisfies the same bound since \(\abs{u_1 u_2}\sim 1\). Thus there exists some \(c\in\R\), \(\abs{c}\sim 1\) such that \(z_2=c z_1+\landauO{\epsilon}\) and \(u_2=\pm\abs{c}^{-1}\ov{u_1}+\landauO{\epsilon/\abs{z_1}}\) since the two proportionality constants \(c\) and \(\pm \abs{c}^{-1}\) are related by~\eqref{eq m1 m2}. On the other hand, from the MDE~\eqref{eq m} we have that
    \begin{equation}\label{eq u1 u2}
        u_2 = u_2^2\abs{z_2}^2 -m_2^2 = \ov{u_1}^2 \abs{z_1}^2 - \ov{m_1}^2 + \landauO{\epsilon} = \ov{u_1} + \landauO{\epsilon}
    \end{equation}
    and thus \(\abs{c}=1+\landauO{\epsilon/\abs{z_1}}\). Finally, since~\eqref{eq:4}--\eqref{eq:45} is assumed to be saturated up to an \(\epsilon^2\)-error, \(\Re u_1 u_2\) and \(\Re z_1 \ov{z_2}\) have the same sign which, together with~\eqref{eq u1 u2}, fixes \(c>0\), and we conclude \(z_2=z_1+\landauO{ \epsilon} \). Finally, with
    \begin{equation}\label{w eq}
        w_2 = \frac{m_2}{u_2} - m_2 =\pm\Bigl(\frac{\ov{m_1}}{\ov{u_1}}-\ov{m_1}\Bigr)+\landauO{\epsilon} = \pm\ov{w_1} + \landauO{\epsilon}
    \end{equation}
    the claim~\eqref{eq z1 z2 claim} follows.

    In case~\ref{case2} the conclusion \(z_2=z_1+\landauO{\epsilon}\) follows trivially from~\eqref{eq m1 m2} and~\eqref{eq u sim 1}. Next, just as in case~\ref{case1}, we conclude~\eqref{m1 pm m2} and therefore from~\eqref{eq m} that
    \[ u_2 = u_2^2 \abs{z_2}^2 - m_2^2 = -\ov{m_1}^2 + \landauO{\epsilon} = \ov{u_1}+\landauO{\epsilon},\]
    and thus~\eqref{eq z1 z2 claim} follows just as in~\eqref{w eq}.

    Finally, we consider the case \(\abs{m_i}\ll 1\), i.e.~\ref{case4} and~\ref{case5}. If \(\abs{m_i}\ll1\), then from~\eqref{m z u almost id}, \(\abs{1-\abs{z_i u_i}^2}\ll1\), and therefore from~\eqref{eq m}, \(\abs{1-\abs{u_i}}\ll1\) and consequently \(\abs{1-u_i\abs{z_i}^2}=\abs{m_i^2/u_i}\ll 1\) and \(\abs{1-u_i}+\abs{1-\abs{z_i}^2}\ll1\). If \(\abs{m_1}\lesssim \sqrt{\epsilon}\), then it follows from~\eqref{eq m1 m2} that also \(\abs{m_2}\lesssim\sqrt{\epsilon}\). From solving the equation~\eqref{eq m} for \(u_i\) we find
    \begin{equation}\label{eq u approx cusp}
        u_i= \frac{1+\sqrt{1+4\abs{z_i}^2 m_i^2}}{2\abs{z_i}^2} = \frac{1}{\abs{z_i}^2} + \landauO{\abs{m_i}^2},
    \end{equation}
    where the sign choice is fixed due to \(\abs{1-u_i}\ll 1\).

    In case~\ref{case4} from~\(\abs{m_i}\lesssim\sqrt{\epsilon}\) it follows that \(u_i=\abs{z_i}^{-2}+\landauO{\epsilon}\), and thus with~\eqref{eq:5}--\eqref{eq:6} and \(\Re u_1 u_2\sim1\) we can conclude
    \begin{equation}\label{Im z u relation}
        \abs{\Im z_1\ov{z_2}} = \frac{\Re z_1\ov{z_2}}{\Re u_1 u_2}\abs{\Im u_1 u_2}+\landauO{\epsilon} = \landauO{\epsilon}, \quad \abs{\Im u_1 u_2}=\landauO{\epsilon}.
    \end{equation}
    Together with~\eqref{eq m1 m2} and the saturation of~\eqref{eq:4}--\eqref{eq:45}, we obtain \(z_1=z_2+\landauO{\epsilon}\) and \(u_1=\ov{u_2}+\landauO{\epsilon}\) by the same argument as after~\eqref{eq Im u z eps}. Equation~\eqref{eq m} implies that \(m_2=\pm\ov{m_1}+\landauO{\epsilon}\) and we are able to conclude~\eqref{eq z1 z2 claim} just as in~\eqref{w eq}.

    In case~\ref{case5} from~\eqref{eq m1 m2} we have \(\abs{m_2}\sim\abs{m_1}\). By saturation of~\eqref{eq:5}--\eqref{eq:6} it follows that
    \[\Im \frac{m_1m_2}{\abs{m_1 m_2}} = \landauO*{\frac{\epsilon}{\abs{m_1}}}\]
    and therefore, together with~\eqref{eq m1 m2} we conclude that~\eqref{m1 pm m2} also holds in this case. Now we use the saturation of~\eqref{eq:35}--\eqref{eq:4} to conclude
    \[ \abs{\Im u_1 u_2} \abs{\Im z_1\ov{z_2}} \abs{\Re m_1 m_2} \lesssim \epsilon^2 \Bigl(\abs{\Re m_1 m_2} + \abs{\Im u_1 u_2}\abs{\Im z_1\ov{z_2}}\Bigr).\]
    Together with the fact that \(\abs{\Im u_1 u_2}\abs{\Im z_1\ov{z_2}}\lesssim\abs{m_i}^2\sim\abs{\Re m_1 m_2}\) from~\eqref{m1 pm m2},~\eqref{eq u approx cusp}, this implies \(\abs{\Im u_1 u_2}\abs{\Im z_1\ov{z_2}}\lesssim \epsilon^2\). Finally, the \(\epsilon^2\)-saturation of~\eqref{eq:5}--\eqref{eq:6} shows that~\eqref{eq z u relations zsmall} (with \(\abs{z_1}\sim\abs{z_2}\sim 1\)) also holds in case~\ref{case5} and we are able to conclude~\eqref{eq z1 z2 claim} just like in case~\ref{case1}.
\end{proof}

\section{Asymptotic independence of resolvents: Proof of Proposition~\ref{prop:indmr}}\label{sec:IND}
For any fixed \(z\in\C\) let \(H^z\) be defined in~\eqref{eq:herher}. Recall that we denote the eigenvalues of \(H^z\) by \(\{\lambda_{\pm i}^z\}_{i\in [n]}\), with \(\lambda_{-i}^z=-\lambda_i^z\), and by \(\{{\bm w}_{\pm i}^z\}_{i\in [n]}\) their corresponding orthonormal eigenvectors. As a consequence of the symmetry of the spectrum of \(H^z\) with respect to zero, its eigenvectors are of the form \({\bm w}_{\pm i}^z=({\bm u}_i^z,\pm {\bm v}_i^z)\), for any \(i\in [n]\). The eigenvectors of \(H^z\) are not well defined if \(H^z\) has multiple eigenvalues. This minor inconvenience can be easily solved by a tiny Gaussian regularization (see~\eqref{eq:gpert} and Remark~\ref{rem:gpaaa} later).

\begin{convention}\label{rem:no0}
    We omitted the index \(i=0\) in the definition of the eigenvalues of \(H^z\). In the remainder of this section we always assume that all the indices are not zero, e.g we use the notation
    \[\sum_{i=-n}^n := \sum_{i=-n}^{-1}+ \sum_{i=1}^n,\]
    and we use \(\abs{i}\le A\), for some \(A>0\), to denote \(0<\abs{i}\le A\), etc.
\end{convention}

The main result of this section is the proof of Proposition~\ref{prop:indmr} which follows by Proposition~\ref{prop:indeig} and the local law in Theorem~\ref{theo:Gll}.

\begin{proposition}[Asymptotic independence of small eigenvalues of \(H^{z_l}\)]\label{prop:indeig}
    Fix \(p\in \N\), and let \(\{\lambda_{\pm i}^{z_l}\}_{i=1}^n\) be the eigenvalues of \(H^{z_l}\), with \(l\in [p]\). For any \(\omega_d, \omega_h, \omega_f>0\) sufficiently small constants such that \(\omega_h\ll \omega_f\ll \omega_d\ll 1\), there exist constants \(\omega\), \(\widehat{\omega}, \delta_0,\delta_1>0\), with \(\omega_h\ll \delta_m\ll \widehat{\omega}\ll \omega\ll\omega_f\), for \(m=0,1\), such that for any fixed \(z_1,\dots,z_p\in \C \) so that \(\abs{z_l}\le 1-n^{-\omega_h}\), \(\abs{z_l-z_m},\abs{z_l-\overline{z}_m}, \abs{z_l-\overline{z}_l}\ge n^{-\omega_d}\), with \(l,m\in [p]\), \(l\ne m\), it follows that
    \begin{equation}\label{eq:indA}
        \begin{split}
            \E  \prod_{l=1}^p \frac{1}{n}&\sum_{\abs{i_l}\le n^{\widehat{\omega}}} \frac{\eta_l}{(\lambda_{i_l}^{z_l})^2+\eta_l^2}=\prod_{l=1}^p\E   \frac{1}{n}\sum_{\abs{i_l}\le n^{\widehat{\omega}}} \frac{\eta_l}{(\lambda_{i_l}^{z_l})^2+\eta_l^2} \\
            &\qquad\quad+\mathcal{O}\left(\frac{n^{\widehat{\omega}}}{n^{1+\omega}} \sum_{l=1}^p\frac{1}{\eta_l}\times\prod_{m=1}^p\left( 1+\frac{n^\xi}{n\eta_m}\right)+\frac{n^{p\xi+2\delta_0} n^{\omega_f}}{n^{3/2}}\sum_{l=1}^p \frac{1}{\eta_l}+\frac{n^{p\delta_0+\delta_1}}{n^{\widehat{\omega}}}\right),
        \end{split}
    \end{equation}
    for any \(\xi>0\), where \(\eta_1,\dots, \eta_p\in [n^{-1-\delta_0},n^{-1+\delta_1}]\) and the implicit constant in \(\mathcal{O}(\cdot)\) may depend on \(p\).
\end{proposition}

\begin{proof}[Proof of Proposition~\ref{prop:indmr}]
    Let \(\rho^{z_l}\) be the self consistent density of states of \(H^{z_l}\), and define its quantiles \(\gamma_i^{z_l}\) by
    \[
        \frac{i}{n}=\int_0^{\gamma_i^{z_l}} \rho^{z_l}(x) \dif x, \qquad i\in [n],
    \]
    and \(\gamma_{-i}^{z_l}=-\gamma_i^{z_l}\) for \(i\in [n]\). Then, using the local law in Theorem~\ref{theo:Gll}, by standard application of Helffer-Sj\"ostrand formula (see e.g.~\cite[Lemma 7.1, Theorem 7.6]{MR3068390} or~\cite[Section 5]{MR2871147} for a detailed derivation), we conclude the following rigidity bound
    \begin{equation}
        \label{eq:hoplastrig}
        \abs{\lambda_i^{z_l}-\gamma_i^{z_l}}\le \frac{n^{100\omega_h}}{n}, \qquad \abs{i}\le n^{1-10\omega_h},
    \end{equation}
    with very high probability, uniformly in \(\abs{z_l}\le 1-n^{-\omega_h}\). Then Proposition~\ref{prop:indmr} follows by Proposition~\ref{prop:indeig} and~\eqref{eq:hoplastrig} exactly as in~\cite[Section~\ref{cplx-sec:ririri}]{1912.04100}. We remark that in the current case we additionally require that \(\abs{z_l-\ov{z}_m}\), \(\abs{z_l-\ov{z}_l}\gtrsim n^{-\omega_d}\) compared to~\cite[Proposition~\ref{cplx-prop:indeig}]{1912.04100}, but this does not cause any change in the proof in~\cite[Section~\ref{cplx-sec:ririri}]{1912.04100}.
\end{proof}

Section~\ref{sec:IND} is divided as follows: in Section~\ref{sec:PO} we state the main technical results needed to prove Proposition~\ref{prop:indeig} and conclude its proof.
In Section~\ref{sec:un} we prove Theorem~\ref{theo:un}, which will follow by the results stated in Section~\ref{sec:PO}.
In Section~\ref{sec:lamblamb} we estimate the overlaps of eigenvectors, corresponding to small indices, of \(H^{z_l}\), \(H^{z_m}\) for \(l\ne m\); this is the main input to prove the asymptotic independence in Proposition~\ref{prop:indeig}.
In Section~\ref{sec:newse} and Section~\ref{sec:wpe} we prove several technical results stated in Section~\ref{sec:PO}.
In Section~\ref{sec:hos} we present Proposition~\ref{cplx-pro:ciala} which is a modification of the path-wise coupling of DBMs close to zero from~\cite[Proposition~\ref{cplx-pro:ciala}]{1912.04100} to the case when the driving martingales in the DBM have a small correlation.
This is needed to deal with the (small) correlation of \({\bm \lambda}^{z_l}\), the eigenvalues of \(H^{z_l}\), for different \(l\)'s.

\subsection{Overview of the proof of Proposition~\ref{prop:indeig}}\label{sec:PO}

The main result of this section is the proof Proposition~\ref{prop:indeig}, which is essentially about the asymptotic independence of the eigenvalues \(\lambda_i^{z_l}\), \(\lambda_j^{z_m}\), for \(l\ne m\) and small indices \(i\) and \(j\). We do not prove this feature directly, instead we will compare \(\lambda_i^{z_l}\), \(\lambda_j^{z_m}\) with similar eigenvalues \smash{\(\mu_i^{(l)}\), \(\mu_j^{(m)}\)} coming from independent Ginibre matrices, for which independence is straightforward by construction. The comparison is done by exploiting the strong local equilibration of the Dyson Brownian motion (DBM) in several steps. For convenience, we record the sequence of approximations in Figure~\ref{fig:processes}. We remark that \(z_1,\dots, z_p\) are fixed as in Proposition~\ref{prop:indeig} throughout this section.

First, via a standard Green's function comparison argument (GFT) in Lemma~\ref{lem:GFTGFT} we prove that we may replace \(X\) by an i.i.d.\ matrix with a small Gaussian component. In the next step we make use of this Gaussian component and interpret the eigenvalues \(\bm\lambda^z\) of \(H^z\) as the short-time evolution \(\bm\lambda^z(t)\) of the eigenvalues of an auxiliary matrix \(H_t^z\) according to the Dyson Brownian motion. Proposition~\ref{prop:indeig} is thus reduced to proving asymptotic independence of the flows \({\bm \lambda}^{z_l}(t)\) for different \(l\in [p]\) after a short time \(t=t_f\), a bit bigger than \(n^{-1}\).
The corresponding DBM describing the eigenvalues of \(H_t^z\) (see~\eqref{eq:impnewDBM} later) differs from the standard DBM in two related aspects: (i) the driving martingales are weakly correlated, (ii) the interaction term has a coefficient slightly deviating from one. Note that the stochastic driving terms \(b_i\) in~\eqref{eq:impnewDBM} are martingales but not Brownian motions (see Appendix~\ref{sec:derdbm} for more details). Both effects come from the small but non-trivial overlap of the eigenvectors \( {\bm w}_i^{z_l}\) with \( \overline{\bm w}_j^{z_l}\). They also influence the well-posedness of the DBM, so an extra care is necessary. We therefore define two comparison processes. First we regularise the DBM by (i) setting the coefficient of the interaction equal to one, (ii) slightly reducing the diffusion term, and (iii) cutting off the possible large values of the correlation. The resulting process, denoted by \(\mathring{\bm \lambda}(t)\) (see~\eqref{eq:nuproc} later), will be called the \emph{regularised DBM}. Second, we artificially remove the correlation in the driving martingales for large indices. This \emph{partially correlated DBM}, defined in~\eqref{eq:nuproc3} below, will be denoted by \(\widetilde{\bm \lambda}(t)\). We will show that in both steps the error is much smaller than the relevant scale \(1/n\). After these preparations, we can directly compare the \emph{partially correlated DBM} \(\widetilde{\bm \lambda}(t)\) with its Ginibre counterpart \(\widetilde{\bm \mu}(t)\) (see~\eqref{eq:nuproc2} later) since their distribution is the same. Finally, we remove the partial correlation in the process \(\widetilde{\bm \mu}(t)\) by comparing it with a purely independent Ginibre DBM \({\bm \mu}(t)\), defined in~\eqref{eq:ginev} below.

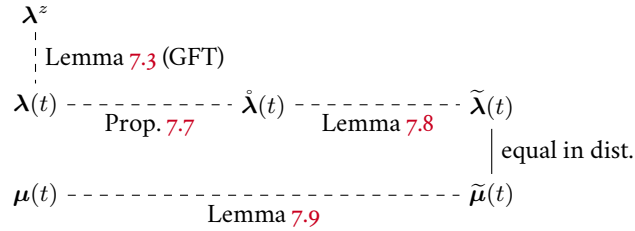
\begin{figure}[h]
    \centering
    \begin{tikzpicture}[node distance = 1.2cm and 3cm,on grid]
        \begin{scope}
            \node (a) at (0,2) {\(\bm\lambda^{z}\)};
            \node (b1) [below= of a] {\(\bm\lambda(t)\)};
            \node (b2) [right= of b1] {\(\mathring{\bm\lambda}(t)\)};
            \node (b3) [right= of b2] {\(\wt{\bm\lambda}(t)\)};
            \node (c2) [below= of b3] {\(\wt{\bm\mu}(t)\)};
            \node (c1) [below= of b1] {\(\bm\mu(t)\)};
            \path (a) edge[dashed] node[right] {Lemma~\ref{lem:GFTGFT} (GFT)} (b1);
            \path (b1) edge[dashed] node[below] {Prop.~\ref{prop:realprop}} (b2);
            \path (b2) edge[dashed] node[below] {Lemma~\ref{lem:firststepmason}} (b3);
            \path (b3) edge node[right] {equal in dist.} (c2);
            \path (c1) edge[dashed] node[below] {Lemma~\ref{lem:secondstepmason}} (c2);
        \end{scope}
    \end{tikzpicture}
    \caption{Proof overview for Proposition~\ref{prop:indeig}: The collections of eigenvalues \(\bm\lambda^{z_l}\) of \(H^{z_l}\) for different \(l\)'s are approximated by several stochastic processes. The processes \({\bm \mu}=\bm\mu^{(l)}\) are independent for different \(l\)'s by definition.}\label{fig:processes}
\end{figure}

Now we define these processes precisely. From now on we assume that \(p=2\) in Proposition~\ref{prop:indeig} to make our presentation clearer. The case \(p\ge 3\) is completely analogous. Consider the Ornstein-Uhlenbeck (OU) flow
\begin{equation}
    \label{eq:OUflow}
    d\widehat{X}_t=-\frac{1}{2}\widehat{X}_t \dif t+\frac{\dif \widehat{B}_t}{\sqrt{n}}, \qquad \widehat{X}_0=X,
\end{equation}
for a time
\begin{equation}
    \label{eq:time1}
    t_f:= \frac{n^{\omega_f}}{n},
\end{equation}
with some small exponent \(\omega_f>0\) given as in Proposition~\ref{prop:indeig}, in order to add a small Gaussian component to \(X\). Here \(\widehat{B}_t\) in~\eqref{eq:OUflow} is a standard matrix valued real Brownian motion, i.e.\ \( \widehat{B}_{ab}\), \(a,b\in[n]\) are i.i.d.\ standard real Brownian motions, independent of \(\widehat{X}_0\). Then we can construct an i.i.d.\ matrix \(\widecheck{X}_{t_f}\) such that
\begin{equation}
    \label{eq:consOU}
    \widehat{X}_{t_f}\stackrel{\mathsf{d}}{=}\widecheck{X}_{t_f}+\sqrt{ct_f} U,
\end{equation}
for some explicit constant \(c>0\) very close to \(1\), and \(U\) is a real Ginibre matrix independent of \(\widecheck{X}_{t_f}\). Using a simple Green's function comparison argument (GFT), by~\cite[Lemma~\ref{cplx-lem:GFTGFT}]{1912.04100}, we conclude the following lemma.
\begin{lemma}\label{lem:GFTGFT}
    The eigenvalues of \(H^{z_l}\) and the eigenvalues of \(\wh H^{z_l}_{t_f}\), with \(t_f=n^{-1+\omega_f}\) obtained from replacing \(X\) by \(\wh X_{t_f}\), are close in the sense that for any sufficiently small \(\omega_f,\delta_0,\delta_1>0\) it holds
    \begin{equation}
        \label{eq:stgft2}
        \E  \prod_{l=1}^2\frac{1}{n}\sum_{\abs{i_l}\le n} \frac{\eta_l}{(\lambda_{i_l}(H^{z_l}))^2+\eta_l^2}=\E  \prod_{l=1}^2\frac{1}{n}\sum_{\abs{i_l}\le n} \frac{\eta_l}{(\lambda_{i_l}(\wh H^{z_l}_{t_f}))^2+\eta_l^2}+\mathcal{O}\left(\frac{n^{2\xi+2\delta_0} t_f}{n^{1/2}}\sum_{l=1}^2 \frac{1}{\eta_l}\right),
    \end{equation}
    where \(\eta_l\in [n^{-1-\delta_0},n^{-1+\delta_1}]\).
\end{lemma}

Next, we consider the matrix flow
\begin{equation}
    \label{eq:bigDBM}
    \dif X_t=\frac{\dif B_t}{\sqrt{n}}, \qquad X_0=\widecheck{X}_{t_f},
\end{equation}
and denote by \(H_t^z\) the Hermitisation of \(X_t-z\). Here \(B_t\) is a real standard matrix valued Brownian motion independent of \(X_0\) and \(\widehat{B}_t\). Note that by construction \(X_{ct_f}\) is such that
\begin{equation}
    \label{eq:impGFT}
    X_{ct_f}\stackrel{\mathsf{d}}{=}\widehat{X}_{t_f}.
\end{equation}

Denote the eigenvalues and eigenvectors of \(H_t^z\) by
\[{\bm \lambda}^z(t)=\set*{\lambda_{\pm i}^z(t)\given i\in[n]},\quad  \set*{{\bm w}_{\pm i}^z(t) \given i\in[n]}=\set*{({\bm u}_i^z(t),\pm {\bm v}_i^z(t)) \given i\in[n]},\]
and the resolvent by \(G_t^z(w):=(H_t^z-w)^{-1}\) for \(w\in\HC\). For any \(\bm{w}=({\bm u}, {\bm v})\), with \({\bm u},{\bm v}\in \C^n\) define the projections \(P_1,P_2\colon\C^{2n}\to\C^n\) by
\begin{equation}
    \label{eq:defpro}
    P_1{\bm w}= {\bm u}, \qquad P_2{\bm w} ={\bm v},
\end{equation}
and, for any \(z,z'\in \C\), define the \emph{eigenvector overlaps} by
\begin{equation}
    \label{eq:defbiglam}
    \Theta_{ij}^{z,z'}=\Theta_{ij}^{z,z'}(t):=  4\Re[\braket{ P_1 {\bm w}_j^{z'}(t),P_1{\bm w}_i^{z}(t)}\braket{ P_2 {\bm w}_i^{z}(t),P_2 {\bm w}_j^{z'}(t)}], \quad \abs{i},\abs{j}\le n.
\end{equation}
Note that by the spectral symmetry of \(H_t^z\) it holds
\begin{equation}
    \label{eq:sym}
    \Theta_{ij}^{z,z}=\delta_{i,j}-\delta_{i,-j}, \quad \Theta_{ij}^{z,z'}=\Theta_{ji}^{z',z}, \quad \abs{\Theta_{ij}^{z,z'}}\le 1,
\end{equation}
for any \(\abs{i},\abs{j}\le n\). The coefficients \(\Theta_{ij}^{z,z'}(t)\) are small with high probability due to the following lemma whose proof is postponed to Section~\ref{sec:lamblamb}.
\begin{lemma}[Eigenvectors overlaps are small]\label{lem:lambound}
    For any sufficiently small constants \(\omega_h, \omega_d>0\), there exists \(\omega_E>0\) so that for any \(z,z'\in \C\) such that \(\abs{z},\abs{z'}\le 1-n^{-\omega_h}\), \(\abs{z-z'}\ge n^{-\omega_d}\), we have
    \begin{equation}
        \label{eq:lab1}
        \sup_{0\le t\le T}\sup_{\abs{i},\abs{j}\le n} \abs*{\Theta_{ij}^{z,z'}(t)}\le n^{-\omega_E},
    \end{equation}
    with very high probability for any fixed \(T\ge 0\).
\end{lemma}

Most of the DBM analysis is performed for a fixed \(z\in \{z_1,z_2\}\), with \(z_1,z_2\) as in Proposition~\ref{prop:indeig}, for this purpose we introduce the notation
\begin{equation}
    \label{eq:deflamzz}
    \Lambda_{ij}^z(t):= \Theta_{ij}^{z,\overline{z}}(t),
\end{equation}
for any \(\abs{i},\abs{j}\le n\). In particular, note \(\Theta_{ij}^{z,\overline{z}}=\Theta_{ij}^{\overline{z},z}\) and so that by~\eqref{eq:sym} it follows that \(\Lambda_{ij}^z(t)=\Lambda_{ji}^z(t)\).

By the derivation of the DBM in Appendix~\ref{sec:derdbm}, using the fact that \(\overline{\bm w}^z={\bm w}^{\overline{z}}\), for \(z=z_l\) with \(l\in [2]\), it follows that~\eqref{eq:bigDBM} induces the flow
\begin{equation}
    \label{eq:impnewDBM}
    \dif \lambda_i^z(t)=\frac{\dif b_i^z}{\sqrt{n}}+\frac{1}{2n}\sum_{j\ne i} \frac{1+\Lambda_{ij}^z(t)}{\lambda_i^z(t)-\lambda_j^z(t)}\dif t, \quad  \lambda_i^z(0)=\lambda_i^z, \qquad \abs{i}\le n,
\end{equation}
on the eigenvalues \(\{\lambda_{i}^z(t)\}_{\abs{i}\le n}\) of \(H_t^z\). Here \(\{\lambda_i^z\}_{\abs{i}\le n}\) are the eigenvalues of the initial matrix \(H^z\). The martingales \(\{b_i^z\}_{i\in [n]}\), with \(b_i^z(0)=0\), and $\Lambda_{ij}^z(t)$, the overlap of eigenvectors in~\eqref{eq:deflamzz},~\eqref{eq:defbiglam}, are defined on a probability space \(\Omega_b\) equipped with the filtration
\begin{equation}
    \label{eq:deffiltr}
    (\mathcal{F}_{b,t})_{0\le t\le T}:=\bigl(\sigma(X_0, (B_s)_{0\le s\le t})\bigr)_{0\le t\le T},
\end{equation}
where \(B_s\) is defined in~\eqref{eq:bigDBM}. The martingale differentials in~\eqref{eq:impnewDBM} are such that (see~\eqref{eq:defBBB}--\eqref{eq:halcov})
\begin{equation}
    \label{eq:recallBBB}
    \begin{split}
        \dif b^z_i:={}&\dif B^z_{ii}+\dif \overline{B^z_{ii}},\quad \text{with} \quad  \dif B^z_{ij}:=\braket{\bm u^z_i,(\dif B) \bm v^z_i}, \quad i,j\in [n], \\
        \Exp*{ \dif b_i^z \dif b_j^z\given \mathcal{F}_{b,t}}={}&\frac{\delta_{ij}+\Lambda_{ij}^z(t)}{2} \dif t, \quad i,j\in [n],
    \end{split}
\end{equation}
and \(\dif b_{-i}^z=-\dif b_i^z\) for \(i\in [n]\). Here we used the notation \(\Omega_b\) for the probability space to emphasize that is the space where the martingales \({\bm b}^z\) are defined, since in Section~\ref{sec:COMPPRO} we will introduce another probability space which we will denote by \(\Omega_\beta\).

In the remainder of this section we will apply Lemma~\ref{lem:lambound} for \(z=z_1, z'=z_2\) and \(z=z_1, z'=\overline{z}_2\) and \(z=z_l, z'=\overline{z}_l\), for \(l\in [2]\), with \(z_1,z_2\) fixed as in Proposition~\ref{prop:indeig}. We recall that throughout this section we assumed that \(p=2\) in Proposition~\ref{prop:indeig}. Note that \(\Lambda_{ij}^{z_1}\), \(\Lambda_{ij}^{z_2}\), \(\Theta_{ij}^{z_1,z_2}\), \(\Theta_{ij}^{z_1,\overline{z}_2}\) with \(\abs{i},\abs{j}\le n\), are not well-defined if \(H_t^{z_1}\), \(H_t^{z_2}\) have multiple eigenvalues. This minor inconvenience can easily be resolved by a tiny regularization
as in~\cite[Lemma 6.2]{1908.04060} (which is the singular values counterpart of~\cite[Proposition 2.3]{MR4009717}).
Using this result,  we may,  without loss of generality, assume that the eigenvalues of \(H_t^{z_l}\) are almost surely distinct for any fixed time \(t\ge 0\). Indeed, if this were not the case then we replace \(H_0^{z_l}\) by
\begin{equation}
    \label{eq:gpert}
    H_{0, \mathrm{reg}}^{z_l}:=\left( \begin{matrix}
            0                                  & X-z_l+e^{-n}Q \\
            X^\ast-\overline{z}_l+e^{-n}Q^\ast & 0
        \end{matrix} \right),
\end{equation}
with \(Q\) being a complex \(n\times n\) Ginibre matrix independent of \(X\), i.e.\ we may regularize \(X\) by adding an exponentially small Gaussian component. Then, by~\cite[Lemma 6.2]{1908.04060}, \(H_{t,\mathrm{reg}}^{z_l}\), the evolution of \(H_{0, \mathrm{reg}}^{z_l}\) along the flow~\eqref{eq:bigDBM}, does not have multiple eigenvalues almost surely; additionally, the eigenvalues of \(H_{0, \mathrm{reg}}^{z_l}\) and the ones of \(H_0^{z_l}\) are exponentially close. Hence, by Fubini's theorem, \(\{\Lambda_{ij}^{z_l}(t)\}_{\abs{i},\abs{j}\le n}\), with \(l\in [2]\), and \(\{\Theta_{ij}^{z_1,z_2}(t)\}_{\abs{i},\abs{j}\le n}\), \(\{\Theta_{ij}^{z_1,\overline{z}_2}(t)\}_{\abs{i},\abs{j}\le n}\) are well-defined for almost all \(t\ge 0\); we set them equal to zero whenever they are not well defined.

\begin{remark}\label{rem:gpaaa}
    The perturbation of \(X\) in~\eqref{eq:gpert} is exponentially small, hence does not change anything in the proof of the local laws in Theorem~\ref{theo:Gll} and Theorem~\ref{thm local law G2} or in the Green's function comparison (GFT) argument in Lemma~\ref{lem:GFTGFT}, since these proofs deal with scales much bigger than \(e^{-n}\). This implies that any local law or GFT result which holds for \(H_t^{z_l}\) then holds true for \(H_{t, \mathrm{reg}}^{z_l}\) as well. Hence, in the remainder of this section we assume that~\cite[Lemma 6.2]{1908.04060} holds true for \(H_t^{z_l}\) (the unperturbed matrix).
\end{remark}

The process~\eqref{eq:impnewDBM} is well-defined in the sense of Proposition~\ref{prop:wpmainDBM}, whose proof is postponed to Section~\ref{sec:wpe}.

\begin{proposition}[The DBM in~\eqref{eq:impnewDBM} is well-posed]\label{prop:wpmainDBM}
    Fix \(z\in \{z_1,z_2\}\), and let \(H_t^z\) be defined by the flow~\eqref{eq:bigDBM}. Then the eigenvalues \({\bm \lambda}(t)\) of \(H_t^z\) are the unique strong solution to~\eqref{eq:impnewDBM} on \([0,T]\), for any \(T>0\), such that
    \({\bm \lambda}(t)\) is adapted to the filtration \((\mathcal{F}_{b,t})_{0\le t\le T}\),  \({\bm \lambda}(t)\) is \(\gamma\)-H\"older continuous for any \(\gamma\in (0,1/2)\), and
    \[\Prob\Bigl(\lambda_{-n}(t)<\cdots < \lambda_{-1}(t)<0<\lambda_1(t)< \cdots < \lambda_{n}(t), \text{ for almost all } t\in [0,T]\Bigr)=1.\]
\end{proposition}

In order to prove that the term \(\Lambda_{ij}^z\) in~\eqref{eq:impnewDBM} is irrelevant, we will couple the driving martingales in~\eqref{eq:impnewDBM} with the ones of a DBM that does not have the additional term \(\Lambda_{ij}^z\) (see~\eqref{eq:nuproc} below). For this purpose we have to consider the correlation of \(\{b_i^{z_1}\}_{\abs{i}\le n}\), \(\{b_i^{z_2}\}_{\abs{i}\le n}\) for two different \(z_1,z_2\in \C\) as in Proposition~\ref{prop:indeig}. In the following we will focus only on the driving martingales with positive indices, since the ones with negative indices are defined by symmetry. The martingales \({\bm b}^{z_l}=\{b_i^{z_l}\}_{i\in [n]}\), with \(l=1,2\), are defined on a common probability space equipped with the filtration \((\mathcal{F}_{b,t})_{0\le t\le T}\) from~\eqref{eq:deffiltr}.

We consider \({\bm b}^{z_1}\), \({\bm b}^{z_2}\) jointly as a \(2n\)-dimensional martingale \(({\bm b}^{z_1},{\bm b}^{z_2})\). Define the naturally reordered indices
\[\mathfrak{i}=(l-1)n+i, \qquad \mathfrak{j}=(m-1)n+j,\]
with \(l,m\in [2]\), \(i,j\in [n]\), and \(\mathfrak{i}, \mathfrak{j}\in [2n]\). Then the correlation between \({\bm b}^{z_1}\), \({\bm b}^{z_2}\) is given by
\begin{equation}\label{eq:defCC}
    C_{\mathfrak{i}\mathfrak{j}}(t)\dif t:= \Exp*{\dif b_i^{z_l}\dif  b_j^{z_m}\given \mathcal{F}_{b,t}}=\frac{\Theta_{ij}^{z_l,z_m}(t)+\Theta_{ij}^{z_l,\overline{z}_m}(t)}{2} \dif t \qquad \mathfrak{i},\mathfrak{j} \in [2n].
\end{equation}
Note that \(C(t)\) is a positive semi-definite matrix. In particular, taking also negative indices into account, for a fixed \(z\in \{z_1,z_2\}\), the family of martingales \({\bm b}^z=\{b_i^z\}_{\abs{i}\le n}\) is such that
\begin{equation}\label{eq:defcovz}
    \Exp*{\dif b_i^z \dif b_j^z\given  \mathcal{F}_{b,t}}=\frac{\delta_{i,j}-\delta_{i,-j}+\Lambda_{ij}^z(t)}{2} \dif t, \qquad \abs{i},\abs{j}\le n.
\end{equation}

\subsubsection{Comparison of \texorpdfstring{\(\bm{\lambda}\)}{l} with the regularised process \texorpdfstring{\(\mathring{\bm{\lambda}}\)}{l}}\label{sec:remlamb}
By Lemma~\ref{lem:lambound} the overlaps  \(\Theta_{ij}^{z,z'}\) are typically small for any \(z,z'\in\C\) such that \(\abs{z},\abs{z'}\le 1-n^{-\omega_h}\) and \(\abs{z-z'}\ge n^{-\omega_d}\). We now define their cut-off versions (see~\eqref{eq:ringlamthe} below). We only consider positive indices, since negative indices are defined by symmetry. Throughout this section we use the convention that regularised objects will be denoted by circles. Let \(z_l\), with \(l\in [2]\) be fixed throughout Section~\ref{sec:IND} as in Proposition~\ref{prop:indeig}. Define the \(2n\times 2n\) matrix \(\mathring{C}(t)\) by
\begin{equation}\label{eq:ringmat}
    \mathring{C}_{\mathfrak{i}\mathfrak{j}}(t):=\frac{\mathring{\Theta}_{ij}^{z_l,z_m}(t)+\mathring{\Theta}_{ij}^{z_l,\overline{z}_m}(t)}{2} \qquad i,j \in [n], \quad \mathfrak{i},\mathfrak{j}\in [2n],
\end{equation}
where \(\mathring{\Theta}_{ij}^{z_l,z_l}=\delta_{ij}\) for \(i,j \in [n]\), and
\begin{equation}\label{eq:ringlamthe}
    \begin{split}
        \mathring{\Theta}_{ij}^{z_1,z_2}(t):&=\Theta_{ij}^{z_1,z_2}(t) \cdot \bm1\left(\mathcal{A}(t)\le n^{-\omega_E}\right), \quad \mathring{\Theta}_{ij}^{z_l,\overline{z}_m}(t):=\Theta_{ij}^{z_l,\overline{z}_m}(t) \cdot \bm1\left(\mathcal{A}(t)\le n^{-\omega_E}\right), \\
        \mathcal{A}(t)&=\mathcal{A}^{z_1,z_2}(t):=\max_{\abs{i},\abs{j}\le n}\abs{\Lambda_{ij}^{z_1}(t)}+\abs{\Lambda_{ij}^{z_2}(t)}+\abs{\Theta_{ij}^{z_1,\overline{z}_2}(t)}+\abs{\Theta_{ij}^{z_1,z_2}(t)}
    \end{split}
\end{equation}
for any \(l,m\in [2]\), recalling that \(\Lambda_{ij}^{z_l}=\Theta_{ij}^{z_l,\overline{z}_l}\). Note that by Lemma~\ref{lem:lambound} it follows that \(\mathring{C}(t)=C(t)\) on a set of very high probability, and \(\mathring{C}(t)=\frac{1}{2}I\), with \(I\) the \(2n\times 2n\) identity matrix, on the complement of this set, for any \(t\in [0,T]\). In particular, \(\mathring{C}(t)\) is positive semi-definite for any \(t\in [0,T]\), since \(C(t)\), defined as a covariance in~\eqref{eq:defCC}, is positive semi-definite. The purpose of the cut-off in~\eqref{eq:ringmat} it is to ensure the well-posedness of the process~\eqref{eq:nuproc} below.

We compare the processes \({\bm \lambda}^{z_l}(t)\) in~\eqref{eq:impnewDBM} with the \emph{regularised processes} \(\mathring{\bm \lambda}^{z_l}(t)\) defined, for \(z=z_l\), by
\begin{equation}\label{eq:nuproc}
    \dif \mathring{\lambda}_i^z=\frac{\dif \mathring{b}_i^z }{\sqrt{n(1+n^{-\omega_r})}}+\frac{1}{2n}\sum_{j\ne i} \frac{1}{\mathring{\lambda}_i^z-\mathring{\lambda}_j^z}\dif t, \qquad \mathring{\lambda}_i^z(0)=\lambda_i^z(0), \qquad \abs{i}\le n,
\end{equation}
with \(\omega_r>0\) such that \(\omega_f\ll \omega_r\ll \omega_E\). We organise the martingales \({\bm b}^{z_1},\bm b^{z_2}\) with positive indices into a single \(2n\)-dimensional vector \({\bm b}=({\bm b}^{z_1},{\bm b}^{z_2})\) with a correlation structure given by~\eqref{eq:defCC}. Then by Doob's martingale representation theorem~\cite[Theorem~18.12]{MR1876169} there exists a standard Brownian motion \(\mathfrak{w}=(\mathfrak{w}^{(1)}, \mathfrak{w}^{(2)})\in \R^{2n}\) realized on an extension \((\widetilde{\Omega}_b,\widetilde{\mathcal{F}}_{b,t})\) of the original  probability space \((\Omega_b,\mathcal{F}_{b,t})\) such that \(\dif {\bm b}=\sqrt{C}\dif \mathfrak{w}\), with \(\sqrt{C}=\sqrt{C(t)}\) the matrix square root of \(C(t)\). Moreover, \( \mathfrak{w}(t)\) and \(C(t)\) are adapted to the filtration \(\widetilde{\mathcal{F}}_{b,t}\). Then the martingales \(\mathring{{\bm b}}^{z_l}=\{\mathring{b}_i^{z_l}\}_{i\in [n]}\), with \(l\in [2]\), are defined by \(\mathring{{\bm b}}^{z_l}(0)=0\) and
\begin{equation}
    \label{eq:circbm}
    \begin{pmatrix}
        \dif \mathring{{\bm b}}^{z_1}(t) \\
        \dif \mathring{{\bm b}}^{z_2}(t)
    \end{pmatrix} := \sqrt{\mathring{C}(t)}\begin{pmatrix}
        \dif \mathfrak{w}^{(1)}(t) \\
        \dif \mathfrak{w}^{(2)}(t)
    \end{pmatrix},
\end{equation}
where \smash{\(\sqrt{\mathring C(t)}\)} denotes the matrix square root of the positive semi-definite matrix \(\mathring{C}(t)\). For negative indices we define \(\mathring{b}_{-i}=-\mathring{b}_i\), with \(i\in [n]\). The purpose of the additional factor \(1+n^{-\omega_r}\) in~\eqref{eq:nuproc} is to ensure the well-posedness of the process, since \(\mathring{\bm b}^z\) is a small deformation of a family of i.i.d.\ Brownian motions with variance \(1/2\), and the well-posedness of~\eqref{eq:nuproc} is already critical for those Brownian motions (it corresponds to the GOE case, i.e.\ \(\beta=1\)). The well-posedness of the process~\eqref{eq:nuproc} is proven in Appendix~\ref{sec:wp}. The main result of this section is the following proposition, whose proof is deferred to Section~\ref{sec:newse}.

\begin{proposition}[The \emph{regularised process} \(\mathring{\bm \lambda}\) is close to \({\bm \lambda}\)]\label{prop:realprop}
    For any sufficiently small \(\omega_d,\omega_h, \omega_f>0\) such that \(\omega_h\ll\omega_f\ll 1\) there exist small constants \(\widehat{\omega}, \omega>0\) such that \(\omega_h \ll \widehat{\omega}\ll \omega\ll \omega_f\), and that for \(\abs{z_l-\overline{z}_l},\abs{z_l-\overline{z}_m},\abs{z_l-z_m}\ge n^{-\omega_d}\), \(\abs{z_l}\le 1-n^{-\omega_h}\), with \(l\ne m\), it holds
    \[ \abs{\lambda_i^{z_l}(ct_f)-\mathring{\lambda}_i^{z_l}(ct_f)}\le n^{-1-\omega}, \qquad \abs{i}\le n^{\widehat{\omega}}, \]
    with very high probability, where \(t_f=n^{-1+\omega_f}\) and \(c>0\) is defined in~\eqref{eq:consOU}.
\end{proposition}

\subsubsection{Definition of the partially correlated processes \texorpdfstring{\(\widetilde{\bm \lambda}\), \(\widetilde{\bm \mu}\)}{lu}}\label{sec:COMPPRO}
The construction of the \emph{partially correlated processes} for \(\mathring{\bm \lambda}^{z_l}(t)\) is exactly the same as in the complex case~\cite[Section~\ref{cplx-sec:COMPPRO}]{1912.04100}; we present it here as well for completeness. We want to compare the correlated processes \(\mathring{\bm \lambda}^{z_l}(t)\), with \(l=1,2\), defined on a probability space \(\widetilde{\Omega}_b\) equipped with a filtration \(\widetilde{\mathcal{F}}_{b,t}\) with carefully constructed independent processes \({\bm \mu}^{(l)}(t)\), \(l=1,2\) on a different probability space \(\Omega_\beta\) equipped with a filtration \(\mathcal{F}_{\beta,t}\), which is defined in~\eqref{eq:newfiltrbet} below. We choose \({\bm \mu}^{(l)}(t)\) to be a \emph{complex Ginibre DBM}, i.e.\ it is given as the solution of
\begin{equation}
    \label{eq:ginev}
    \dif\mu_i^{(l)}(t)=\frac{\dif \beta_i^{(l)}}{\sqrt{2n}}+\frac{1}{2n}\sum_{j\ne  i} \frac{1}{\mu_i^{(l)}(t)-\mu_j^{(l)}(t)} \dif t, \quad \mu_i^{(l)}(0)=\mu_i^{(l)}, \qquad \abs{i}\le n,
\end{equation}
with \(\mu_i^{(l)}\) the singular values, taken with positive and negative sign, of independent complex Ginibre matrices \(X^{(l)}\), and \({\bm \beta}^{(l)}=\{\beta_i^{(l)}\}_{i\in [n]}\) being independent vectors of i.i.d.\ standard real Brownian motions, and \(\beta_{-i}^{(l)}=-\beta_i^{(l)}\) for \(i\in [n]\). The filtration \(\cF_{\beta,t}\) is defined by
\begin{equation}
    \label{eq:newfiltrbet}
    \bigl(\mathcal{F}_{\beta,t}\bigr)_{0\le t\le T}:=\bigl(\sigma(X^{(l)}, ({\bm \beta}_s^{(l)})_{0\le s\le t}, (\widetilde{\bm \zeta}_s^{(l)})_{0\le s\le t};{l\in [2]})\bigr)_{0\le t\le T},
\end{equation}
with \(\widetilde{\bm \zeta}^{(l)}\) standard real i.i.d.\ Brownian motions, independent of \({\bm \beta}^{(l)}\), which will be used later in the definition of the processes in~\eqref{eq:nuproc2}.

The comparison of \(\mathring{\bm \lambda}^{z_l}(t)\) and \({\bm \mu}^{(l)}(t)\) is done via two intermediate \emph{partially correlated processes} \(\widetilde{{\bm \lambda}}^{(l)}(t)\), \( \widetilde{{\bm \mu}}^{(l)}(t)\) so that for a time \(t\ge 0\) large enough \(\widetilde{\lambda}_i^{(l)}(t)\), \(\widetilde{\mu}_i^{(l)}(t)\) for small indices \(i\) will be close to \(\mathring{\lambda}^{z_l}_i(t)\) and \(\mu_i^{(l)}(t)\), respectively, with very high probability. Additionally, the processes \(\widetilde{{\bm \lambda}}^{(l)}\), \( \widetilde{{\bm \mu}}^{(l)}\) will be constructed
such that they have the same joint distribution:
\begin{equation}
    \label{eq:nefdis}
    \left( \widetilde{{\bm \lambda}}^{(1)}(t),\widetilde{{\bm \lambda}}^{(2)}(t)\right)_{0\le t\le T}\stackrel{\mathsf{d}}{=}\left(\widetilde{{\bm \mu}}^{(1)}(t), \widetilde{{\bm \mu}}^{(2)}(t)\right)_{0\le t\le T},
\end{equation}
for any \(T>0\).

Fix \(\omega_A>0\) such that \(\omega_h\ll \omega_A\ll \omega_f\), and for \(l\in [2]\) define the process \(\widetilde{\bm \lambda}^{(l)}(t)\) to be the solution of
\begin{equation}
    \label{eq:nuproc3}
    \dif \widetilde{\lambda}^{(l)}_i(t)=\frac{1}{2n}\sum_{j\ne  i} \frac{1}{\widetilde{\lambda}^{(l)}_i(t)-\widetilde{\lambda}^{(l)}_j(t)} \dif t+\begin{cases}
        \bigl(n (1+n^{-\omega_r})\bigr)^{-1/2}\dif \mathring{b}_i^{z_l}, & \abs{i}\le n^{\omega_A}     \\
        (2n)^{-1/2}\dif \widetilde{b}_i^{(l)},                           & n^{\omega_A}< \abs{i}\le n,
    \end{cases}
\end{equation}
with initial data \(\widetilde{\bm \lambda}^{(l)}(0)\) being the singular values, taken with positive and negative sign, of independent complex Ginibre matrices \(\widetilde{Y}^{(l)}\) independent of \({\bm \lambda}^{z_l}(0)\). Here \(\dif \mathring{b}_i^{z_l}\) is the martingale differential from~\eqref{eq:nuproc} which is used for small indices in~\eqref{eq:nuproc3}. For large indices we define the driving martingales to be an independent collection \(\set{\{\widetilde{b}_i^{(l)}\}_{i=n^{\omega_A}+1}^n \given l\in [2]}\) of two vector-valued i.i.d.\ standard real Brownian motions which are also independent of \(\set{\{\mathring{b}_{\pm i}^{z_l}\}_{i=1}^n\given l\in [2]}\), and that \(\widetilde{b}_{-i}^{(l)}=-\widetilde{b}_i^{(l)}\) for \(i\in [n]\). The martingales \(\mathring{\bm b}^{z_l}\), with \(l \in [2]\), and \(\set{\{\widetilde{b}_i^{(l)}\}_{i=n^{\omega_A}+1}^n \given l\in [2]}\) are defined on a common probability space  that we continue to denote by \(\widetilde{\Omega}_b\) with the common filtration \(\widetilde{\mathcal{F}}_{b,t}\), given by
\[
    \bigl(\widetilde{\mathcal{F}}_{b,t}\bigr)_{0\le t\le T}:=\bigl(\sigma(X_0,\widetilde{Y}^{(l)}, (B_s)_{0\le s\le t}, (\widetilde{\bm b}^{(l)})_{0\le s\le t};{l\in [2]})\bigr)_{0\le t\le T}.
\]
The well-posedness of~\eqref{eq:nuproc3}, and of~\eqref{eq:nuproc2} below, readily follows by exactly the same arguments as in Appendix~\ref{sec:wp}.

Notice that \(\mathring{\bm \lambda}(t)\) and \(\widetilde{\bm \lambda}(t)\) differ in two aspects: the driving martingales with large indices for \(\widetilde{\bm \lambda}(t)\) are set to be independent, and the initial conditions are different. Lemma~\ref{lem:firststepmason} below states that these differences are negligible for our purposes (i.e.\ after time \(ct_1\) the two processes at small indices are closer than the rigidity scale \(1/n\)). Its proof is postponed to Section~\ref{sec:noncelpiu}. Let \(\rho_{\mathrm{sc}}(E) =\frac{1}{2\pi}\sqrt{4-E^2}\) denote the semicircle density.

\begin{lemma}[The partially correlated process \(\widetilde{\bm \lambda}\) is close to \(\mathring{\bm \lambda}\)]\label{lem:firststepmason}
    Let \(\mathring{\bm \lambda}^{z_l}(t)\), \(\widetilde{\bm \lambda}^{(l)}(t)\), with \(l\in [2]\), be the processes defined in~\eqref{eq:nuproc} and~\eqref{eq:nuproc3}, respectively. For any sufficiently small \(\omega_h,\omega_f>0\) such that \(\omega_h\ll \omega_f\ll 1\) there exist constants \(\omega, \widehat{\omega}>0\) such that \(\omega_h\ll \widehat{\omega}\ll \omega\ll \omega_f\), and that for \(\abs{z_l}\le 1-n^{-\omega_h}\) it holds
    \begin{equation}
        \label{eq:firshpb}
        \abs[\bigr]{\rho^{z_l}(0)\mathring{\lambda}_i^{z_l}(ct_f)-\rho_{\mathrm{sc}}(0) \widetilde{\lambda}_i^{(l)}(ct_f) }\le n^{-1-\omega}, \qquad \abs{i}\le n^{\widehat{\omega}},
    \end{equation}
    with very high probability, where \(t_f:= n^{-1+\omega_f}\) and \(c>0\) is defined in~\eqref{eq:consOU}.
\end{lemma}

Finally, \(\widetilde{{\bm \mu}}^{(l)}(t)\), the comparison process of \({\bm \mu}^{(l)}(t)\), is given as the solution of the following DBM
\begin{equation}\label{eq:nuproc2}
    \dif \widetilde{\mu}^{(l)}_i(t)=\frac{1}{2n}\sum_{j\ne  i} \frac{1}{\widetilde{\mu}^{(l)}_i(t)-\widetilde{\mu}^{(l)}_j(t)} \dif t+\begin{cases}
        \bigl(n (1+n^{-\omega_r})\bigr)^{-1/2} \dif \mathring{\zeta}_i^{z_l}, & \abs{i}\le n^{\omega_A},    \\
        (2n)^{-1/2}\dif \widetilde{\zeta}_i^{(l)},                            & n^{\omega_A}< \abs{i}\le n,
    \end{cases}
\end{equation}
with initial data \(\widetilde{\bm \mu}^{(l)}(0)={\bm \mu}^{(l)}\). We now explain how to construct the driving martingales in~\eqref{eq:nuproc2} so that~\eqref{eq:nefdis} is satisfied. For this purpose we closely follow~\cite[Eqs.~\eqref{cplx-eq:vecla}--\eqref{cplx-eq:BMsost}]{1912.04100}. We only consider positive indices, since the negative indices are defined by symmetry. Define the \(2n^{\omega_A}\)-dimensional martingale \(\underline{\mathring{\bm b}}:=\set{\{\mathring{b}_i^{z_l}\}_{i\in [n^{\omega_A}]}\given l\in [2]}\).  Throughout this section underlined vectors or matrices denote their restriction to the first \(i\in [n^{\omega_A}]\) indices within each \(l\)-group, i.e.
\[
    {\bm v}\in \C^{2n}\Longrightarrow \underline{\bm v}\in \C^{2n^{\omega_A}}, \quad \text{with} \quad \underline{v}_i:=\begin{cases}
        v_i                & \text{if} \quad i\in [n^{\omega_A}]     \\
        v_{i+n^{\omega_A}} & \text{if} \quad  i\in n+[n^{\omega_A}].
    \end{cases}
\]
Then we define \(\underline{\mathring{C}}(t)\) as the \(2n^{\omega_A}\times 2n^{\omega_A}\) positive semi-definite matrix which consists of the four blocks corresponding to index pairs \(\{ (i,j)\in [n^{\omega_A}]^2\}\) of the matrix \(\mathring{C}(t)\) defined in~\eqref{eq:ringmat}.
Similarly to~\eqref{eq:circbm}, by Doob's martingale representation theorem, we obtain \(\dif \underline{\mathring{\bm b}}=(\underline{\mathring{C}})^{1/2} \dif {\bm \theta}\) with \({\bm \theta}(t):=\set{\{\theta_i^{(l)}(t)\}_{i\in [n^{\omega_A}]}\given l\in [2]}\) a family of i.i.d.\ standard real Brownian motions. We define an independent copy \(\underline{\mathring{C}}^\#(s)\) of \(\underline{\mathring{C}}(s)\) and \(\underline{\bm \beta}:=\set{\{\beta_i^{(l)}\}_{i\in [n^{\omega_A}]}\given l\in [2]}\) such that \((\underline{\mathring{C}}^\#(t),\underline{\bm \beta}(t))\) has the same joint distribution as \((\underline{\mathring{C}}(t),{\bm\theta}(t))\).
We then define the families \(\underline{\mathring{\bm \zeta}}:=\set{\{\mathring{\zeta}_i^{z_l}\}_{i\in [n^{\omega_A}]}\given l\in [2]}\) by \(\underline{\mathring{\bm \zeta}}(0)=0\) and
\begin{equation}
    \label{eq:defpor}
    \dif \underline{\mathring{\bm \zeta}}(t):= \left(\underline{\mathring{C}}^\#(t)\right)^{1/2} \dif \underline{\bm \beta}(t),
\end{equation}
and extend this to negative indices by \(\zeta_{-i}^{z_l}=-\zeta_i^{z_l}\) for \(i\in [n^{\omega_A}]\). For indices \(n^{\omega_A}< \abs{i}\le n\), instead, we choose \(\{\widetilde{\zeta}_{\pm i}^{(l)}\}_{i=n^{\omega_A}+1}^n\) to be independent families
(independent of each other for different \(l\)'s, and also independent of \(\bm\beta\)) of i.i.d.\ Brownian motions defined on the same probability space \(\Omega_\beta\). Note that~\eqref{eq:nefdis} follows by the construction in~\eqref{eq:defpor}.

Similarly to Lemma~\ref{lem:firststepmason} we also have that \({\bm \mu}(t)\) and \(\widetilde{\bm \mu}(t)\) are close thanks to the carefully designed relation between their driving Brownian motions. The proof of this lemma is postponed to Section~\ref{sec:noncelpiu}.

\begin{lemma}[The partially correlated process \(\widetilde{\bm \mu}\) is close to \({\bm \mu}\)]\label{lem:secondstepmason}
    For any sufficiently small \(\omega_d,\omega_h, \omega_f>0\), there exist constants \(\omega, \widehat{\omega}>0\) such that \(\omega_h\ll \widehat{\omega}\ll \omega\ll \omega_f\), and that for \(\abs{z_l-z_m},\abs{z_l-\overline{z}_m}, \abs{z_l-\overline{z}_l}\ge n^{-\omega_d}\), \(\abs{z_l}\le 1-n^{-\omega_h}\), with \(l,m\in [2]\), \(l\ne m\), it holds
    \begin{equation}
        \label{eq:firshpb2}
        \abs*{\mu_i^{(l)}(ct_f)- \widetilde{\mu}_i^{(l)}(ct_f) }\le n^{-1-\omega}, \quad \abs{i}\le n^{\widehat{\omega}}, \quad l\in [2],
    \end{equation}
    with very high probability, where \(t_f=n^{-1+\omega_f}\) and \(c>0\) is defined in~\eqref{eq:consOU}.
\end{lemma}

\subsubsection{Proof of Proposition~\ref{prop:indeig}}\label{sec:INDFI}

In this section we conclude the proof of Proposition~\ref{prop:indeig} using the comparison processes defined in Section~\ref{sec:remlamb} and Section~\ref{sec:COMPPRO}. We recall that \(p=2\) for simplicity. More precisely, we use that the processes \({\bm \lambda}^{z_l}(t)\), \(\mathring{\bm \lambda}^{z_l}(t)\) and \(\mathring{\bm \lambda}^{z_l}(t)\), \(\widetilde{\bm \lambda}^{(l)}(t)\) and \(\widetilde{\bm \mu}^{(l)}(t)\), \({\bm \mu}^{(l)}(t)\) are close path-wise at time \(t_1\), as stated in Proposition~\ref{prop:realprop}, Lemma~\ref{lem:firststepmason}, and Lemma~\ref{lem:secondstepmason}, respectively, choosing \(\omega,\widehat{\omega}\) as the minimum of the ones in the statements of this three results. In particular, by these results and Lemma~\ref{lem:GFTGFT} we readily conclude the following lemma, whose proof is postponed to the end of this section.

\begin{lemma}\label{lem:las}
    Let \({\bm \lambda}^{z_l}\) be the eigenvalues of \(H^{z_l}\), and let \({\bm \mu}^{(l)}(t)\) be the solution of~\eqref{eq:ginev}. Let \(\omega,\widehat{\omega},\omega_h>0\) given as above, and define \(\nu_{z_l}:=\rho_{sc}(0)/\rho^{z_l}(0)\), then for any small \(\omega_f>0\) such that \(\omega_h\ll \omega_f\) there exists \(\delta_0,\delta_1\) such that \(\omega_h\ll \delta_m \ll \widehat{\omega}\), for \(m=0,1\), and that
    \begin{equation}
        \label{eq:hhh4}
        \E  \prod_{l=1}^2\frac{1}{n}\sum_{\abs{i_l}\le n^{\widehat{\omega}}} \frac{\eta_l}{(\lambda_{i_l}^{z_l} )^2+\eta_l^2}= \E \prod_{l=1}^2\frac{1}{n}\sum_{\abs{i_l}\le n^{\widehat{\omega}}} \frac{\eta_l}{(\mu_{i_l}^{(l)}(ct_f) \nu_{z_l})^2+\eta_l^2}+\mathcal{O}(\Psi),
    \end{equation}
    where \(t_f=n^{-1+\omega_f}\), \(\eta_l\in [n^{-1-\delta_0},n^{-1+\delta_1}]\), and the error term is given by
    \begin{equation}
        \label{eq:psierr}
        \Psi:= \frac{n^{\widehat{\omega}}}{n^{1+\omega}}\left(\sum_{l=1}^2\frac{1}{\eta_l}\right)\cdot \prod_{l=1}^2\left(1+\frac{n^\xi}{n\eta_l}\right)+\frac{n^{2\xi+2\delta_0} t_f}{n^{1/2}}\sum_{l=1}^2 \frac{1}{\eta_l}+\frac{n^{2(\delta_1+\delta_0)}}{n^{\widehat{\omega}}}.
    \end{equation}
\end{lemma}

We remark that \(\Psi\) in~\eqref{eq:psierr} denotes a different error term compared with the error terms in~\eqref{eq psi error} and~\eqref{eq:smallpsi}.

By the definition of the processes \({\bm \mu}^{(l)}(t)\) in~\eqref{eq:ginev} it follows that \({\bm \mu}^{(l)}(t)\), \({\bm \mu}^{(m)}(t)\) are independent for \(l\ne m\) and so that
\begin{equation}
    \label{eq:hhh6}
    \E  \prod_{l=1}^2\frac{1}{n}\sum_{\abs{i_l}\le n^{\widehat{\omega}}} \frac{\eta_l}{(\mu_{i_l}^{(l)}(ct_f)\nu_{z_l})^2+\eta_l^2}= \prod_{l=1}^2\E \frac{1}{n}\sum_{\abs{i_l}\le n^{\widehat{\omega}}} \frac{\eta_l}{(\mu_{i_l}^{(l)}(ct_f)\nu_{z_l})^2+\eta_l^2}.
\end{equation}
Then, similarly to Lemma~\ref{lem:las}, we conclude that
\begin{equation}
    \label{eq:hhh5}
    \prod_{l=1}^2\E \frac{1}{n}\sum_{\abs{i_l}\le n^{\widehat{\omega}}} \frac{\eta_l}{(\lambda_{i_l}^{z_l})^2+\eta_l^2}= \prod_{l=1}^2\E \frac{1}{n}\sum_{\abs{i_l}\le n^{\widehat{\omega}}} \frac{\eta_l}{(\mu_{i_l}^{(l)}(ct_f)\nu_{z_l})^2+\eta_l^2}+\mathcal{O}(\Psi).
\end{equation}
Finally, combining~\eqref{eq:hhh4}--\eqref{eq:hhh5} we conclude the proof of Proposition~\ref{prop:indeig}.\qed%

We remark that in order to prove~\eqref{eq:hhh5} it would not be necessary to introduce the additional comparison processes \(\widetilde{\bm \lambda}^{(l)}\) and \(\widetilde{\bm \mu}^{(l)}\) of Section~\ref{sec:COMPPRO}, since in~\eqref{eq:hhh5} the product is outside the expectation, so one can compare the expectations one by one; the correlation between these processes for different \(l\)'s plays no role. Hence, already the usual coupling (see e.g.~\cite{MR3541852,MR3916329, MR3914908}) between the processes \({\bm \lambda}^{z_l}(t)\), \({\bm \mu}^{(l)}(t)\) defined in~\eqref{eq:impnewDBM} and~\eqref{eq:ginev}, respectively, would be sufficient to prove~\eqref{eq:hhh5}. On the other hand, the comparison processes \(\mathring{\bm \lambda}^{z_l}(t)\) are anyway needed in order to remove the coefficients \(\Lambda_{ij}\) (which are small with very high probability) from the interaction term in~\eqref{eq:impnewDBM}.

We conclude this section with the proof of Lemma~\ref{lem:las}.

\begin{proof}[Proof of Lemma~\ref{lem:las}]
    In the following, to simplify notations, we assume that the scaling factors \(\nu_{z_l}\) are equal to one. First of all, we notice that the summation over the indices \(n^{\widehat{\omega}}<\abs{i}\le n\) in~\eqref{eq:stgft2} can be removed, using the eigenvalue rigidity~\eqref{eq:hoplastrig} similarly to~\cite[Eq. \eqref{cplx-eq:dedrez}--\eqref{cplx-eq:rigbb}]{1912.04100}, at a price of an additional error term \(n^{2(\delta_1+\delta_0)-\widehat{\omega}}\):
    \begin{equation}
        \label{eq:newsplithh}
        \E  \prod_{l=1}^2\frac{1}{n}\sum_{\abs{i_l}\le n^{\widehat{\omega}}} \frac{\eta_l}{(\lambda_{i_l}(H^{z_l}))^2+\eta_l^2}=  \E  \prod_{l=1}^2\frac{1}{n}\sum_{\abs{i_l}\le n} \frac{\eta_l}{(\lambda_{i_l}(H^{z_l}))^2+\eta_l^2}+\mathcal{O} \left(\frac{n^{2(\delta_1+\delta_0)}}{n^{\widehat{\omega}}}\right).
    \end{equation}
    The error term is negligible by choosing \(\delta_0,\delta_1\) to be such that \(\omega_h\ll\delta_m\ll \widehat{\omega}\), for \(m=0,1\). Then, from the GFT Lemma~\ref{lem:GFTGFT}, and~\eqref{eq:impGFT}, using~\eqref{eq:newsplithh} again, this time for \(\lambda_{i_l}^{z_l}(ct_f)\), we have that
    \begin{equation}\label{eq:hhh0}
        \begin{split}
            \E  \prod_{l=1}^2\frac{1}{n}\sum_{\abs{i_l}\le n^{\widehat{\omega}}} \frac{\eta_l}{(\lambda_{i_l}(H^{z_l}))^2+\eta_l^2}&=\E  \prod_{l=1}^2\frac{1}{n}\sum_{\abs{i_l}\le n^{\widehat{\omega}}} \frac{\eta_l}{(\lambda_{i_l}^{z_l}(ct_f))^2 +\eta_l^2} \\
            &\quad +\mathcal{O}\left(\frac{n^{2\xi+2\delta_0} t_f}{n^{1/2}}\sum_{l=1}^2 \frac{1}{\eta_l}+\frac{n^{2(\delta_1+\delta_0)}}{n^{\widehat{\omega}}}\right).
        \end{split}
    \end{equation}
    We remark that the rigidity for \(\lambda_{i_l}^{z_l}(ct_f)\) is obtained by Theorem~\ref{theo:Gll} exactly as in~\eqref{eq:hoplastrig}. Next, by the same computations as in~\cite[Lemma~\ref{cplx-lem:stanc}]{1912.04100} by writing the difference of l.h.s.\ and r.h.s.\ of~\eqref{eq:hhh1} as a telescopic sum and then using the very high probability bound from Proposition~\ref{prop:realprop} we get
    \begin{equation}
        \label{eq:hhh1}
        \E  \prod_{l=1}^2\frac{1}{n}\sum_{\abs{i_l}\le n^{\widehat{\omega}}} \frac{\eta_l}{(\lambda_{i_l}^{z_l}(ct_f))^2+\eta_l^2}=\E  \prod_{l=1}^2\frac{1}{n}\sum_{\abs{i_l}\le n^{\widehat{\omega}}} \frac{\eta_l}{(\mathring{\lambda}_{i_l}^{(l)}(ct_f))^2+\eta_l^2}+\mathcal{O}(\Psi).
    \end{equation}

    Similarly to~\eqref{eq:hhh1}, by Lemma~\ref{lem:firststepmason} it also follows that
    \begin{equation}
        \E  \prod_{l=1}^2\frac{1}{n}\sum_{\abs{i_l}\le n^{\widehat{\omega}}} \frac{\eta_l}{(\mathring{\lambda}_{i_l}^{z_l}(ct_f))^2+\eta_l^2}=\E  \prod_{l=1}^2\frac{1}{n}\sum_{\abs{i_l}\le n^{\widehat{\omega}}} \frac{\eta_l}{(\widetilde{\lambda}_{i_l}^{(l)}(ct_f))^2+\eta_l^2}+\mathcal{O}(\Psi).
    \end{equation}
    By~\eqref{eq:nefdis} it readily follows that
    \begin{equation}
        \label{eq:hhh2}
        \E  \prod_{l=1}^2\frac{1}{n}\sum_{\abs{i_l}\le n^{\widehat{\omega}}} \frac{\eta_l}{(\widetilde{\lambda}_{i_l}^{(l)}(ct_f))^2+\eta_l^2}=\E  \prod_{l=1}^2\frac{1}{n}\sum_{\abs{i_l}\le n^{\widehat{\omega}}} \frac{\eta_l}{(\widetilde{\mu}_{i_l}^{(l)}(ct_f))^2+\eta_l^2}.
    \end{equation}
    Moreover, by~\eqref{eq:firshpb2}, similarly to~\eqref{eq:hhh1}, we conclude
    \begin{equation}
        \label{eq:hhh3}
        \E  \prod_{l=1}^2\frac{1}{n}\sum_{\abs{i_l}\le n^{\widehat{\omega}}} \frac{\eta_l}{(\widetilde{\mu}_{i_l}^{(l)}(ct_f))^2+\eta_l^2}=\E  \prod_{l=1}^2\frac{1}{n}\sum_{\abs{i_l}\le n^{\widehat{\omega}}} \frac{\eta_l}{(\mu_{i_l}^{(l)}(ct_f))^2+\eta_l^2}+\mathcal{O}(\Psi).
    \end{equation}
    Combining~\eqref{eq:hhh0}--\eqref{eq:hhh3}, we conclude the proof of~\eqref{eq:hhh4}.
\end{proof}

Finally, we conclude Section~\ref{sec:PO} by listing the scales needed in the entire Section~\ref{sec:IND}
and explain the dependences among them.

\subsubsection{Relations among the scales in the proof of Proposition~\ref{prop:indeig}}

Throughout Section~\ref{sec:IND} various scales are characterized by exponents  of \(n\), denoted by \(\omega\)'s, that we will also refer to scales for simplicity.

All the scales in the proof of Proposition~\ref{prop:indeig} depend on the exponents \(\omega_d, \omega_h, \omega_f\ll 1\).
We recall that \(\omega_d, \omega_h\) are  the exponents such that
Lemma~\ref{lem:lambound} on eigenvector overlaps holds under the assumption \(\abs{z_l-z_m},\abs{z_l-\overline{z}_m}, \abs{z_l-\overline{z}_l}\ge n^{-\omega_d}\), and  \(\abs{z_l}\le 1-n^{-\omega_h}\). The exponent \(\omega_f\) determines the time \(t_f= n^{-1+\omega_f}\)
to run the DBM so that it reaches its local equilibrium and thus to prove the asymptotic independence of \(\lambda_i^{z_l}(ct_f)\)
and \(\lambda_j^{z_m}(ct_f)\), with \(c>0\) defined in~\eqref{eq:consOU}, for small indices \(i,j\) and \(l\ne m\).

The most important scales in the proof of Proposition~\ref{prop:indeig}  are \(\omega,\widehat{\omega},\delta_0, \delta_1,\omega_E\). The scale \(\omega_E\)
is determined in Lemma~\ref{lem:lambound} and it
controls the correlations among the driving martingales originating from the eigenvector overlaps in~\eqref{eq:sym}--\eqref{eq:deflamzz}.
The scale \(\omega\) gives the \(n^{-1-\omega}\) precision of the coupling between various processes while
\(\widehat{\omega}\) determines the range of indices \(\abs{i}\le n^{\widehat{\omega}}\) for which this coupling is effective. These scales are chosen much bigger
than \(\omega_h\) and they  are determined in
Proposition~\ref{prop:realprop}, Lemma~\ref{lem:firststepmason} and Lemma~\ref{lem:secondstepmason}, that describe these couplings. Each of these
results gives  an upper bound on the scales \(\omega, \widehat{\omega}\), at the end we will choose the smallest of them.
Finally, \(\delta_0, \delta_1\) describe the scale of the range of the \(\eta\)'s
in Proposition~\ref{prop:indeig}. These two scales
are determined in Lemma~\ref{lem:las}, given \(\omega, \widehat{\omega}\) from the previous step.
Putting all these steps together, we constructed \(\omega, \widehat{\omega}, \delta_0, \delta_1\) claimed in
Proposition~\ref{prop:indeig} and hence also in Proposition~\ref{prop:indmr}.
These scales are related as
\begin{equation}\label{eq:chain}
    \omega_h\ll \delta_m \ll \widehat{\omega}\ll \omega\ll \omega_f\ll \omega_E\ll1, \qquad \omega_E=4\omega_d,
\end{equation}
for \(m=0,1\).

Along the proof of Proposition~\ref{prop:indeig}  four auxiliary scales,
\(\omega_L,\omega_A,\omega_r,\omega_c\), are also introduced.
The scale \(\omega_L\)  describes the range of  interaction in the short range approximation processes
\(\widehat{\bm x}^{z_l}(t,\alpha)\) (see~\eqref{eq:defshortrange} later),
while \(\omega_A\) is the scale for which we can (partially) couple the driving martingales of the regularized
processes  \(\mathring{\bm \lambda}^{z_l}(t)\) with the driving Brownian motions of Ginibre processes \({\bm \mu}^{(l)}(t)\).
The scale  \(\omega_c\) is a cut-off in the energy estimate in Lemma~\ref{lem:ee}, see~\eqref{eq:cutoff1}. Finally,
\(\omega_r\) reduces the variance of the driving martingales by a factor \((1+ n^{-\omega_r})^{-1}\) to ensure the  well-posedness
of the processes \(\mathring{\bm \lambda}^{z_l}(t)\), \(\widetilde{\bm \lambda}^{(l)}(t)\), \(\widetilde{\bm \mu}^{(l)}\), \({\bm x}^{z_l}(t,\alpha)\) defined in~\eqref{eq:nuproc},~\eqref{eq:nuproc3},~\eqref{eq:nuproc2}, and~\eqref{eq:defintpro}, respectively.
These scales are inserted in the chain~\eqref{eq:chain} as follows
\begin{equation}
    \omega_h \ll \omega_A\ll \omega_f\ll \omega_L\ll \omega_c\ll \omega_r\ll \omega_E.
\end{equation}
Note that there are no relations required among \(\omega_A\) and \(\omega, \widehat{\omega},\delta_m\).

\subsection{Universality and independence of the singular values of \texorpdfstring{\(X-z_1,X-z_2\)}{X-z1,X-z2} close to zero: Proof of Theorems~\ref{theo:un} and~\ref{theo:indun}}\label{sec:un}
In the following we present only the proof of Theorem~\ref{theo:indun}, since the proof of Theorem~\ref{theo:un} proceeds exactly in the same way. Universality of the joint distribution of the singular values of \(X-z_1\) and \(X-z_2\) follows by universality for the joint distribution of the eigenvalues of \(H^{z_1}\) and \(H^{z_2}\), which is defined in~\eqref{eq:herm}, since the eigenvalues of \(H^{z_l}\) are exactly the singular values of \(X-z_l\) taken with positive and negative sign. From now on we only consider the eigenvalues of \(H^{z_l}\), with \(z_l\in\C\) such that \(\abs{\Im z_l}\sim 1\), \(|z_1-z_2|,|z_1-\overline{z}_2|\sim 1\), and \(\abs{z_l}\le 1-\epsilon\) for some small fixed \(\epsilon>0\).

For \(l\in [2]\), denote by \(\{\lambda_i^{z_l}\}_{\abs{i}\le n}\) the eigenvalues of \(H^{z_l}\) and by \(\{ \lambda_i^{z_l}(t)\}_{\abs{i}\le n}\) their evolution under the DBM flow~\eqref{eq:impnewDBM}. Define \(\{\mu_i^{(l)}(t)\}_{\abs{i}\le n}\), for \(l\in [2]\), to be the solution of~\eqref{eq:ginev} with initial data \(\{\mu_i^{(l)}\}_{\abs{i}\le n}\), which are the eigenvalues of independent complex Ginibre matrices \(\widetilde{X}^{(1)}, \widetilde{X}^{(2)}\). Then, defining the comparison processes  \(\mathring{\bm \lambda}^{z_l}(t)\), \(\widetilde{\bm \lambda}^{(l)}(t)\), \(\widetilde{\bm \mu}^{(l)}(t)\) as in Sections~\ref{sec:remlamb}--\ref{sec:COMPPRO}, and combining Proposition~\ref{prop:realprop}, Lemma~\ref{lem:firststepmason}, and Lemma~\ref{lem:secondstepmason}, we conclude that for any sufficiently small \(\omega_f>0\) there exist \(\omega, \widehat{\omega}>0\) such that \(\widehat{\omega}\ll \omega\ll \omega_f\), and that
\begin{equation}
    \label{eq:hpunv}
    \abs{\rho^{z_l}(0)\lambda_i^{z_l}(ct_f)-\rho_{\mathrm{sc}}(0)\mu_i^{(l)}(ct_f)}\le n^{-1-\omega}, \qquad \abs{i}\le n^{\widehat{\omega}},
\end{equation}
with very high probability, with \(c>0\) defined in~\eqref{eq:consOU}.

Then, by a simple Green's function comparison argument (GFT) as in Lemma~\ref{lem:GFTGFT}, using~\eqref{eq:hpunv}, by exactly the same computations as in the proof of~\cite[Proposition 3.1 in Section 7]{MR4026551} adapted to the bulk scaling, i.e.\ changing \(\mathfrak{b}_{r,t_1}\to 0\) and \(N^{3/4}\to 2n\), using the notation therein, we conclude Theorem~\ref{theo:indun}.

\subsection{Bound on the eigenvector overlaps}\label{sec:lamblamb}
In this section we prove the bound on the eigenvector overlaps, as stated in Lemma~\ref{lem:lambound}. For any \(T>0\), and any \(t\in [0,T]\), denote by \(\rho_t^z\) the self consistent density of states (scDOS) of the Hermitised matrix \(H_t^z\), and define its quantiles by
\begin{equation}
    \label{eq:defquant}
    \frac{i}{n}=\int_0^{\gamma_i^z(t)} \rho_t^z(x) \dif x, \qquad i\in [n],
\end{equation}
and \(\gamma_{-i}^z(t)=-\gamma_i^z(t)\) for \(i\in [n]\). Similarly to~\eqref{eq:hoplastrig}, as a consequence of Theorem~\ref{theo:Gll} and the fact that the eigenvalues of \(H_t^{z_l}\) are \(\gamma\)-H\"older continuous in time for any \(\gamma\in (0,1/2)\) by Weyl's inequality, by standard application of Helffer-Sj\"ostrand formula, we conclude the following rigidity bound
\begin{equation}
    \label{eq:hoplastrigt}
    \sup_{0\le t\le T}\abs{\lambda_i^{z_l}(t)-\gamma_i^{z_l}(t)}\le \frac{n^{100\omega_h}}{n^{2/3} (n+1-i)^{1/3}}, \qquad i\in [n],
\end{equation}
with very high probability, uniformly in \(\abs{z_l}\le 1-n^{-\omega_h}\). A bound similar to~\eqref{eq:hoplastrigt} holds for negative indices as well. We remark that the H\"older continuity of the eigenvalues of \(H_t^{z_l}\) is used to prove~\eqref{eq:hoplastrigt} uniformly in time, using a standard grid argument.

The main input to prove Lemma~\ref{lem:lambound} is Theorem~\ref{thm local law G2} combined with Lemma~\ref{lem:lowbeta}.
\begin{proof}[Proof of Lemma~\ref{lem:lambound}]
    Recall that \(P_1 {\bm w}_i^z={\bm u}_i^z\) and \(P_2 {\bm w}_i^z=\text{sign}(i){\bm v}_i^z\), for \(\abs{i}\le n\), by~\eqref{eq:defpro}. In the following we consider \(z,z'\in\C\) such that \(\abs{z},\abs{z'}\le 1-n^{-\omega_h}\), \(\abs{z-z'}\ge n^{-\omega_d}\), for some sufficiently small \(\omega_h,\omega_d>0\).

    Eigenvector overlaps can be estimated by traces of products of resolvents. More precisely, for any \(\eta\ge n^{-2/3+\epsilon_*}\), for some small fixed \(\epsilon_*>0\), and any \(\abs{i_0},\abs{j_0}\le n\), using the rigidity bound~\eqref{eq:hoplastrigt}, similarly to~\cite[Eq.~\eqref{cplx-eq:ooo}]{1912.04100}, we have that
    \begin{equation}
        \begin{split}
            &\abs{\braket{ {\bm u}_{i_0}^z(t), {\bm u}_{j_0}^{z'}(t)}}^2\lesssim \eta^2\Tr \bigl( \Im G^z(\gamma_{i_0}^z(t)+\ii \eta)\bigr)E_1\bigl(\Im G^{z'}(\gamma_{j_0}^{z'}(t)+\ii \eta)\bigr)E_1, \\
            &\abs{\braket{ {\bm v}_{i_0}^z(t), {\bm v}_{j_0}^{z'}(t)}}^2\lesssim \eta^2\Tr \bigl( \Im G^z(\gamma_{i_0}^z(t)+\ii \eta)\bigr)E_2\bigl( \Im G^{z'}(\gamma_{j_0}^{z'}(t)+\ii \eta)\bigr)E_2,
        \end{split}
    \end{equation}
    with \(E_1,E_2\) defined in~\eqref{e1e2}. By Theorem~\ref{thm local law G2}, combined with Lemma~\ref{lem:lowbeta}, choosing \(\eta=n^{-12/23}\), say, the error term in the r.h.s.\ of~\eqref{final local law} is bounded by \(n^{-1/23}n^{2\omega_d+100\omega_h}\), hence we conclude the bound in~\eqref{eq:lab1} for any fixed time \(t\in [0,T]\), choosing \(\omega_E=-(2\omega_d+100\omega_h-1/23)\), for any \(\omega_h\ll \omega_d\le 1/100\).

    Moreover, the bound~\eqref{eq:lab1} holds uniformly in time by a union bound, using a standard grid argument and H\"older continuity in the form
    \[ \norm{ \Im G_t^z \Im G_t^{z'}-\Im G_s^z \Im G_s^{z'}}\lesssim n^3 \left(\norm{ H_t^z-H_s^z}+\norm{ H_t^{z'}-H_s^{z'}}\right)\lesssim n^{7/2} \abs{t-s}^{1/2} \]
    for any \(s,t\in [0,T]\), where the spectral parameters in the resolvents have imaginary parts at least \(\eta >1/n\). This concludes the proof of Lemma~\ref{lem:lambound}.
\end{proof}

\subsection{Proof of Proposition~\ref{prop:realprop}}\label{sec:newse}
Throughout this section we use the notation \(z=z_l\), with \(l\in [2]\), with \(z_1,z_2\) fixed as in Proposition Proposition~\ref{prop:realprop}.

\begin{remark}
    In the remainder of this section we assume that \(\abs{z}\le 1-\epsilon\), with some positive \(\epsilon>0\) instead of \(n^{-\omega_h}\), in order to make our presentation clearer. One may follow the \(\epsilon\)-dependence throughout the proofs and find that all the estimates deteriorate with some fixed \(\epsilon^{-1}\) power, say \(\epsilon^{-100}\). Thus, when \(\abs{z}\le 1-n^{-\omega_h}\) is assumed, we get an additional factor \(n^{100\omega_h}\)
    but this does not play any role since \(\omega_h\) is the smallest exponent (e.g.\ see Proposition~\ref{prop:realprop}) in the analysis of the processes~\eqref{eq:impnewDBM},~\eqref{eq:nuproc}.
\end{remark}

The proof of Proposition~\ref{prop:realprop} consists of several parts  that we first sketch. The process \(\mathring{{\bm\lambda}}^z(t)\) differs
from \({\bm \lambda}^z(t)\)  in three aspects: (i) the coefficients \(\Lambda_{ij}^z(t)\) in the SDE~\eqref{eq:impnewDBM} for \({\bm \lambda}^z(t)\) are removed; (ii) large values of the correlation of the driving martingales is cut off, and (iii) the martingale term is slightly reduced by a factor \((1+n^{\omega_r})^{-1/2}\). We deal with these differences in two steps. The substantial step is the first one, from Section~\ref{sec:interpol} to Section~\ref{sec:ee}, where we handle (i) by interpolation, using short range approximation and energy method. This is followed by a more technical second step in Section~\ref{sec:dddaah}, where we handle (ii) and (iii) using a stopping time  controlled by a well chosen Lyapunov function to show that the correlation typically remains below the cut-off level.

A similar  analysis has been done in~\cite[Section 4]{MR4009717} (which has been used in the singular value setup in~\cite[Eq.~(3.13)]{1908.04060}) but our more complicated setting requires major modifications. In particular,~\eqref{eq:impnewDBM} has to be compared to~\cite[Eq.~(4.1)]{MR4009717} with \(\dif M_i=0\), \(Z_i=0\), and identifying \(\Lambda_{ij}^z\) with \(\gamma_{ij}\), using the notations therein. One major difference is that we now have a much weaker estimate \(\abs{\Lambda_{ij}^z}\le n^{-\omega_E}\) than the bound \(\abs{\gamma_{ij}}\le n^{-1+a}\), for some small fixed \(a>0\), used in~\cite{MR4009717}. We therefore need to introduce an additional cut-off function \({\bm \chi}\) in the energy estimate in Section~\ref{sec:ee}.

\subsubsection{Interpolation process}\label{sec:interpol}

In order to compare the processes \({\bm \lambda}^z\) and
\(\mathring{{\bm\lambda}}^z\) from~\eqref{eq:impnewDBM} and~\eqref{eq:nuproc} we start with defining an interpolation process, for any \(\alpha\in [0,1]\), as
\begin{equation}\label{eq:defintpro}
    \dif x_i^z(t,\alpha)= \frac{\dif \mathring{b}_i^z}{\sqrt{n(1+n^{-\omega_r})}}+\frac{1}{2n}\sum_{j\ne i} \frac{1+\alpha\mathring{\Lambda}_{ij}^z(t)}{x_i^z(t,\alpha)-x_j^z(t,\alpha)}\dif t, \quad x_i^z(0,\alpha)=\lambda_i^z(0), \quad \abs{i}\le n.
\end{equation}
We recall that \(\omega_f\ll\omega_r\ll \omega_E\). We use the notation \(x_i^z(t,\alpha)\) instead of \(z_i(t,\alpha)\) as in~\cite[Eq.~(4.12)]{MR4009717} to stress the dependence of \(x_i^z(t,\alpha)\) on \(z\in\C\). The well-posedness of the process~\eqref{eq:defintpro} is proven in Appendix~\ref{sec:wp} for any fixed \(\alpha\in [0,1]\). In particular, the particles keep their order \(x_i^z(t,\alpha)< x_{i+1}^z(t,\alpha)\). Additionally, by Lemma~\ref{lem:lip} it follows that the differentiation with respect to \(\alpha\) of the process \({\bm x}^z(t,\alpha)\) is well-defined.

Note that the process \({\bm x}^z(t,\alpha)\) does not fully interpolate between \(\mathring{{\bm\lambda}}^z(t)\) and \({\bm \lambda}^z(t)\); it handles only the removal of the \(\mathring{\Lambda}_{ij}\) term.
Indeed, it holds \({\bm x}^z(t,0)=\mathring{{\bm\lambda}}^z(t)\) for any \(t\in [0,T]\), but \({\bm x}^z(t,1)\) is not equal to \({\bm \lambda}^z(t)\). Thus we will proceed in two steps as already explained:
\begin{enumerate}[label=Step \arabic*]
    \item\label{step1} The process \({\bm x}^z(t,\alpha)\) does not change much in \(\alpha\in [0,1]\) for particles close to zero (by Lemma~\ref{lem:ee} below), i.e.\ \(x_i^z(t,1)-x_i^z(t,0)\) is much smaller than the rigidity scale \(1/n\) for small indices;
    \item\label{step2} The process \({\bm x}^z(t,1)\) is very close to \({\bm \lambda}^z(t)\) for all indices (see Lemma~\ref{lem:comp} below).
\end{enumerate} We start with the analysis of the interpolation process \({\bm x}^z(t,\alpha)\), then in Section~\ref{sec:dddaah} we state and prove Lemma~\ref{lem:comp}.

\subsubsection{Local law for the interpolation process}

In order to analyse the interpolation process \({\bm x}^z(t,\alpha)\), we first need to establish a local law for the Stieltjes transform of the empirical particle density. This will be used for a rigidity estimate to identify the location of \(x_i(t,\alpha)\) with a precision \(n^{-1+\epsilon}\), for some small \(\epsilon>0\), that is above the final target precision but it is needed as an a priori bound. Note that, unlike for \({\bm \lambda}^z(t)\), for \({\bm x}^z(t,\alpha)\) there is no obvious matrix ensemble behind this process, so local law and rigidity have to be proven directly from
its defining equation~\eqref{eq:defintpro}.

Define the Stieltjes transform of the empirical particle density by
\begin{equation}\label{eq:sttra}
    m_n(w,t,\alpha)=m_n^z(w,t,\alpha):=  \frac{1}{2n}\sum_{\abs{i}\le n} \frac{1}{x_i^z(t,\alpha)-w},
\end{equation}
and denote the Stieltjes transform of \(\rho^z\), the \emph{self-consistent density of states (scDOS)}  of \(H^z\), by \(m^z(w)\). Moreover, we denote the Stieltjes transform of \(\rho_t^z\),  the free convolution of \(\rho^z\) with the semicircular flow up to time \(t\), by \(m_t^z(w)\). Using the definition of the quantiles \(\gamma_i^z(t)\) in~\eqref{eq:defquant}, by Theorem~\ref{theo:Gll} we have that
\begin{equation}\label{eq:llt0}
    \begin{split}
        \sup_{\abs{\Re w}\le 10c_1}\sup_{n^{-1+\gamma}\le \Im w \le 10}\sup_{\alpha\in [0,1]}\abs{m_n(w,0,\alpha)-m^z(w)}&\le \frac{n^\xi C_\epsilon}{n\Im w}, \\
        \sup_{\abs{i}\le 10 c_2 n}\sup_{\alpha\in [0,1]}\abs{x_i^z(0,\alpha)-\gamma_i^z(0)}&\le \frac{C_\epsilon n^\xi}{n},
    \end{split}
\end{equation}
with very high probability for any \(\xi>0\), uniformly in \(\abs{z}\le 1-\epsilon\), for some small fixed \(c_1, c_2, \gamma>0\). We recall that \(C_\epsilon\le \epsilon^{-100}\). The rigidity bound in the second line of~\eqref{eq:llt0} follows by a standard application of Helffer-Sj\"ostrand formula.

In Lemma~\ref{lem:alphall} we prove that~\eqref{eq:llt0} holds true uniformly in \(0\le t\le t_f\).  For its proof,
similarly to~\cite[Section 4.5]{MR4009717}, we follow the analysis of~\cite[Section 3.2]{MR4009708} using~\eqref{eq:llt0} as an input.

\begin{lemma}[Local law and rigidity]\label{lem:alphall}
    Fix \(\abs{z}\le 1-\epsilon\), and assume that~\eqref{eq:llt0} holds with some \(\gamma, c_1,c_2, C_\epsilon>0\), then
    \begin{equation}\label{eq:alphrig}
        \begin{split}
            \sup_{\abs{\Re w}\le 10 c_1}\sup_{n^{-1+\gamma}\le \Im w \le 10}\sup_{\alpha\in [0,1]}\sup_{0\le t\le t_f}\abs{m_n^z(w,t,\alpha)-m_t(w)}&\le \frac{C_\epsilon n^\xi}{n\Im w}, \\
            \sup_{\abs{i}\le 10c_2 n}\sup_{\alpha\in [0,1]}\sup_{0\le t\le t_f}\abs{x_i^z(t,\alpha)-\gamma_i^z(t)}&\le \frac{C_\epsilon n^\xi}{n},
        \end{split}
    \end{equation}
    with very high probability for any \(\xi>0\), with \(\gamma_i^z(t)\sim i/n\) for \(\abs{i}\le 10 c_2 n\) and \(t\in [0,t_f]\).
\end{lemma}
\begin{proof}

    Differentiating~\eqref{eq:sttra}, by~\eqref{eq:defintpro} and It\^{o}'s formula, we get
    \begin{equation}\label{eq:hf}
        \begin{split}
            \dif m_n&=m_n(\partial_w m_n)\dif t-\frac{1}{2n^{3/2}\sqrt{1+n^{-\omega_r}}}\sum_{\abs{i}\le n}\frac{\dif \mathring{b}_i}{(x_i-w)^2}+\frac{\alpha}{4n^2}\sum_{\abs{i},\abs{j}\le n} \frac{\mathring{\Lambda}_{ij}}{(x_i-w)^2(x_j-w)} \dif t \\
            &\quad +\frac{1}{4n^2} \sum_{\abs{i}\le n}\frac{\bigl[1-\alpha-n^{-\omega_r}(1+n^{-\omega_r})^{-1}\bigr]\mathring{\Lambda}_{ii}}{(x_i-w)^3}\dif t.
        \end{split}
    \end{equation}
    Note that by~\eqref{eq:ringmat}--\eqref{eq:ringlamthe} it follows that
    \begin{equation}\label{eq:eq}
        \mathring{\Lambda}_{ij}(t)=\Lambda_{ij}(t),  \qquad \bigl(\mathring{b}_i(s)\bigr)_{0\le s\le t}=\bigl(b_i(s)\bigr)_{0\le s\le t},
    \end{equation}
    with very high probability uniformly in \(0\le t \le t_f\), where \(\Lambda_{ij}\) and \((b_i(s))_{0\le s\le t}\) are defined in~\eqref{eq:defbiglam}--\eqref{eq:deflamzz} and~\eqref{eq:deffiltr}--\eqref{eq:recallBBB}, respectively.

    The equation~\eqref{eq:hf} is the analogue of~\cite[Eq.~(3.20)]{MR4009708} with some differences. First, the last two terms are new and need to be estimated, although the penultimate term in~\eqref{eq:hf} already appeared in~\cite[Eq.~(4.62)]{MR4009717} replacing \(\mathring{\Lambda}_{ij}\) by \(\widehat{\gamma}_{ij}\), using the notation therein.
    Second, the martingales in the second term in the r.h.s.\ of~\eqref{eq:hf} are correlated. Hence, in order to apply the results in~\cite[Section 3.2]{MR4009708} we prove that these additional terms are bounded as in~\cite[Eq.~(4.64)]{MR4009717}.
    Note that in~\cite[Eq.~(4.64)]{MR4009717} the corresponding term to the penultimate term in the r.h.s.\ of~\eqref{eq:hf} is estimated using that \(\widehat{\gamma}_{ij}\le n^{-1+a}\), for some small \(a>0\). In our case, however, the bound on \(\abs{\mathring{\Lambda}}\) is much weaker and a crude estimate by absolute value is not affordable.
    We will use~\eqref{eq:eq} and then the explicit form of \(\Lambda_{ij}\) in~\eqref{eq:defbiglam}--\eqref{eq:deflamzz}, that enables us to perform the two summations and write this term as the trace of the product of two operators (see~\eqref{eq:term2} later).

    Since \(\abs{\mathring{\Lambda}_{ii}}\le n^{-\omega_E}\) by its definition below~\eqref{eq:ringlamthe}, the last term in~\eqref{eq:hf} is easily bounded by
    \begin{equation}
        \label{eq:term1}
        \abs*{\frac{1}{4n^2} \sum_{\abs{i}\le n}\frac{(1-\alpha-n^{-\omega_r}(1+n^{-\omega_r})^{-1})\mathring{\Lambda}_{ii}}{(x_i-w)^3}}\le \frac{\Im m_n(w)}{n^{1+\omega_E}(\Im w)^2}.
    \end{equation}

    Next, we proceed with the estimate of the penultimate term in~\eqref{eq:hf}. Define the operators
    \begin{equation}
        \label{eq:defoperator}
        T(t,\alpha):=  \sum_{\abs{i}\le n} f(x_i(t,\alpha)) {\bm w}_i(t)[{\bm w}_i(t)]^\ast, \qquad S(t,\alpha):=  \sum_{\abs{i}\le n} g(x_i(t,\alpha))\overline{\bm w}_i(t)[\overline{\bm w}_i(t)]^\ast,
    \end{equation}
    where \(\{{\bm w}_i(t)\}_{\abs{i}\le n}\) are the orthonormal eigenvectors in the definition of \(\Lambda_{ij}(t)\) in~\eqref{eq:defbiglam}, and for any fixed \(w\in\HC\) the functions \(f,g\colon\R\to \C\) are defined as
    \begin{equation}
        \label{eq:deffun}
        f(x):= \frac{1}{(x-w)^2}, \qquad g(x):=  \frac{1}{x-w}.
    \end{equation}
    Then, using the definitions~\eqref{eq:defoperator}--\eqref{eq:deffun} and~\eqref{eq:eq}, we bound the last term in the first line of~\eqref{eq:hf} as
    \begin{equation}\label{eq:term2}
        \begin{split}
            \abs*{\frac{\alpha}{4n^2}\sum_{\abs{i},\abs{j}\le n} \frac{\mathring{\Lambda}_{ij}}{(x_i-w)^2(x_j-w)} \dif t }&=\abs*{\frac{\alpha}{2n^2} \left[\Tr \bigl(P_1TP_2 P_2SP_1\bigr)+\overline{\Tr \bigl(P_1TP_2 P_2SP_1\bigr)}\right] } \\
            &\lesssim \frac{1}{n^2}\left[\Im w \Tr  \bigl[P_1TP_2(P_1TP_2)^\ast\bigr]+\frac{\Tr  \bigl[P_1SP_2(P_1SP_2)^\ast\bigr]}{\Im w}\right] \\
            &\lesssim \frac{1}{n^2}\left[\Im w \sum_{\abs{i}\le n} \abs*{ f(x_i)}^2+\frac{1}{\Im w} \sum_{\abs{i}\le n} \abs*{ g(x_i)}^2 \right] \lesssim \frac{\Im m_n(w)}{n(\Im w)^2},
        \end{split}
    \end{equation}
    with very high probability uniformly in \(0\le t\le t_f\). Note that in the first equality of~\eqref{eq:term2} we used that \(\mathring{\Lambda}_{ij}(t)=\Lambda_{ij}(t)\) for any \(0\le t \le t_f\) with very high probability by~\eqref{eq:eq}.

    Finally, in order to conclude the proof, we estimate the martingale term in~\eqref{eq:hf}. For this purpose, using that \(\Exp{\dif \mathring{b}_i \dif \mathring{b}_j\given\mathcal{F}_{b,t}}=(\delta_{i,j}-\delta_{i,-j}+\mathring{\Lambda}_{ij} )/2\dif t\) and proceeding similarly to~\eqref{eq:term2}, we estimate its quadratic variation by
    \begin{equation}\label{eq:term3}
        \begin{split}
            \frac{1}{4n^3(1+n^{-\omega_r})} \sum_{\abs{i},\abs{j}\le n} \frac{\Exp{\dif \mathring{b}_i \dif \mathring{b}_j\given\mathcal{F}_{b,t}}}{(x_i-w)^2(x_j-\overline{w})^2} &=\frac{1}{8n^3(1+n^{-\omega_r})} \sum_{\abs{i}\le n} \frac{1}{\abs{x_i-w}^4} \dif t \\
            &\quad +\frac{1}{8n^3(1+n^{-\omega_r})} \sum_{\abs{i}\le n} \frac{1}{(x_i+w)^2(x_i-\overline{w})^2} \dif t \\
            &\quad +\frac{1}{8n^3(1+n^{-\omega_r})} \sum_{\abs{i},\abs{j}\le n} \frac{\mathring{\Lambda}_{ij}}{(x_i-w)^2(x_j-\overline{w})^2} \dif t \\
            &\lesssim \frac{\Im m_n(w)}{n^2(\Im w)^3}+\frac{1}{n^3} \Tr  \bigl[ P_1 TP_2 (P_1 T P_2)^\ast\bigr]  \dif t \\
            &\lesssim \frac{\Im m_n(w)}{n^2(\Im w)^3},
        \end{split}
    \end{equation}
    where the operator \(T\) is defined in~\eqref{eq:defoperator}, and in the penultimate inequality we used that \(\mathring{\Lambda}_{ij}(t)=\Lambda_{ij}(t)\) for any \(0\le t \le t_f\) with very high probability.

    Combining~\eqref{eq:term1},~\eqref{eq:term2}, and~\eqref{eq:term3} we immediately conclude the proof of the first bound in~\eqref{eq:alphrig} using the arguments of~\cite[Section 3.2]{MR4009708}. The rigidity bound in the second line of~\eqref{eq:alphrig} follows by a standard application of Helffer-Sj\"ostrand (see also below~\eqref{eq:llt0}).
\end{proof}

\subsubsection{Short range approximation}

Since the main contribution to the dynamics of \(x_i^z(t,\alpha)\) comes from the nearby particles, in this section we introduce a \emph{short range approximation} process \(\widehat{\bm x}^z(t,\alpha)\), which will very well approximate
the original process \({\bm x}^z(t,\alpha)\) (see~\eqref{eq:shortlong} below). The actual interpolation analysis comparing \(\alpha=0\)
and \(\alpha=1\) will then be performed on the short range process \(\widehat{\bm x}^z(t,\alpha)\) in Section~\ref{sec:ee}.

Fix \(\omega_L>0\) so that \(\omega_f\ll \omega_L\ll \omega_E\), and define the index set
\begin{equation}
    \label{eq:defA}
    \mathcal{A}:=  \set*{(i,j)\given\abs{i-j}\le n^{\omega_L}}\cup \set*{(i,j)\given \abs{i},\abs{j}> 5c_2n},
\end{equation}
with \(c_2>0\) defined in~\eqref{eq:alphrig}. We remark that in~\cite[Eq.~(4.69)]{MR4009717} the notation \(\omega_l\) is used instead of \(\omega_L\); we decided to change this notation in order to not create confusion with \(\omega_l\) defined in~\cite[Eq.~\eqref{cplx-eq:shortscsetA}]{1912.04100}. Then we define the short range approximation \(\widehat{\bm x}^z(t,\alpha)\) of the process \({\bm x}^z(t,\alpha)\) by
\begin{equation}
    \label{eq:defshortrange}
    \begin{split}
        \dif \widehat{x}_i^z(t,\alpha)&=\frac{\dif \mathring{b}_i^z}{\sqrt{n}}+\frac{1}{2n}\sum_{\substack{j:(i,j)\in\mathcal{A}, \\ j\ne i}} \frac{1+\alpha \mathring{\Lambda}_{ij}(t)}{\widehat{x}_i^z(t,\alpha)-\widehat{x}_j^z(t,\alpha)}\dif t+\frac{1}{2n}\sum_{\substack{j:(i,j)\notin\mathcal{A}, \\ j\ne i}} \frac{1}{x_i^z(t,0)-x_j^z(t,0)}\dif t, \\
        \widehat{x}_i^z(0,\alpha)&=x_i^z(0,\alpha), \qquad \abs{i}\le n.
    \end{split}
\end{equation}
The well-posedness of the process~\eqref{eq:defshortrange} follows by nearly identical computations as in the proof of Proposition~\ref{prop:wp}.

In order to check that the \emph{short range approximation} \(\widehat{\bm x}^z(t,\alpha)\) is close to the process \({\bm x}^z(t,\alpha)\), defined in~\eqref{eq:defintpro}, we start with a trivial bound on \(\abs{x_i^z(t,\alpha)-x_i^z(t,0)}\) (see~\eqref{eq:glo} below) to estimate the difference of particles far away from zero in~\eqref{eq:simpest}, for which we do not have the rigidity bound in~\eqref{eq:alphrig}. Notice that by differentiating~\eqref{eq:defintpro} in \(\alpha\) and estimating \(\abs{\mathring{\Lambda}_{ij}}\) trivially by \(n^{-\omega_E}\), it follows that
\begin{equation}
    \label{eq:glo}
    \sup_{0\le t\le t_f} \sup_{\abs{i}\le n}\sup_{\alpha\in [0,1]} \abs{x_i^z(t,\alpha)-x_i^z(t,0)}\lesssim n^{-\omega_E/2},
\end{equation}
similarly to~\cite[Lemma 4.3]{MR4009717}.

By the rigidity estimate~\eqref{eq:alphrig}, the weak global estimate~\eqref{eq:glo} to estimate the contribution of the far away particles for which we do not know rigidity, and the bound \(\abs{\mathring{\Lambda}_{ij}}\le n^{-\omega_E}\) from~\eqref{eq:ringlamthe} it follows that
\begin{equation}
    \label{eq:simpest}
    \abs*{\frac{1}{2n}\sum_{\substack{j: (i,j)\notin \mathcal{A}, \\ j\ne i}} \frac{1}{x_i^z(t,0)-x_j^z(t,0)}-\frac{1}{2n}\sum_{\substack{j: (i,j)\notin \mathcal{A}, \\ j\ne i}} \frac{1+\alpha\mathring{\Lambda}_{ij}(t)}{x_i^z(t,\alpha)-x_j^z(t,\alpha)}}\lesssim n^{-\omega_E/2}+n^{-\omega_L+\xi},
\end{equation}
for any \(\xi>0\) with very high probability uniformly in \(0\le t \le t_f\). Hence, by exactly the same computations as in~\cite[Lemma 3.8]{MR3914908}, it follows that
\begin{equation}
    \label{eq:shortlong}
    \sup_{\alpha\in [0,1]}\sup_{\abs{i}\le n} \sup_{0\le t\le t_f} \abs{x_i^z(t,\alpha)-\widehat{x}_i^z(t,\alpha)}\le \frac{n^{2\omega_f}}{n}\left(\frac{1}{n^{\omega_E/2}}+\frac{1}{n^{\omega_L}} \right).
\end{equation}

Note that~\eqref{eq:shortlong} implies that the second estimate in~\eqref{eq:alphrig} holds with \(x_i^z\) replaced by \(\widehat{x}_i^z\). In order to conclude the proof of Proposition~\ref{prop:realprop} in the next section we differentiate in  the process \(\widehat{\bm x}^z\) in \(\alpha\) and study the deterministic (discrete) PDE we obtain from~\eqref{eq:defshortrange} after the \(\alpha\)-derivation. Note that the \(\alpha\)-derivative of \(\widehat{\bm x}^z\) is well defined by Lemma~\ref{lem:lip}.

\subsubsection{Energy estimate}\label{sec:ee}

Define \(v_i=v_i^z(t,\alpha):=  \partial_\alpha \widehat{x}_i^z(t,\alpha)\), for any \(\abs{i}\le n\). In the remainder of this section we may omit the \(z\)-dependence since the analysis is performed for a fixed \(z\in\C\) such that \(\abs{z}\le 1-\epsilon\), for some small fixed \(\epsilon>0\). By~\eqref{eq:defshortrange} it follows that \({\bm v}\) is the solution of the equation
\begin{equation}\label{eq:ubo}
    \partial_t v_i=-(B{\bm v})_i+\xi_i, \quad v_i(0)=0, \qquad  \abs{i}\le n,
\end{equation}
where
\begin{equation}\label{eq:kern}
    (B{\bm v})_i:=  \sum_{j:(i,j)\in\mathcal{A}} B_{ij}(v_j-v_i), \quad B_{ij}=B_{ij}(t,\alpha):=  \frac{1+\alpha \mathring{\Lambda}_{ij}(t)}{2n(\widehat{x}_i(t,\alpha)-\widehat{x}_j(t,\alpha))^2} \bm1((i,j)\in\mathcal{A}),
\end{equation}
and
\[\xi_i=\xi_i(t,\alpha):=  \frac{1}{2n}\sum_{j:(i,j)\in\mathcal{A}} \frac{\mathring{\Lambda}_{ij}(t)}{\widehat{x}_i(t,\alpha)-\widehat{x}_j(t,\alpha)}.\]

Before proceeding with the optimal estimate of the \(\ell^\infty\)-norm of \({\bm v}\) in~\eqref{eq:optvb}, we give the following crude bound
\begin{equation}
    \label{eq:infb}
    \sup_{\abs{i}\le n}\sup_{0\le t\le t_f} \sup_{\alpha\in [0,1]} \abs{v_i(t,\alpha)}\lesssim 1,
\end{equation}
that will be needed  as an a priori estimate for the more precise result later.
The bound~\eqref{eq:infb} immediately follows by exactly the same computations as in~\cite[Lemma 4.7]{MR4009717} using that \(\abs{\mathring{\Lambda}_{ij}}\le n^{-\omega_E}\). %

The main technical result to prove~\ref{step1} towards Proposition~\ref{prop:realprop} is the following lemma. In particular,
after integration in \(\alpha\),  Lemma~\ref{lem:ee} proves that the processes \({\bm x}^z(t,1)\) and \({\bm x}^z(t,0)\) are closer than the rigidity scale \(1/n\).

\begin{lemma}\label{lem:ee}
    For any small \(\omega_f>0\) there exist small constants \(\omega, \widehat{\omega}>0\) such that \(\widehat{\omega}\ll \omega\ll \omega_f\) and
    \begin{equation}
        \label{eq:optvb}
        \sup_{\alpha\in [0,1]}\sup_{\abs{i}\le n^{\widehat{\omega}}}\sup_{0\le t\le t_f} \abs{v_i(t)}\le n^{-1-\omega},
    \end{equation}
    with very high probability.
\end{lemma}
This lemma is  based upon the finite speed of propagation mechanism for the dynamics~\eqref{eq:ubo}~\cite[Lemma~9.6]{MR3372074}. Our proof follows~\cite[Lemma~6.2]{MR3606475} that introduced a carefully chosen special cut-off function.
\begin{proof}
    In order to bound \(\abs{v_i(t)}\) for small indices we will bound \(\norm{ {\bm v}{\bm \chi}}_\infty\) for an appropriate cut-off vector \({\bm \chi}\) supported at a few coordinates around zero. More precisely,  we will use an energy estimate to control \(\norm{ {\bm v}{\bm \chi}}_2\) and then we use the trivial bound \(\norm{ {\bm v}{\bm \chi}}_\infty\le \norm{ {\bm v}{\bm \chi}}_2\). This bound would be too crude without the cut-off.

    Fix a small constant \(\omega_c>0\) such that \(\omega_f\ll \omega_L\ll \omega_c\ll \omega_E\), and define
    \begin{equation}
        \label{eq:cutoff1}
        \chi(x):=  e^{-2xn^{1-\omega_c}}, %
    \end{equation}
    for any \(x>0\).
    It is trivial to see that \(\chi\) is Lipschitz, i.e.
    \begin{equation}
        \label{eq:lipb}
        \abs{\chi(x)-\chi(y)}\lesssim e^{-(x\wedge y)n^{1-\omega_c}} \abs{x-y}n^{1-\omega_c},
    \end{equation}
    for any \(x,y\ge 0\), and that
    \begin{equation}
        \label{eq:lipbb}
        \abs{\chi(x)-\chi(y)}\lesssim e^{-(x+y)n^{1-\omega_c}} \abs{x-y}n^{1-\omega_c},
    \end{equation}
    if additionally \(\abs{x-y}\le n^{\omega_c}/(2n)\). Finally we define the vector \({\bm \chi}\) by
    \begin{equation}
        \label{eq:cutoff}
        \chi_i=\chi(\widehat{x}_i):=  e^{-2\abs{\widehat{x}_i}n^{1-\omega_c}}, %
    \end{equation}
    where $\widehat{x}_i=\widehat{x}_i(t,\alpha)$ from \eqref{eq:defshortrange}. Note that \(\chi_i\) is exponentially small if \(n^{3\omega_c/2}\le \abs{i}\le n\) by rigidity~\eqref{eq:alphrig}, the fact that \(\gamma_i^z \sim i/n\), for \(n^{3\omega_c/2}\le \abs{i}\le 10c_2 n\), and $|\widehat{x}_i|\gtrsim 1$ for $\abs{i}> 10c_2 n$. We remark that the lower bound \(n^{3\omega_c/2}\) on \(\abs{i}\) is arbitrary, since \(\chi_i\) is exponentially small for any \(\abs{i}\) much bigger than \(n^{\omega_c}\). Moreover, as a consequence of~\eqref{eq:alphrig} we have that
    \begin{equation}\label{eq:rigg}
        \widehat{x}_i \sim \frac{i}{n} \quad \mbox{for}\quad n^\xi \le \abs{i}\le n, %
    \end{equation}
    with very high probability for any \(\xi>0\).

    By~\eqref{eq:ringmat}--\eqref{eq:circbm}, \eqref{eq:defshortrange}, and \eqref{eq:ubo}, using that $|\widehat{x}_i|=\widehat{x}_{|i|}$ by the symmetry of the spectrum, it follows that
    \begin{equation}\label{eq:dercv}
        \begin{split}
            \dif \norm{ {\bm v}{\bm \chi}}_2^2&=\dif\sum_{\abs{i}\le n} v_i^2 \chi_i^2=-2\sum_i \chi_i^2 v_i(Bv)_i\dif t+\frac{1}{n}\sum_{(i,j)\in\mathcal{A}} \frac{\chi_i^2 v_i\mathring{\Lambda}_{ij}}{\widehat{x}_i-\widehat{x}_j}\dif t\\
            &\quad -2n^{1-\omega_c} \sum_iv_i^2\chi_i^2\dif \widehat{x}_{|i|}+\frac{n^{2-2\omega_c}}{n}\sum_i (1+\mathring{\Lambda}_{ii})v_i^2\chi_i^2\dif t\\
            &= -\sum_{(i,j)\in\mathcal{A}} B_{ij}(v_i\chi_i-v_j\chi_j)^2\dif t+\frac{1}{2n}\sum_{(i,j)\in\mathcal{A}} \frac{(v_i\chi_i-v_j\chi_j)\mathring{\Lambda}_{ij}}{\widehat{x}_i-\widehat{x}_j} \chi_i\dif t \\
            &\quad +\sum_{(i,j)\in\mathcal{A}} B_{ij} v_i v_j(\chi_i-\chi_j)^2\dif t+\frac{1}{2n}\sum_{(i,j)\in\mathcal{A}} \frac{(\chi_i-\chi_j)\mathring{\Lambda}_{ij}}{\widehat{x}_i-\widehat{x}_j} v_j\chi_j\dif t \\
            &\quad -2n^{1-\omega_c} \sum_iv_i^2\chi_i^2\left(\frac{\dif \mathring{b}_{|i|}^z}{\sqrt{n}}+\frac{1}{2n}\sum_{j:(|i|,j)\notin \mathcal{A}}\frac{1}{x_{|i|}^z(t,0)-x_j^z(t,0)}\dif t\right)\\
            &\quad+\frac{n^{2-2\omega_c}}{n}\sum_i (1+\mathring{\Lambda}_{ii})v_i^2\chi_i^2\dif t -\frac{n^{1-\omega_c}}{n}\sum_{(|i|,j)\in\mathcal{A}}\frac{v_i^2\chi_i^2}{\widehat{x}_{|i|}-\widehat{x}_j}\big(1+\alpha \mathring{\Lambda}_{|i|,j}\big)\dif t,
        \end{split}
    \end{equation}
    where, in order to symmetrize the sums, we used that the operator \(B\) and the set \(\mathcal{A}\) are symmetric, i.e.\ \(B_{ij}=B_{ji}\) (see~\eqref{eq:kern}) and \((i,j)\in\mathcal{A}\Leftrightarrow (j,i)\in\mathcal{A}\), and that \(\mathring{\Lambda}_{ij}=\mathring{\Lambda}_{ji}\). In the remainder of the proof we will often use that for $|i-j|\le n^{\omega_L}$, by \eqref{eq:lipb}--\eqref{eq:lipbb}, we have
    \begin{equation}
        \label{eq:usefb}
        |\chi_i-\chi_j|\lesssim e^{-(|\widehat{x}_i|+|\widehat{x}_j|)n^{1-\omega_c}}\big||\widehat{x}_i|-|\widehat{x}_j|\big|n^{1-\omega_c},
    \end{equation}
    which follows by the rigidity~\eqref{eq:rigg} for $|i|,|j|\le 10c_2n$ and from the fact that $|\widehat{x}_i|+|\widehat{x}_j|\sim |\widehat{x}_i|\wedge|\widehat{x}_j|$ for $|i|,|j|>10c_2n$.

    We start estimating the terms in the second line of the r.h.s.\ of~\eqref{eq:dercv}. The most critical term is the first one because of the \((\widehat{x}_i-\widehat{x}_j)^{-2}\) singularity of \(B_{ij}\). We write this term as
    \begin{equation}\label{eq:split}
        \sum_{(i,j)\in\mathcal{A}} B_{ij} v_i v_j(\chi_i-\chi_j)^2=\left(\sum_{\substack{(i,j)\in\mathcal{A}, \\ \abs{i-j}\le n^{\omega_L}}}+\sum_{\substack{(i,j)\in\mathcal{A}, \\ \abs{i-j}>n^{\omega_L}}} \right) B_{ij} v_i v_j(\chi_i-\chi_j)^2.
    \end{equation}
    Then, using~\eqref{eq:lipbb}, \(\norm{ {\bm v}}_\infty\lesssim 1\) by~\eqref{eq:infb}, \(\abs{\mathring{\Lambda}_{ij}}\le n^{-\omega_E}\) by~\eqref{eq:ringlamthe}, the bound~\eqref{eq:usefb}, and that \(\omega_L\ll \omega_c\), we bound the first sum by
    \begin{equation}\label{eq:A2}
        \begin{split}
            &\abs*{\sum_{\substack{(i,j)\in\mathcal{A}, \\ \abs{i-j}\le n^{\omega_L}}} B_{ij} v_i v_j(\chi_i-\chi_j)^2}\\
            &\quad\lesssim \frac{1}{n}\sum_{\substack{(i,j)\in\mathcal{A}, \\ \abs{i-j}\le n^{\omega_L}}}  \frac{1+\abs{\mathring{\Lambda}_{ij}}}{(\widehat{x}_i-\widehat{x}_j)^2} \abs{v_i v_j} \frac{n^2\abs{\widehat{x}_i-\widehat{x}_j}^2}{n^{2\omega_c}}e^{-2(\abs{\widehat{x}_i}+\abs{\widehat{x}_j})n^{1-\omega_c}} \\
            &\quad\lesssim n^{1-2\omega_c}\left(\sum_{\abs{i},\abs{j}\le n^{3\omega_c/2}}+\sum_{\substack{\abs{i} \le n^{3\omega_c/2},\abs{j}\ge n^{3\omega_c/2}, \\\abs{i-j}\le n^{\omega_L}}}\right)\abs{v_i}\abs{v_j}e^{-2(\abs{\widehat{x}_i}+\abs{\widehat{x}_j})n^{1-\omega_c}} \\
            &\quad\lesssim n^{1- \omega_c/2} \norm{ {\bm v}{\bm \chi} }_2^2+e^{-\frac{1}{2}n^{\omega_c/2}},
        \end{split}
    \end{equation}
    with very high probability.

    Define the set
    \[
        \mathcal{A}_1:= \set*{(i,j)\given \abs{i},\abs{j}\ge 5c_2n}\cap \set*{(i,j)\given\abs{i-j}> n^{\omega_L}}=\mathcal{A}\cap \set*{(i,j)\given\abs{i-j}> n^{\omega_L}},
    \]
    which is symmetric. The second sum in~\eqref{eq:split}, using~\eqref{eq:lipb},~\eqref{eq:infb}, is bounded by
    \begin{equation}\label{eq:A21}
        \abs*{\sum_{(i,j)\in\mathcal{A}_1} B_{ij} v_i v_j(\chi_i-\chi_j)^2}\lesssim n^{1-2\omega_c}\sum_{(i,j)\in\mathcal{A}_1}e^{-2(\abs{\widehat{x}_i}\wedge\abs{\widehat{x}_j})n^{1-\omega_c}}\le e^{-n^{1-\omega_c}/2},
    \end{equation}
    with very high probability.

    Next, we consider the second term in the second line of the r.h.s.\ of~\eqref{eq:dercv}. Using~\eqref{eq:lipbb}, and that \(\abs{\mathring{\Lambda}_{ij}}\le n^{-\omega_E}\), proceeding similarly to~\eqref{eq:A2}--\eqref{eq:A21}, we bound this term as
    \begin{equation}
        \label{eq:A3}
        \begin{split}
            \abs*{\frac{1}{n}\sum_{(i,j)\in\mathcal{A}} \frac{(\chi_i-\chi_j)\mathring{\Lambda}_{ij}}{\widehat{x}_i-\widehat{x}_j} v_j\chi_j } &\lesssim \abs*{\frac{1}{n}\sum_{\substack{(i,j)\in\mathcal{A} \\ \abs{i-j}\le n^{\omega_L}}} \frac{(\chi_i-\chi_j)\mathring{\Lambda}_{ij}}{\widehat{x}_i-\widehat{x}_j} v_j\chi_j }+\abs*{\frac{1}{n}\sum_{(i,j)\in\mathcal{A}_1} \frac{(\chi_i-\chi_j)\mathring{\Lambda}_{ij}}{\widehat{x}_i-\widehat{x}_j} v_j\chi_j } \\
            &\lesssim \sum_{\substack{(i,j)\in\mathcal{A} \\ \abs{i-j}\le n^{\omega_L}}} \frac{\abs{\mathring{\Lambda}_{ij}}}{\abs{\widehat{x}_i-\widehat{x}_j}} \frac{\abs{\widehat{x}_i-\widehat{x}_j}}{n^{\omega_c}}\abs{v_j}\chi_j e^{-(\abs{\widehat{x}_i}+\abs{\widehat{x}_j})n^{1-\omega_c}}+e^{-n/2} \\
            &\lesssim \frac{1}{n^{\omega_c+\omega_E}} \sum_{\abs{i},\abs{j}\le n^{3\omega_c/2}}\abs{v_j}\chi_j+e^{-\frac{1}{2}n^{\omega_c/2}} \\
            &\lesssim \frac{1}{n^{\omega_c/4+\omega_E}} \norm{ {\bm v}{\bm \chi}}_2+e^{-\frac{1}{2}n^{\omega_c/2}},
        \end{split}
    \end{equation}
    with very high probability uniformly in \(0\le t\le t_f\).

    We now consider the first line in the r.h.s.\ of~\eqref{eq:dercv}. Since \(1+\alpha\mathring{\Lambda}_{ij}\ge 1/2\), we conclude that
    \begin{equation}
        \label{eq:A1}
        \begin{split}
            \abs*{\frac{1}{n}\sum_{(i,j)\in\mathcal{A}} \frac{(v_i\chi_i-v_j\chi_j)\mathring{\Lambda}_{ij}}{\widehat{x}_i-\widehat{x}_j} \chi_i}&\le \frac{1}{C}\sum_{(i,j)\in\mathcal{A}} B_{ij}(v_i\chi_i-v_j\chi_j)^2+\frac{C}{n}\sum_{(i,j)\in\mathcal{A}} \abs{\mathring{\Lambda}_{ij}}^2 \chi_i^2\\
            &\le \frac{1}{C}\sum_{(i,j)\in\mathcal{A}} B_{ij}(v_i\chi_i-v_j\chi_j)^2+ \frac{C}{n}\sum_{\abs{i},\abs{j}\le n^{3\omega_c/2}} \abs{\mathring{\Lambda}_{ij}}^2 \chi_i^2+e^{-\frac{1}{2}n^{\omega_c/2}} \\
            &\le  \frac{1}{C}\sum_{(i,j)\in\mathcal{A}} B_{ij}(v_i\chi_i-v_j\chi_j)^2+\frac{n^{3\omega_c}}{n^{1+2\omega_E}},
        \end{split}
    \end{equation}
    for some large \(C>0\). The error term in the r.h.s.\ of~\eqref{eq:A1} is affordable since \(\omega_c\ll \omega_E\).

    We are now left with the estimate of the terms in the last two lines of \eqref{eq:dercv}. It is easy to see that the second term in the penultimate line and the first term in the last line of \eqref{eq:dercv} are estimated by $n^{1-\omega_c}\norm{ {\bm v}{\bm \chi}}_2^2$. For the stochastic term by the martingale inequality \cite[Appendix B, Eq. (18)]{doi:10.1137/1.9780898719017}, with $c=0$ for continuous martingales, by \eqref{eq:ringmat}--\eqref{eq:circbm} and the a priori bound \eqref{eq:infb}, we have
    \begin{equation}
        \label{eq:estmart}
        \sup_{0\le t\le t_f}\left|\int_0^tn^{1-\omega_c}\sum_iv_i^2\chi_i^2\frac{\dif \mathring{b}_{|i|}^z}{\sqrt{n}}\right|\lesssim n^{\omega_f/2-\omega_c} \sup_{0\le t\le t_f}\norm{ {\bm v}{\bm \chi}}_2+n^{\omega_f-2\omega_c-\omega_E} \sup_{0\le t\le t_f}\norm{ {\bm v}{\bm \chi}}_2^2
    \end{equation}
    with very high probability for any small $\xi>0$. We now estimate the last term in \eqref{eq:dercv}. Using that $\widehat{x}_{-i}=-\widehat{x}_i$, $v_{-i}=-v_i$, and $\chi_{-i}=\chi_i$, we start by rewriting and estimating it as
    \begin{equation}
        \label{eq:finbhop}
        \begin{split}
            &-n^{-\omega_c}\sum_{(i,j)\in\mathcal{A}, \atop i\ne j,\, i,j>0} \big(1+\alpha\mathring{\Lambda}_{ij}\big)\left[\frac{v_i^2\chi_i^2-v_j^2\chi_j^2}{\widehat{x}_i-\widehat{x}_j}+\frac{v_i^2\chi_i^2+v_j^2\chi_j^2}{\widehat{x}_i+\widehat{x}_j}+\frac{v_i^2\chi_i^2}{2\widehat{x}_i}\bm1(j=i)\right] \\
            &\qquad\quad\le - n^{-\omega_c}\sum_{(i,j)\in\mathcal{A}, \atop i,j>0} \big(1+\alpha\mathring{\Lambda}_{ij}\big)\frac{(v_i\chi_i-v_j\chi_j)(v_i\chi_i+v_j\chi_j)}{\widehat{x}_i-\widehat{x}_j} \\
            &\qquad\quad\le\frac{1}{100}\sum_{(i,j)\in\mathcal{A}, \atop i,j>0}B_{ij}(v_i\chi_i-v_j\chi_j)^2+100n^{1-2\omega_c} \norm{ {\bm v}{\bm \chi}}_2^2 \\
            &\qquad\quad\le\frac{1}{100}\sum_{(i,j)\in\mathcal{A}}B_{ij}(v_i\chi_i-v_j\chi_j)^2+100n^{1-2\omega_c} \norm{ {\bm v}{\bm \chi}}_2^2
        \end{split}
    \end{equation}
    where in the first inequality we used that $\widehat{x}_j\ge0$ for $j>0$ and $1+\alpha\mathring{\Lambda}_{ij}>0$, in the second inequality a Schwarz (recall the definition of $B_{ij}$ from \eqref{eq:kern}), and in the last inequality we dropped the restriction $i,j>0$, using that $B_{ij}>0$. Note that the first term in the rhs. of \eqref{eq:finbhop} can be incorporated in the first negative term in the rhs. of \eqref{eq:dercv}.

    Hence, combining~\eqref{eq:dercv}--\eqref{eq:finbhop}, integrating~\eqref{eq:dercv} from \(0\) to \(t_f=n^{-1+\omega_f}\), ignoring the negative term in the third line of \eqref{eq:dercv},  and using that \(n^{1-\omega_c/2}t_f=n^{\omega_f-\omega_c/2}\) with \(\omega_f\ll \omega_c\ll\omega_E\), we get
    \[
        \sup_{0\le t\le t_f}\norm{ {\bm v}{\bm \chi}}_2^2\le \frac{n^{3\omega_c}t_f}{n^{1+2\omega_E}}.
    \]
    Hence, using the bound
    \[
        \sup_{0\le t\le t_f}\sup_{\abs{i}\le n^{\widehat{\omega}}} \abs{v_i(t)}\le \sup_{0\le t\le t_f}\norm{ {\bm v}{\bm \chi}}_2\le \sqrt{\frac{n^{3\omega_c}t_f}{n^{1+2\omega_E}}},
    \]
    we conclude~\eqref{eq:optvb} for some \(\omega, \widehat{\omega}>0\) such that \(\widehat{\omega}\ll \omega\ll \omega_f\ll\omega_L\ll\omega_c\ll \omega_E\).
\end{proof}

With this proof we completed the main~\ref{step1} in the proof of Proposition~\ref{prop:realprop}, the analysis of the interpolation process \({\bm x}^z(t,\alpha)\).

\subsubsection{The processes \texorpdfstring{\({\bm \lambda}(t)\)}{l(t)} and \texorpdfstring{\({\bm x}^z(t,1)\)}{xz(t,1)} are close}\label{sec:dddaah}
In~\ref{step2} towards the proof of Proposition~\ref{prop:realprop}, we now prove that the processes \({\bm \lambda}(t)\) and \({\bm x}^z(t,1)\) are very close for any \(t\in [0,t_f]\):

\begin{lemma}\label{lem:comp}
    Let \({\bm \lambda}^z(t)\), \({\bm x}^z(t,1)\) be defined in~\eqref{eq:impnewDBM} and~\eqref{eq:defintpro}, respectively, and let \(t_f=  n^{-1+\omega_f}\), then
    \begin{equation}
        \label{eq:closappr}
        \sup_{\abs{i}\le n}\sup_{0\le t\le t_f} \abs{x_i^z(t,1)-\lambda_i^z(t)}\lesssim \frac{n^{\omega_f}}{n^{1+\omega_r}}.
    \end{equation}
    with very high probability.
\end{lemma}

\begin{proof}[Proof of Proposition~\ref{prop:realprop}]
    Proposition~\ref{prop:realprop} follows by exactly the same computations as in~\cite[Section (4.10)]{MR4009717}, combining~\eqref{eq:closappr},~\eqref{eq:shortlong},~\eqref{eq:infb}--\eqref{eq:optvb}.
\end{proof}

\begin{proof}[Proof of Lemma~\ref{lem:comp}]
    The proof of this lemma closely follows~\cite[Lemma 4.2]{MR4009717}. We remark that in our case \(\dif M_i=Z_i=0\) compared to~\cite[Lemma 4.2]{MR4009717}, using the notation therein. Recall the definitions of \(C(t),\Lambda_{ij}^{z_l}(t), \Theta_{ij}^{z_1,z_2}(t), \Theta_{ij}^{z_1,\overline{z}_2}(t)\) and \(\mathring{C}(t),\mathring{\Lambda}_{ij}^{z_l}(t), \mathring{\Theta}_{ij}^{z_1,z_2}(t), \mathring{\Theta}_{ij}^{z_1,\overline{z}_2}(t)\) in~\eqref{eq:defCC},~\eqref{eq:defbiglam},\eqref{eq:deflamzz} and~\eqref{eq:ringmat}--\eqref{eq:ringlamthe}, respectively. In the following we may omit the \(z\)-dependence. Introduce the stopping times
    \begin{align}\label{eq:stoptime1}
        \tau_1 & := \inf\set*{t\ge 0\given \exists \abs{i},\abs{j}\le n;  l\in [2] \text{ s.t. } \abs{\Lambda_{ij}^{z_l}(t)}+\abs{\Theta_{ij}^{z_1,z_2}(t)}+\abs{\Theta_{ij}^{z_1,\overline{z}_2}(t)}> n^{-\omega_E}}, \\\label{eq:stoptime2}
        \tau_2 & := \inf\set*{t\ge 0\given \exists \abs{i}\le n \text{ s.t.\ } \abs{x_i(t,1)}+\abs{\lambda_i(t)}>2R},
    \end{align}
    for some large \(R>0\), and
    \begin{equation}
        \label{eq:taudef}
        \tau:=\tau_1\wedge \tau_2\wedge t_f.
    \end{equation}
    Note that \(\abs{\lambda_i(t)}\le R\) with very high probability, since \({\bm \lambda}(t)\) are the eigenvalues of \(H_t^z\), whose norm is typically bounded. Furthermore, by~\eqref{eq:glo} and the fact that the process \({\bm x}(t,0)\) stays bounded by~\cite[Section 3]{MR4009708} it follows that \(\abs{x_i(t,\alpha)}\le R\) for any \(t\in [0,t_f]\) and \(\alpha \in [0,1]\). We remark that the analysis in~\cite[Section 3]{MR4009708} is done for a process of the form~\eqref{eq:defintpro}, with \(\alpha=0\), when it has i.i.d.\ driving Brownian motions, but the same results apply for our case as well since the correlation in~\eqref{eq:ringmat} does not play any role (see~\eqref{eq:term3}). This, together with Lemma~\ref{lem:ee} applied for \(z=z_1, z'=z_2\) and \(z=z_1, z'=\overline{z}_2\) and \(z=z_l, z'=\overline{z}_l\), implies that
    \[
        \tau=t_f
    \]
    with very high probability. In particular, \(\mathring{\Theta}_{ij}(t)=\Theta_{ij}(t)\) for any \(t\le \tau\),
    hence
    \begin{equation}
        \label{eq:eqbigC}
        C(t)=\mathring{C}(t)
    \end{equation}
    for any \(t\le \tau\).

    In the remainder of the proof, omitting the time- and \(z\)-dependence, we use the notation \({\bm x}={\bm x}^z(t,1)\), \({\bm \lambda}={\bm \lambda}(t)\). Define
    \[
        u_i:=  \lambda_i-x_i, \qquad \abs{i}\le n,
    \]
    then, as a consequence of~\eqref{eq:eqbigC}, subtracting~\eqref{eq:impnewDBM} and~\eqref{eq:defintpro}, it follows that
    \begin{equation}
        \dif u_i=\sum_{j\ne i}B_{ij}(u_j-u_i) \dif t+\frac{A_n}{\sqrt{n}}\dif b_i,
    \end{equation}
    for any \(0\le t\le\tau\), where
    \begin{equation}
        \label{eq:ker}
        B_{ij}=\frac{1+\Lambda_{ij}}{2n(\lambda_i-\lambda_j)(x_i-x_j)}>0,
    \end{equation}
    since \(\abs{\Lambda_{ij}(t)}=\abs{\mathring{\Lambda}_{ij}(t)}\le n^{-\omega_E}\), and
    \begin{equation}
        \label{eq:err}
        A_n=\frac{1}{\sqrt{1+n^{-\omega_r}}}-1=\mathcal{O}(n^{-\omega_r}).
    \end{equation}
    Let \(\nu:=  n^{1+\omega_r}\), and define the Lyapunov function
    \begin{equation}
        \label{eq:defF}
        F(t):=  \frac{1}{\nu}\log \left(\sum_{\abs{i}\le n}e^{\nu u_i(t)}\right).
    \end{equation}

    By It\^{o}'s lemma, for any \(0\le t \le \tau\), we have that
    \begin{equation}
        \label{eq:itoF}
        \begin{split}
            \dif F&= \frac{1}{\sum_{\abs{i}\le n} e^{\nu u_i}}\sum_{\abs{i}\le n} e^{\nu u_i}\sum_{j\ne i} B_{ij}(u_j-u_i) \dif t+\frac{n^{-1/2}A_n}{\sum_{\abs{i}\le n} e^{\nu u_i}}\sum_{\abs{i}\le n} e^{\nu u_i} \dif b_i \\
            &\quad +\frac{n^{-1}\nu A_n^2}{4\sum_{\abs{i}\le n} e^{\nu u_i}}\sum_{\abs{i}\le n} e^{\nu u_i} (1+\Lambda_{ii}) \dif t-\frac{4n^{-1}\nu A_n^2}{\left(\sum_{\abs{i}\le n} e^{\nu u_i}\right)^2}\sum_{\abs{i}, \abs{j}\le n} e^{\nu u_i}e^{\nu u_j} \Exp*{ \dif b_i \dif b_j\given \widetilde{\mathcal{F}}_{b,t}}.
        \end{split}
    \end{equation}
    Note that the first term in the r.h.s.\ of~\eqref{eq:itoF} is negative since the map \(x\mapsto e^{\nu x}\) is increasing. The second and third term in the r.h.s.\ of~\eqref{eq:itoF}, using that \(1+\Lambda_{ii}\le 2\), are bounded exactly as in~\cite[Eqs.~(4.37)--(4.38)]{MR4009717} by
    \[
        \frac{n^\xi t_f^{1/2}}{n^{1/2+\omega_r}}+\frac{t_f\nu}{n^{1+2\omega_r}},
    \]
    with very high probability for any \(\xi>0\).

    Note that
    \[
        \sum_{\abs{i}, \abs{j}\le n} e^{\nu u_i}e^{\nu u_j}\Exp*{ \dif b_i \dif b_j\given \widetilde{\mathcal{F}}_{b,t}} \ge 0,
    \]
    hence, the last term in the r.h.s.\ of~\eqref{eq:itoF} is always non positive. This implies that
    \[
        \sup_{0\le t\le t_f} F(t)\le F(0)+\frac{t_f\nu A_n^2}{n}+\frac{n^\xi t_f^{1/2}A_n}{n^{1/2}},
    \]
    for any \(\xi>0\). Then, since
    \[
        F(0)=\frac{\log (2n)}{n^{1+\omega_r}}, \qquad F(t)\ge \sup_{\abs{i}\le n} u_i(t),
    \]
    we conclude the upper bound in~\eqref{eq:closappr}. Then noticing that \(u_{-i}=-u_i\) for \(i\in [n]\), we conclude the lower bound as well.
\end{proof}

\subsection{Path-wise coupling close to zero: Proof of Lemmata~\ref{lem:firststepmason}--\ref{lem:secondstepmason}}\label{sec:hos}

This section is the main technical result used in the proof of Lemmata~\ref{lem:firststepmason}--\ref{lem:secondstepmason}. In Proposition~\ref{pro:ciala} we will show that the points with small indices in the two processes become very close to each other on a certain time scale \(t_f=n^{-1+\omega_f}\), for any small \(\omega_f>0\).

The main result of this section (Proposition~\ref{pro:ciala}) is stated for general deterministic initial data \({\bm s(0)}\) satisfying a certain regularity condition (see Definition~\ref{eq:defregpro}  later) even if for its applications in the proof of Proposition~\ref{prop:indeig} we only consider initial data which are eigenvalues of i.i.d.\ random matrices. The initial data \({\bm r}(0)\), without loss of generality, are assumed to be the singular values of a Ginibre matrix (see also below~\eqref{eq:mupr} for a more detailed explanation). For notational convenience we formulate the result for two general processes \({\bm s}\) and \({\bm r}\) and later we specialize them to our
application.

Fix a small constant \(0<\omega_r\ll 1\), and define the processes \(s_i(t)\), \(r_i(t)\) to be the solution of
\begin{equation}
    \label{eq:lambdapr}
    \dif s_i(t)=\sqrt{\frac{1}{2 n(1+n^{-\omega_r})}}\dif \mathfrak{b}^s_i(t)+\frac{1}{2n}\sum_{j\ne  i} \frac{1}{s_i(t)-s_j(t)} \dif t, \qquad 1\le \abs{i}\le n,
\end{equation}
and
\begin{equation}
    \label{eq:mupr}
    \dif r_i(t)=\sqrt{\frac{1}{2 n(1+n^{-\omega_r})}}\dif \mathfrak{b}^r_i(t)+\frac{1}{2n}\sum_{j\ne  i} \frac{1}{r_i(t)-r_j(t)} \dif t, \qquad 1\le \abs{i}\le n,
\end{equation}
with initial data \(s_i(0)=s_i\), \(r_i(0)=r_i\), where \({\bm s}=\{s_{\pm i}\}_{i\in [n]}\) and \({\bm r}=\{r_{\pm i}\}_{i\in  [n]}\) are two independent sets of particles such that \(s_{-i}=-s_i\) and \(r_{-i}=-r_i\) for \(i\in [n]\). The driving martingales \(\{\mathfrak{b}^s_i\}_{i\in [n]}\), \(\{\mathfrak{b}^r_i\}_{i\in [n]}\) in~\eqref{eq:lambdapr}--\eqref{eq:mupr} are two families satisfying Assumption~\ref{ass:close} below, and they are such that \(\mathfrak{b}^s_{-i}=-\mathfrak{b}^s_i\), \(\mathfrak{b}^r_{-i}=-\mathfrak{b}^r_i\) for \(i\in [n]\). The coefficient \((1+n^{-\omega_r})^{-1/2}\) ensures the well-posedness of the processes~\eqref{eq:lambdapr}--\eqref{eq:mupr} (see Appendix~\ref{sec:wp}), but it does not play any role in the proof of Proposition~\ref{pro:ciala} below.

For convenience we also assume that \(\{r_{\pm i}\}_{i=1}^n\) are the singular values of \(\widetilde{X}\), with \(\widetilde{X}\) a Ginibre matrix. This is not a restriction; indeed, once a process with general initial data \({\bm s}\) is shown to be close to the reference process with Ginibre initial data, then processes with any two initial data will be close.

On the correlation structure between the two families of i.i.d.\ Brownian motions \(\{\mathfrak{b}^s_i\}_{i=1}^n\), \(\{\mathfrak{b}^r_i\}_{i=1}^n\) and the initial data \(\{s_{\pm i}\}_{i\in[n]}\) we make the following assumptions.

\begin{assumption}\label{ass:close}
    Fix \(\omega_K,\omega_Q>0\) such that \(\omega_K\ll \omega_r\ll \omega_Q\ll 1\), with \(\omega_r\) defined in~\eqref{eq:lambdapr}--\eqref{eq:mupr}, and define  the \(n\)-dependent parameter \(K=K_n=n^{\omega_K}\). Suppose that the families \(\{\mathfrak{b}^s_{\pm i}\}_{i=1}^n\), \(\{\mathfrak{b}^r_{\pm i}\}_{i=1}^n\) in~\eqref{eq:lambdapr}--\eqref{eq:mupr} are realised on a common probability space with a common filtration \(\mathcal{F}_t\). Let
    \begin{equation}\label{eq:defL}
        L_{ij}(t) \dif t:= \Exp*{\bigl(\dif \mathfrak{b}^s_i(t)-\dif \mathfrak{b}^r_i(t)\bigr) \bigl(\dif \mathfrak{b}^s_j(t)-\dif \mathfrak{b}^r_j(t)\bigr)\given \mathcal{F}_t}
    \end{equation}
    denote the covariance of the increments conditioned on \(\mathcal{F}_t\). The processes satisfy the following assumptions:
    \begin{enumerate}[label= (\alph*)]
        \item\label{close1}The two families of martingales \(\{ \mathfrak{b}^s_i\}_{i=1}^n\), \(\{ \mathfrak{b}^r_i\}_{i=1}^n\) are such that
        \begin{equation}\label{eq:newbyet}
            \Exp*{\dif \mathfrak{b}^{q_1}_i(t)\dif \mathfrak{b}^{q_2}_j(t)\given \mathcal{F}_t}=\bigl[\delta_{ij}\delta_{q_1q_2}+\Xi_{ij}^{q_1,q_2}(t)\bigr]\dif t, \quad \abs{\Xi_{ij}^{q_1,q_2}(t)}\le n^{-\omega_Q},
        \end{equation}
        for any \(i,j\in [n]\), \(q_1,q_2\in\{s,r\}\). The quantities in~\eqref{eq:newbyet} for negative \(i,j\)-indices are defined by symmetry.
        \item\label{close3} The subfamilies \(\{\mathfrak{b}^s_{\pm i}\}_{i=1}^K\), \(\{\mathfrak{b}^r_{\pm i}\}_{i=1}^K\) are very strongly dependent  in the sense that for any \(\abs{i}, \abs{j}\le K\) it holds
        \begin{equation}\label{eq:assbqv}
            \abs{L_{ij}(t)}\le n^{-\omega_Q}
        \end{equation}
        with very high probability for any fixed \(t\ge 0\).
    \end{enumerate}
\end{assumption}

\begin{definition}[{\((g,G)\)-regular points~\cite[Definition~\ref{cplx-eq:defregpro}]{1912.04100}}]\label{eq:defregpro}
    Fix a very small \(\nu>0\), and choose \(g\), \(G\) such that
    \[
        n^{-1+\nu}\le g\le n^{-2\nu}, \qquad G\le n^{-\nu}.
    \]
    A set of \(2n\)-points \({\bm s}=\{s_i\}_{\abs{i}\le n}\) on \(\R \) is called \((g,G)\)-\emph{regular} if there exist constants \(c_\nu,C_\nu>0\) such that
    \begin{equation}
        \label{eq:upbv}
        c_\nu \le \frac{1}{2n}\Im \sum_{i=-n}^n \frac{1}{s_i-(E+\ii \eta)}\le C_\nu,
    \end{equation}
    for any \(\abs{E}\le G\), \(\eta \in [g, 10]\), and if there is a constant \(C_s\) large enough such that \(\norm{ {\bm s}}_\infty\le n^{C_s}\). Moreover, \(c_\nu,C_\nu\sim 1\) if \(\eta \in [g, n^{-2\nu}]  \) and \(c_\nu\ge n^{-100\nu}\), \(C_\nu\le n^{100\nu}\) if
    \(\eta\in \interval{oc}{n^{-2\nu}, 10}\).
\end{definition}

Let \(\rho_{\mathrm{fc},t}(E)\) be the scDOS of the particles \(\{s_{\pm i}(t)\}_{i\in [n]}\) that is given by the semicircular flow acting on the scDOS of the initial data \(\{s_{\pm i}(0)\}_{i\in [n]}\), see~\cite[Eqs.~(2.5)--(2.6)]{MR3914908}.

\begin{proposition}[Path-wise coupling close to zero]\label{pro:ciala}
    Let the processes \({\bm s}(t)=\{s_{\pm i}(t)\}_{i\in [n]}\), \({\bm r}(t)=\{r_{\pm i}(t)\}_{i\in [n]}\) be the solutions of~\eqref{eq:lambdapr} and~\eqref{eq:mupr}, respectively, and assume that the driving martingales in~\eqref{eq:lambdapr}--\eqref{eq:mupr} satisfy Assumption~\ref{ass:close} for some \(\omega_K,
    \omega_Q>0\). Additionally, assume that \({\bm s}(0)\) is \((g,G)\)-regular in the sense of Definition~\ref{eq:defregpro} and that \({\bm r}(0)\) are
    the singular values of a Ginibre matrix. Then for any small \(\omega_f,\nu>0\) such that \(\nu\ll \omega_K\ll \omega_f\ll \omega_Q\) and that \(gn^\nu\le t_f\le n^{-\nu}G^2\), there exist constants \(\omega, \widehat{\omega}>0\) such that \(\nu\ll \widehat{\omega}\ll\omega\ll\omega_f\), and
    \begin{equation}
        \label{eq:hihihi}
        \abs*{ \rho_{\mathrm{fc},t_1}(0) s_i(t_f)- \rho_{\mathrm{sc}}(0) r_i(t_f)}\le n^{-1-\omega}, \qquad \abs{i}\le n^{\widehat{\omega}},
    \end{equation}
    with very high probability, where \(t_f:= n^{-1+\omega_f}\).
\end{proposition}
\begin{proof}
    The proof of Proposition~\ref{pro:ciala} is nearly identical to the proof of~\cite[Proposition~\ref{cplx-pro:ciala}]{1912.04100}, which itself follows the proof of fixed energy universality in~\cite{MR3541852, MR3914908}, adapted to the block structure~\eqref{eq:herher} in~\cite{MR3916329} (see also~\cite{1812.10376} for a different technique to prove universality, adapted to the block structure in~\cite{1912.05473}). We will not repeat the whole proof, just explain the modification. The only difference of Proposition~\ref{pro:ciala} compared to~\cite[Proposition~\ref{cplx-pro:ciala}]{1912.04100} is that here we allow the driving martingales in~\eqref{eq:lambdapr}--\eqref{eq:mupr} to have a (small) correlation (compare Assumption~\ref{ass:close} with a non zero \(\Xi_{ij}^{q_1,q_2}\) to~\cite[Assumption~\ref{cplx-ass:close}]{1912.04100}). The additional pre-factor \((1+n^{-\omega_r})^{-1/2}\) does not play any role.

    The correlation of the driving martingales in~\eqref{eq:lambdapr}--\eqref{eq:mupr} causes a difference in the estimate of~\cite[Eq.~\eqref{cplx-eq:replace2}]{1912.04100}. In particular, the bound on
    \begin{equation}
        \label{eq:newma}
        \dif M_t=\frac{1}{2n}\sum_{\abs{i}\le n}(w_i-f_i) f_i' \dif C_i(t,\alpha), \qquad \dif C_i(t,\alpha):=\frac{\alpha\dif \mathfrak{b}^s+(1-\alpha)\dif \mathfrak{b}^r}{\sqrt{2n(1+n^{-\omega_r})}},
    \end{equation}
    using the notation in~\cite[Eq.~\eqref{cplx-eq:replace2}]{1912.04100}, will be slightly different. In the remainder of the proof we present how~\cite[Eqs.~\eqref{cplx-eq:replace2}--\eqref{cplx-eq:qM4}]{1912.04100} changes in the current setup. Using that by~\cite[Eqs.~(3.119)--(3.120)]{MR3914908} we have
    \begin{equation}
        \label{eq:ffsfinh}
        \abs{f_i}+\abs{f_i'}+\abs{w_i}\le n^{-D}, \qquad n^{\omega_A}< \abs{i}\le n,
    \end{equation}
    for \(\omega_A=\omega_K\) (with \(\omega_K\) defined in Assumption~\ref{ass:close}), and for any \(D>0\) with very high probability, we bound the quadratic variation of~\eqref{eq:newma} by
    \begin{equation}
        \label{eq:sperboh}
        \dif \braket{ M}_t=\frac{1}{4n^2}\sum_{1\le \abs{i}, \abs{j}\le n^{\omega_A}} (w_i-f_i)(w_j-f_j) f_i' f_j' \Exp*{\dif C_i(\alpha,t)\dif C_j(\alpha,t)\given\mathcal{F}_t} +\mathcal{O}\left(n^{-100}\right).
    \end{equation}
    We remark that here we estimated the regime when \(\abs{i}\) or \(\abs{j}\) are larger than \(n^{\omega_A}\) differently compared to~\cite[Eq.~\eqref{cplx-eq:qM1}]{1912.04100}, since, unlike in~\cite[Eq.~\eqref{cplx-eq:qM1}]{1912.04100}, \(\Exp*{\dif C_i(t,\alpha)\dif C_j(t,\alpha)\given\mathcal{F}_t}\ne \delta_{ij}\), hence here we anyway need to estimate the double sum using~\eqref{eq:ffsfinh}.

    Then, by~\ref{close1}--\ref{close3} of Assumption~\ref{ass:close}, for \(\abs{i}, \abs{j}\le n^{\omega_A}\) we have
    \begin{equation}
        \label{eq:qM2}
        \begin{split}
            \Exp*{\dif C_i(t,\alpha)\dif C_j(t,\alpha)\given\mathcal{F}_t}&= \frac{\delta_{ij}+\alpha^2\Xi_{ij}^{s,s}(t)+(1-\alpha)^2\Xi_{ij}^{r,r}(t)}{2n(1+n^{-\omega_r})} \dif t \\
            &\quad +\frac{\alpha(1-\alpha)}{2n(1+n^{-\omega_r})}\Exp*{\bigl(\dif \mathfrak{b}_i^s\dif \mathfrak{b}_j^r+\dif \mathfrak{b}_i^r\dif \mathfrak{b}_j^s\bigr)\given\mathcal{F}_t},
        \end{split}
    \end{equation}
    and that
    \begin{equation}
        \label{eq:qM3}
        \begin{split}
            \abs*{ \Exp*{\dif \mathfrak{b}_i^s\dif \mathfrak{b}_j^r \given \mathcal{F}_t} } &=\abs*{\Exp*{ (\dif \mathfrak{b}_i^s-\dif \mathfrak{b}_i^r)\dif \mathfrak{b}_j^r\given\mathcal{F}_t}+(\delta_{ij}+\Xi_{ij}^{r,r}(t)) \dif t} \\
            &\lesssim (\abs{L_{ii}(t)}^{1/2}+\abs{\Xi_{ij}^{r,r}(t)}+\delta_{ij} )\dif t,
        \end{split}
    \end{equation}
    where in the last step we used Kunita-Watanabe inequality for the quadratic variation \((\dif \mathfrak{b}_i^s-\dif \mathfrak{b}_i^r)\dif \mathfrak{b}_j^r\).

    Combining~\eqref{eq:sperboh}--\eqref{eq:qM3}, and adding back the sum over \(n^{\omega_A}<\abs{i}\le n\) of \((w_i-f_i)^2(f_i')^2\) at the price of an additional error \(\mathcal{O}(n^{-100})\), omitting the \(t\)-dependence, we finally conclude that
    \begin{equation}
        \label{eq:qM4}
        \begin{split}
            \dif \braket{ M}_t&\lesssim \frac{1}{n^3} \sum_{1\le \abs{i}\le n} (w_i-f_i)^2 (f_i')^2 \dif t \\
            &\quad + \frac{1}{n^3}\sum_{\abs{i}, \abs{j}\le n^{\omega_A}} \left(\abs{L_{ii}}^{1/2}+\abs{\Xi_{ij}^{s,s}}+\abs{\Xi_{ij}^{r,r}}\right) \abs*{ (w_i-f_i)(w_j-f_j) f_i' f_j' } \dif t+\mathcal{O}\left(n^{-100}\right).
        \end{split}
    \end{equation}
    Since \(\abs{L_{ii}}+\abs{\Xi_{ij}^{q_1,q_2}}\le n^{-\omega_Q}\), for any \(\abs{i},\abs{j}\le n\), \(q_1,q_2\in \{s,r\}\), and \(\omega_A=\omega_K\ll \omega_Q\) by~\eqref{eq:newbyet}--\eqref{eq:assbqv}, using Cauchy-Schwarz in~\eqref{eq:qM4}, we conclude that
    \begin{equation}
        \label{eq:qM5}
        \dif \braket{ M}_t\lesssim\frac{1}{n^3} \sum_{1\le \abs{i}\le n} (w_i-f_i)^2 (f_i')^2 \dif t+\mathcal{O}\left(n^{-100}\right),
    \end{equation}
    which is exactly the same bound as in~\cite[Eq.~\eqref{cplx-eq:qM5}]{1912.04100} (except for the tiny error \(\mathcal{O}(n^{-100})\) that is negligible).
    Proceeding exactly as in~\cite{1912.04100}, we conclude the proof of Proposition~\ref{pro:ciala}.
\end{proof}

\subsubsection{Proof of Lemma~\ref{lem:firststepmason} and  Lemma~\ref{lem:secondstepmason}}\label{sec:noncelpiu}

The fact that the processes \(\mathring{\bm \lambda}(t)\), \(\widetilde{\bm \lambda}(t)\) and \(\widetilde{\bm \mu}(t)\), \({\bm \mu}(t)\) satisfy the hypotheses of Proposition~\ref{pro:ciala} for the choices \(\nu=\omega_h\), \(\omega_K=\omega_A\), \(\omega_Q=\omega_E\), and \(\Xi_{ij}^{q_1,q_2}=\Theta_{ij}^{z_1,\overline{z}_2}\) follows by Lemma~\ref{lem:lambound} applied for \(z=z_1, z'=z_2\) and \(z=z_1, z'=\overline{z}_2\) and \(z=z_l, z'=\overline{z}_l\), and exactly the same computations as in~\cite[Section~\ref{cplx-sec:noncelpiu}]{1912.04100}. We remark that the processes \({\bm \mu}^{(l)}(t)\) do not have the additional coefficient \((1+n^{-\omega_r})\) in the driving Brownian motions, but this does not play any role in the application of Proposition~\ref{pro:ciala} since it causes an error term \(n^{-1-\omega_r}\) that is much smaller then the bound \(n^{-1-\omega}\) in~\eqref{eq:firshpb2}. Then, by Proposition~\ref{pro:ciala}, the results in Lemma~\ref{lem:firststepmason} and Lemma~\ref{lem:secondstepmason} immediately follow.\qed%

\subsection{Proof of Proposition~\ref{prop:wpmainDBM}}\label{sec:wpe}
First of all we notice that \({\bm \lambda}(t)\) is \(\gamma\)-H\"older continuous for any \(\gamma\in (0,1/2)\) by Weyl's inequality. Then the proof of Proposition~\ref{prop:wpmainDBM} consists of two main steps, (i) proving that the eigenvalues \({\bm \lambda}(t)\) are a strong solution of~\eqref{eq:impnewDBM} as long as there are no collisions, and (ii) proving that there are no collisions for almost all \(t\in [0,T]\).

The proof that the eigenvalues \({\bm \lambda}(t)\) are a solution of~\eqref{eq:impnewDBM} is deferred to Appendix~\ref{sec:derdbm}. The fact that there are no collisions for almost all \(t\in [0,T]\) is ensured by~\cite[Lemma 6.2]{1908.04060} following nearly the same computations as in~\cite[Theorem 5.2]{MR4009717} (see also~\cite[Theorem 6.3]{1908.04060} for its adaptation to the \(2\times 2\) block structure).  The only difference in our case compared to the proof of~\cite[Theorem 5.2]{MR4009717} is that the martingales \(\dif M_i(t)\) (cf.~\cite[Eq.~(5.4)]{MR4009717}) are defined as
\begin{equation}
    \label{eq:diffmart}
    \dif M_i(t):=\frac{\dif b_i^z(t)}{\sqrt{n}}, \qquad \abs{i}\le n,
\end{equation}
with \(\{b_i^z\}_{i\in [n]}\) having non trivial covariance~\eqref{eq:recallBBB}. This fact does not play any role in that proof, since the only information about \(\dif {\bm M}=\{\dif M_i\}_{\abs{i}\le n}\) used in~\cite[Theorem 5.2]{MR4009717} is that it has bounded quadratic variation and that \({\bm M}(t)\) is \(\gamma\)-H\"older continuous for any \(\gamma\in (0,1/2)\), which is clearly the case for \(\dif {\bm M}\) defined in~\eqref{eq:diffmart}.\qed%

\appendix

\section{The interpolation process is well defined}\label{sec:wp}
We recall that the eigenvectors of \(H^z\) are of the form \({\bm w}_{\pm i}^z=({\bm u}_i^z,\pm {\bm v}_i^z)\) for any \(i\in [n]\), as a consequence of the symmetry of the spectrum of \(H^z\) with respect to zero. Consider the matrix flow
\begin{equation}
    \label{eq:mflowapp}
    \dif X_t=\frac{\dif B_t}{\sqrt{n}}, \qquad X_0=X,
\end{equation}
with \(B_t\) being a standard real matrix valued Brownian motion. Let \(H_t^z\) denote the Hermitisation of \(X_t-z\), and \(\{{\bm w}_i^z(t)\}_{\abs{i}\le n}\) its eigenvectors. We recall that the eigenvectors \(\{{\bm w}_i^z(t)\}_{\abs{i}\le n}\) are almost surely well defined, since \(H_t^z\) does not have multiple eigenvalues almost surely by~\eqref{eq:gpert}. We set the eigenvectors equal to zero where they are not well defined. Recall the definitions of the coefficients \(\Lambda_{ij}^z(t)\), \(\mathring{\Lambda}_{ij}^z(t)\) from~\eqref{eq:defbiglam},~\eqref{eq:deflamzz} and~\eqref{eq:ringlamthe}, respectively. Set
\[
    \Delta_n:= \set*{(x_i)_{\abs{i}\le n}\in \R^{2n}\given 0<x_1<\cdots< x_n, \, x_{-i}=-x_i, \, \forall i\in [n]},
\]
and let \(C(\R_+,\Delta_n)\) be the space of continuous functions \(f:\R_+\to \Delta_n\). Let \(\omega_E>0\) be the exponent in~\eqref{eq:ringlamthe}, and let \(\omega_r>0\) be such that \(\omega_r\ll \omega_E\). In this appendix we prove that for any \(\alpha\in [0,1]\) the system of SDEs
\begin{equation}\label{eq:impnewDBMwp}
    \dif x_i^z(t,\alpha)=\frac{\dif \mathring{b}_i^z(t)}{\sqrt{n(1+n^{-\omega_r})}}+\frac{1}{2n}\sum_{j\ne i} \frac{1+\alpha\mathring{\Lambda}_{ij}^z(t)}{x_i^z(t,\alpha)-x_j^z(t,\alpha)}\dif t, \quad x_i^z(0,\alpha)=x_i(0), \quad \abs{i}\le n,
\end{equation}
with \({\bm x}(0)\in \Delta_n\), admits a strong solution for any \(t\ge 0\). For \(T>0\), by~\eqref{eq:ringmat}, the martingales \(\{\mathring{b}_i^z\}_{\abs{i}\le [n]}\), defined on a filtration \((\widetilde{\mathcal{F}}_{b,t})_{0\le t\le T}\), are such that \(\mathring{b}_{-i}^z=-\mathring{b}_i^z\) for \(i\in [n]\), and that
\begin{equation}\label{eq:covwp}
    \Exp*{\dif \mathring{b}_i^z \dif \mathring{b}_j^z\given\widetilde{\mathcal{F}}_{b,t}}=\frac{\delta_{i,j}-\delta_{i,-j}+\mathring{\Lambda}_{ij}^z(t)}{2} \dif t, \qquad \abs{i},\abs{j}\le n.
\end{equation}

The main result of this section is Proposition~\ref{prop:wp} below.
Its proof follows closely~\cite[Proposition 5.4]{MR4009717}, which is inspired by the proof of~\cite[Lemma~4.3.3]{MR2760897}. We nevertheless present the proof of Proposition~\ref{prop:wp} for completeness, explaining the differences compared with~\cite[Proposition 5.4]{MR4009717} as a consequence of the correlation in~\eqref{eq:covwp}.

\begin{proposition}\label{prop:wp}
    Fix any \(z\in \C\), and let \({\bm x}(0)\in \Delta_n\). Then for any fixed  \(\alpha\in [0,1]\) there exists a unique strong solution \({\bm x}(t,\alpha)=
    {\bm x}^z(t,\alpha)\in C(\R_+,\Delta_n)\) to the system of SDE~\eqref{eq:impnewDBMwp} with initial condition \({\bm x}(0)\).
\end{proposition}

We will mostly omit the \(z\)-dependence since the analysis of~\eqref{eq:impnewDBMwp} is done for any fixed \(z\in\C\); in particular, we will use the notation \(\mathring{\Lambda}_{ij}=\mathring{\Lambda}_{ij}^z\). By~\eqref{eq:defbiglam},~\eqref{eq:deflamzz} and~\eqref{eq:ringlamthe} it follows that \(\mathring{\Lambda}_{ij}(t)=\mathring{\Lambda}_{ji}(t)\), and that \(\abs{\mathring{\Lambda}_{ij}(t)}\le n^{-\omega_E}\), for any \(t\ge 0\).

\begin{proof}
    We follow the notations used in the proof of~\cite[Proposition 5.4]{MR4009717} to make  the comparison clearer.
    Moreover, we do not keep track of the \(n\)-dependence of the constants, since throughout the proof \(n\) is fixed. By a simple time rescaling, we rewrite the process~\eqref{eq:impnewDBMwp} as
    \begin{equation}
        \label{eq:impnewDBMwpresc}
        \dif x_i(t,\alpha)=\dif \mathring{b}_i(t)+\frac{1}{2}\sum_{j\ne i} \frac{1+\theta_{ij}(t)}{x_i(t,\alpha)-x_j(t,\alpha)}\dif t, \qquad \abs{i}\le n,
    \end{equation}
    where \(\theta_{ij}(t):= \alpha\mathring{\Lambda}_{ij}(1+n^{-\omega_r})+n^{-\omega_r}\) is such that \(\theta_{ij}(t)=\theta_{ji}(t)\). Note that \(c_1\le \theta_{ij}(t)\le c_2\) for any \(t\ge 0\) and \(\alpha\in [0,1]\), with \(c_1=n^{-\omega_r}/2\), \(c_2=1\). For any \(\epsilon>0\) define the bounded Lipschitz function \(\phi_\epsilon\colon\R\to\R\) as
    \[
        \phi_\epsilon(x):=  \begin{cases}
            x^{-1},          & \abs{x}\ge \epsilon, \\
            \epsilon^{-2} x, & \abs{x}< \epsilon,
        \end{cases}
    \]
    that cuts off the singularity of \(x^{-1}\) at zero.

    Introduce the system of cut-off SDEs
    \begin{equation}\label{eq:cphi}
        \dif x_i^\epsilon(t,\alpha)=\dif \mathring{b}_i(t)+\frac{1}{2}\sum_{j\ne i} (1+\theta_{ij}(t))\phi_\epsilon(x_i^\epsilon(t,\alpha)-x_j^\epsilon(t,\alpha)) \dif t, \quad \abs{i}\le n,
    \end{equation}
    which admits a unique strong solution (see e.g.~\cite[Theorem 2.9 of Section 5]{MR0917065}) as a consequence of \(\phi_\epsilon\) being Lipschitz and the fact that \(\dif \mathring{\bm b}=(\mathring{C})^{1/2} \dif \mathfrak{w}\) (see~\eqref{eq:circbm}). Define the stopping times
    \begin{equation}
        \label{eq:stpt}
        \tau_\epsilon=\tau_\epsilon(\alpha):=  \inf \set*{t\given\min_{\abs{i},\abs{j}\le n}\abs*{x_i^\epsilon(t,\alpha)-x_j^\epsilon(t,\alpha)}\le \epsilon \quad\text{or}\quad \norm{ {\bm x}^\epsilon(t,\alpha)}_\infty\ge \epsilon^{-1}}.
    \end{equation}
    By strong uniqueness we have that \({\bm x}^{\epsilon_2}(t,\alpha)={\bm x}^{\epsilon_1}(t,\alpha)\) for any \(t\in [0,\tau_{\epsilon_2}]\) if \(0<\epsilon_1<\epsilon_2\). Note that \(\tau_{\epsilon_2}\le \tau_{\epsilon_1}\) for \(\epsilon_1<\epsilon_2\), thus the limit \(\tau=\tau(\alpha):=\lim_{\epsilon\to 0} \tau_\epsilon(\alpha)\) exists, and \({\bm x}(t,\alpha):=\lim_{\epsilon\to 0} {\bm x}^\epsilon(t,\alpha)\) defines a strong solution to~\eqref{eq:impnewDBMwpresc} on \( \interval{co}{0,\tau} \). Moreover, by continuity in time, \({\bm x}(t,\alpha)\) remains ordered as \(0<x_1(t,\alpha)<\cdots< x_n(t,\alpha)\) and \(x_{-i}(t,\alpha)=-x_i(t,\alpha)\) for \(i\in [n]\). Additionally, for the square of the \(\ell^2\)-norm \(\norm{ {\bm x}}_2^2=\sum_i x_i^2\) a simple calculation shows that
    \begin{equation}
        \label{eq:simpc}
        \dif \norm{ {\bm x}(t,\alpha)}_2^2=\frac{1}{2}\left(\sum_{j\ne i} (1+\theta_{ij})+\sum_{\abs{i},\abs{j}\le n} \mathring{\Lambda}_{ij}\right)\dif t+\dif M_1,
    \end{equation}
    with \(\dif M_1\) being a martingale term. This implies that \(\E \norm{ {\bm x}(t\wedge s)}_2^2\le c(1+t)\) for any stopping time \(s<\tau\) and
    for any \(t\ge 0\), where \(c\) depends on \(n\).

    Let \(a>0\) be a large constant that we will choose later in the proof, and define \(a_k\) recursively by \(a_0:=a\), \(a_{k+1}:=a_k^5\) for \(k\ge 0\). Consider the Lyapunov function
    \begin{equation}
        \label{eq:lyafun}
        f({\bm x}):=-2\sum_{k\ne l} a_{\abs{k-l}}\log \abs{x_k-x_l}.
    \end{equation}
    Then by It\^{o}'s formula we get
    \begin{equation}
        \label{eq:itoA}
        \dif f({\bm x})=A({\bm x}(t,\alpha))\dif t+\dif M_2(t),
    \end{equation}
    with
    \begin{equation}
        \label{eq:defAxta}
        \begin{split}
            A({\bm x}(t,\alpha)):={}&-2\sum_{l\ne i, j\ne i}\frac{(1+\theta_{ij})a_{\abs{i-l}}}{(x_i(t,\alpha)-x_l(t,\alpha))(x_i(t,\alpha)-x_j(t,\alpha))}+\sum_{\abs{i}\le n}\frac{a_{\abs{2i}}}{(2x_i(t,\alpha))^2} \\
            &\quad+\sum_{j\ne i} \frac{a_{\abs{i-j}}(1+\mathring{\Lambda}_{ii}(t)-\mathring{\Lambda}_{ij}(t))}{(x_i(t,\alpha)-x_j(t,\alpha))^2},
        \end{split}
    \end{equation}
    where \(\dif M_2\) is a martingale given by
    \[
        \dif M_2(t)=-2\sum_{j\ne i}\frac{a_{\abs{i-j}}\dif \mathring{b}_i(t)}{x_i(t,\alpha)-x_j(t,\alpha)}.
    \]
    In the following we will often omit the time dependence. Note that the term in~\eqref{eq:defAxta} containing \(\mathring{\Lambda}_{ii}-\mathring{\Lambda}_{ij}\) is new compared to~\cite[Eq.~(5.39)]{MR4009717}, since it comes from the correlation of the martingales \(\{\mathring{b}_i\}_{\abs{i}\le n}\), whilst in~\cite[Eq.~(5.39)]{MR4009717} i.i.d.\ Brownian motions have been considered. In the remainder of the proof we show that the term \(\mathring{\Lambda}_{ii}-\mathring{\Lambda}_{ij}\) is negligible using the fact that \(\abs{\mathring{\Lambda}_{ij}}\le n^{-\omega_E}\), and so that this term can be absorbed in the negative term coming from the first sum in the r.h.s.\ of~\eqref{eq:defAxta} for \(l=j\).

    We now prove that \(A({\bm x}(t,\alpha))\le 0\) if \(a>0\) is sufficiently large. Firstly, we write \(A({\bm x}(t,\alpha))\) as
    \begin{equation}
        \label{eq:rewA}
        \begin{split}
            A({\bm x}(t,\alpha))&=-2\sum_{\substack{l\ne i, j\ne i \\ j\ne l}}\frac{(1+\theta_{ij})a_{\abs{i-l}}}{(x_i-x_l)(x_i-x_j)}-\sum_{j\ne \pm i}\frac{a_{\abs{i-j}}(1+2\theta_{ij}-\mathring{\Lambda}_{ii}+\mathring{\Lambda}_{ij})}{(x_i-x_j)^2} \\
            &\quad -2\sum_{\abs{i}\le n} \frac{a_{\abs{2i}}(\theta_{-i,i}-\mathring{\Lambda}_{ii})}{(2x_i)^2}.
        \end{split}
    \end{equation}
    Then, using that the first sum in~\eqref{eq:rewA} is non-positive for \((i-l)(i-j)>0\), and that \(c_1\le \theta_{ij}\le c_2\), with \(c_1=n^{-\omega_r}\), we bound \(A({\bm x}(t,\alpha))\) as follows
    \begin{equation}
        \label{eq:revA}
        A({\bm x}(t,\alpha))\le -2(1+c_2)\sum_{\abs{i}\le n}\sum_{(i-l)(i-j)<0}\frac{a_{\abs{i-l}}}{(x_i-x_l)(x_i-x_j)}-c_1 \sum_{j\ne i}\frac{a_{\abs{i-j}}}{(x_i-x_j)^2}-\sum_{j\ne \pm i}\frac{a_{\abs{i-j}}}{(x_i-x_j)^2}.
    \end{equation}
    In~\eqref{eq:revA} we used that
    \[
        \theta_{ij}-\mathring{\Lambda}_{ii}+\mathring{\Lambda}_{ij}\ge \frac{c_1}{2}, \qquad \theta_{-i,i}-\mathring{\Lambda}_{ii}\ge \frac{c_1}{2},
    \]
    since \(\theta_{ij}\ge c_1=n^{-\omega_r}\) and \(\abs{\mathring{\Lambda}_{ij}}\le n^{-\omega_E}\), where \(\omega_r\ll \omega_E\). This shows that the correlations of the martingales \(\{\mathring{b}_i\}_{\abs{i}\le n}\) is negligible. Note that the r.h.s.\ of~\eqref{eq:revA} has exactly the same form as~\cite[Eq.~(5.42)]{MR4009717}, since the third term in~\eqref{eq:revA} is non-positive. Hence, following exactly the same computations as in~\cite[Eqs.~(5.43)--(5.46)]{MR4009717}, choosing \(a>n^{10}\), we conclude that
    \begin{equation}
        \label{eq:boundonA}
        A({\bm x}(t,\alpha))\le \left[\frac{2(1+c_2)}{a}-c_1\right]\sum_{j\ne i}\frac{a_{\abs{i-j}}}{(x_i-x_j)^2},
    \end{equation}
    which is negative for \(a\) sufficiently large.

    Fix \(a>0\) large enough so that \(A({\bm x}(t,\alpha))\le 0\), then for any stopping time \(s<\tau\), and any \(t\ge 0\) we have
    \begin{equation}
        \E [f({\bm x}(t\wedge s,\alpha))]\le \E  [f({\bm x}(0,\alpha))].
    \end{equation}
    Hence, by~\cite[Eqs.~(5.48)--(5.49)]{MR4009717}, using that \(\E \norm{ {\bm x}(t\wedge \tau_\epsilon)}_2^2\le c(1+t)\), it follows that
    \[
        \log(\epsilon^{-1})\Prob (\tau_\epsilon< t)\le c,
    \]
    and so that \(\Prob (\tau<t)=0\), letting \(\epsilon\to 0\). Since \(t\ge 0\) is arbitrary, this implies that \(\Prob (\tau<+\infty)=0\), i.e.~\eqref{eq:impnewDBMwpresc} has a unique strong solution on \( (0, \infty) \) such that \({\bm x}(t,\alpha)\in \Delta_n\) for any \(t\ge 0\) and \(\alpha\in [0,1]\).
\end{proof}

Additionally, by a similar argument as in~\cite[Proposition~5.5]{MR4009717}, we conclude the following lemma.

\begin{lemma}\label{lem:lip}
    Let \({\bm x}(t,\alpha)\) be the unique strong solution of~\eqref{eq:impnewDBMwp} with initial data \({\bm x}(0,\alpha)\in \Delta_n\), for any \(\alpha\in [0,1]\), and assume that there exists \(L>0\) such that \(\norm{ {\bm x}(0,\alpha_1)-{\bm x}(0,\alpha_2)}_2\le L \abs{\alpha_1-\alpha_2}\), for any \(\alpha_1,\alpha_2\in [0,1]\). Then \({\bm x}(t,\alpha)\) is Lipschitz in \(\alpha\in [0,1]\) for any \(t\ge 0\) on an event \(\Omega\) such that \(\Prob (\Omega)=1\), and its derivative satisfies
    \begin{equation}
        \begin{split}
            \partial_\alpha x_i(t,\alpha)=\partial_\alpha x_i(0,\alpha)+\frac{1}{2 n}\int_0^t \sum_{j\ne i} \frac{[1+\alpha\mathring{\Lambda}_{ij}(s)][\partial_\alpha x_j(s,\alpha)-\partial_\alpha x_i(s,\alpha)]}{(x_i(s,\alpha)-x_j(s,\alpha))^2} \dif s \\
            +\frac{1}{2 n}\int_0^t \sum_{j\ne i} \frac{\mathring{\Lambda}_{ij}(s)}{x_i(s,\alpha)-x_j(s,\alpha)} \dif s.
        \end{split}
    \end{equation}
\end{lemma}

\section{Derivation of the DBM for singular values in the real case}\label{sec:derdbm}
Let \(X\) be an \(n\times n\) real random matrix, and define \(Y^z:=  X-z\). Consider the matrix flow~\eqref{eq:mflowapp} defined on a probability space \(\Omega\) equipped with a filtration \((\mathcal{F}_t)_{0\le t\le T}\), and denote by \(H_t^z\) the Hermitisation of \(X_t-z\). We now derive~\eqref{eq:impnewDBM}, under the assumption that the eigenvalues are all distinct. This derivation is easily made complete by the argument in the proof Proposition~\ref{prop:wpmainDBM} in Section~\ref{sec:wpe}.

Let \(\{\lambda^z_i(t),-\lambda^z_i(t)\}_{i\in [n]}\) be the eigenvalues of \(H_t^z\), and denote by \(\{{\bm w}^z_i(t), {\bm w}^z_{-i}(t)\}_{i\in [n]}\) their corresponding orthonormal eigenvectors, i.e.\ for any \(i,j\in[n]\), omitting the \(t\)-dependence, we have that
\begin{equation}
    \label{eigeq}
    H^z{\bm w}^z_{\pm i}=\pm\lambda^z_i {\bm w}_{\pm i}^z, \qquad ({\bm w}_i^z)^\ast {\bm w}^z_j=\delta_{ij},\qquad ({\bm w}_i^z)^\ast {\bm w}_{-j}^z=0.
\end{equation}
In particular, for any \(i\in [n]\), by the block structure of \(H^z\) it follows that
\begin{equation}
    \label{singval}
    {\bm w}^z_{\pm i}=( {\bm u}^z_i, \pm {\bm v}^z_i), \qquad Y^z{\bm v}^z_i=\lambda^z_i {\bm u}^z_i, \qquad (Y^z)^\ast {\bm u}^z_i=\lambda^z_i {\bm v}^z_i.
\end{equation}
Moreover, since \(\{{\bm w}^z_{\pm i}\}_{i=1}^n\) is an orthonormal basis, we conclude that
\begin{equation}
    \label{hnorm}
    ({\bm u}^z_i)^\ast {\bm u}^z_i=({\bm v}^z_i)^\ast {\bm v}^z_i=\frac{1}{2}.
\end{equation}

In the following, for any fixed entry \(x_{ab}\) of \(X\), we denote the derivative in the \(x_{ab}\) direction by
\begin{equation}
    \dot{f}:=\frac{\partial f}{\partial x_{ab}},
\end{equation}
where \(f=f(X)\) is a function of the matrix \(X\). From now on we only consider positive indices \(1\le i\le n\). We may also drop the \(z\) and \(t\) dependence to make our notation lighter. For any \(i,j\in [n]\), differentiating~\eqref{eigeq} we obtain
\begin{align}
    \label{findl1}
    \dot{H} {\bm w}_i+H\dot{\bm w}_i                         & =\dot{\lambda}_i {\bm w}_i+\lambda_i \dot{\bm w}_i, \\
    \label{findl2}
    \dot{\bm w}_i^\ast{\bm w}_j+{\bm w}_i^\ast\dot{\bm w}_j  & =0,                                                 \\
    \label{findl3}
    {\bm w}_i^\ast \dot{\bm w}_i+\dot{\bm w}_i^\ast{\bm w}_i & =0.
\end{align}
Note that~\eqref{findl3} implies that \(\Re[{\bm w}_i^\ast \dot{\bm w}_i]=0\). Moreover, since the eigenvectors are defined modulo a phase, we can choose eigenvectors such that \(\Im[{\bm w}_i^\ast \dot{\bm w}_i]=0\) for any \(t\ge 0\) hence \({\bm w}^\ast_i{\bm w}_i=0\). Then, multiplying~\eqref{findl1} by \({\bm w}_i^\ast\) we conclude that
\begin{equation}
    \label{dereig}
    \dot{\lambda}_i={\bm u}_i^\ast\dot{Y}{\bm v}_i+{\bm v}_i^\ast\dot{Y}^\ast{\bm u}_i.
\end{equation}
Moreover, multiplying~\eqref{findl1} by \({\bm w}_j^\ast\), with \(j\ne i\), and by \({\bm w}_{-j}^\ast\), we get
\begin{equation}
    \label{diffl}
    (\lambda_i-\lambda_j){\bm w}_j^\ast\dot{\bm w}_i={\bm w}_j^\ast \dot{H} {\bm w}_i, \qquad (\lambda_i+\lambda_j){\bm w}_{-j}^\ast\dot{\bm w}_i={\bm w}_{-j}^\ast \dot{H} {\bm w}_i,
\end{equation}
respectively.
By~\eqref{findl3} and \({\bm w}_i^\ast {\bm w}_i=0\) it follows that
\begin{equation}
    \label{dereigv1}
    \dot{\bm w}_i=\sum_{\substack{j \in [n], \\ j\ne i}} ({\bm w}_j^\ast\dot{\bm w}_i) {\bm w}_j+\sum_{j\in [n]} ({\bm w}_{-j}^\ast \dot{\bm w}_i){\bm w}_{-j},
\end{equation}
hence, by~\eqref{diffl}, we conclude
\begin{equation}
    \label{dereigv2}
    \dot{\bm w}_i=\sum_{j\ne i} \frac{{\bm v}_j^\ast\dot{Y}^\ast{\bm u}_i+{\bm u}_j^\ast\dot{Y}{\bm v}_i}{\lambda_i-\lambda_j} {\bm w}_j+\sum_j \frac{{\bm u}_j^\ast\dot{Y}{\bm v}_i-{\bm v}_j^\ast\dot{Y}^\ast {\bm u}_i}{\lambda_i+\lambda_j} {\bm w}_{-j}.
\end{equation}

Throughout this appendix we use the convention that for any vectors \({\bm v}\in \C^n\) we denote its entries by \(v(a)\), with \(a\in [n]\). By~\eqref{dereig}--\eqref{dereigv2} it follows that
\begin{equation}
    \label{eq:firstderlam}
    \frac{\partial \lambda_i}{\partial x_{ab}}=2\Re[u_i^\ast(a)v_i(b)],
\end{equation}
and that
\[
    \begin{split}
        \frac{\partial w_i}{\partial x_{ab}}(k)&=\sum_{j\ne i} \left[ \frac{u_j^\ast(a)v_i(b)+v_j^\ast(b)u_i(a)}{\lambda_i-\lambda_j}w_j(k)+  \frac{u_j^\ast(a)v_i(b)-v_j^\ast(b)u_i(a)}{\lambda_i+\lambda_j}w_{-j}(k)\right] \\
        &\quad+\frac{u_i^\ast(a)v_i(b)-v_i^\ast(b)u_i(a)}{2\lambda_i}w_{-i}(k).
    \end{split}
\]
By Ito's formula we have that
\begin{equation}
    \label{eq:itoforder}
    \dif \lambda_i=\sum_{ab} \frac{\partial \lambda_i}{\partial x_{ab}} \dif x_{ab}+\frac{1}{2}\sum_{ab}\sum_{kl} \frac{\partial^2 \lambda_i}{\partial x_{ab}\partial x_{kl}}\dif x_{ab} \dif  x_{kl}.
\end{equation}
Then we compute
\begin{equation}
    \label{eq:secondderlam}
    \begin{split}
        \frac{\partial^2 \lambda_i}{\partial x_{ab}\partial x_{kl}}&=2\Re\left[\frac{\partial v_i^\ast}{\partial x_{ab}}(l)u_i(k)+v_i^\ast(l)\frac{\partial u_i}{\partial x_{ab}}(k)\right] \\
        &=2\Re\Bigg[\sum_{j\ne i} \left[ \frac{u_j(a)v_i^\ast(b)+v_j(b)u_i^\ast(a)}{\lambda_i-\lambda_j}v_j^\ast(l)u_i(k)-  \frac{u_j(a)v_i^\ast(b)-v_j(b)u_i^\ast(a)}{\lambda_i+\lambda_j}v^\ast_j(l)u_i(k)\right] \\
            &\quad-\frac{u_i(a)v_i^\ast(b)-v_i(b)u_i^\ast(a)}{2\lambda_i}v^\ast_i(l)u_i(k)+\frac{u_i^\ast(a)v_i(b)-v_i^\ast(b)u_i(a)}{2\lambda_i}u_i(k)v_i^\ast(l) \\
            &\quad +\sum_{j\ne i} \left[ \frac{u_j^\ast(a)v_i(b)+v_j^\ast(b)u_i(a)}{\lambda_i-\lambda_j}u_j(k)v_i^\ast(l)+  \frac{u_j^\ast(a)v_i(b)-v_j^\ast(b)u_i(a)}{\lambda_i+\lambda_j}u_j(k)v_i^\ast(l)\right]\Bigg].
    \end{split}
\end{equation}
Hence, combining~\eqref{eq:firstderlam}--\eqref{eq:secondderlam}, we finally conclude that
\begin{equation}
    \label{eq:newdbmrc}
    \begin{split}
        \dif \lambda^z_i &=\frac{\dif b_i^z}{\sqrt{n}}+\frac{1}{2n}\sum_{j\ne i} \left[\frac{1+4\Re[\braket{ \overline{u^z_j},u^z_i}\braket{ v^z_i,\overline{v_j^z}}]}{\lambda^z_i-\lambda^z_j}+\frac{1+4\Re[\braket{ \overline{u_j^z},u^z_i}\braket{ v_i^z,-\overline{v_j^z}}]}{\lambda^z_i+\lambda^z_j}\right] \dif t \\
        &\quad + \frac{1+4\Re[\braket{ \overline{u_i^z},u^z_i}\braket{ v^z_i,-\overline{v_i^z}}]}{4n\lambda_i^z} \dif t.
    \end{split}
\end{equation}
In~\eqref{eq:newdbmrc} we used the convention that for any vector \({\bm v}\in \C^n\) by \(\overline{\bm v}\) we denote the vector with entries \(\overline{v}(a)=\overline{v(a)}\), for any \(a\in [n]\). The driving martingales in~\eqref{eq:newdbmrc} are defined as
\begin{equation}
    \label{eq:defBBB}
    \dif b^z_i:=\dif B^z_{ii}+\dif \overline{B^z_{ii}},\quad \text{with} \quad  \dif B^z_{ij}:= \sum_{ab} (u_i^z)^\ast(a) \dif B_{ab} v^z_j(b),
\end{equation}
with \(B=B_t\) the matrix valued Brownian motion in~\eqref{eq:mflowapp}, and their covariance given by
\begin{equation}
    \label{eq:halcov}
    \Exp*{ \dif b_i^z \dif b_j^z\given\mathcal{F}_t}=\frac{\delta_{ij}+4\Re\left[\braket{ \overline{u_j^z},u_i^z}\braket{ v_i^z, \overline{v_j^z}} \right]}{2} \dif t.
\end{equation}
Note that \(\{b_i^z\}_{i\in [n]}\) defined in~\eqref{eq:defBBB} are not Brownian motions, as a consequence of the non deterministic quadratic variation~\eqref{eq:halcov}.

\printbibliography%

\end{document}

%% file: packages.tex
\usepackage[T1]{fontenc} 
\usepackage[english]{babel}
\usepackage[p,osf]{cochineal}
\usepackage[varqu,varl,var0]{inconsolata}
\usepackage[scale=.95,type1]{cabin} 
\usepackage{mathalfa}
\usepackage[utf8]{inputenc} 
\usepackage{amsthm}
\usepackage{amsmath}
\usepackage{amssymb}
\usepackage{amsfonts}
\usepackage{amsaddr}
\usepackage{xspace}
\usepackage{enumitem} 
\usepackage{csquotes}
\usepackage{xcolor}
\usepackage[colorlinks]{hyperref}
\definecolor{intcolor}{HTML}{CA0020}
\definecolor{extcolor}{HTML}{0571B0}
\hypersetup{breaklinks,linkcolor=intcolor,citecolor=intcolor,filecolor=black,urlcolor=extcolor,menucolor=intcolor,runcolor=extcolor}
\usepackage{graphicx}
\usepackage{mathtools}
\mathtoolsset{centercolon}
\usepackage{bm}
\usepackage{tikz}
\usepackage{pgf}
\usetikzlibrary{positioning}
\usepackage{xparse}
\ExplSyntaxOn
\NewDocumentCommand{\interval}{sO{}mm}
 {
  \IfBooleanTF{#1}
   {
    \yannick_interval:NNnnn \left \right { } { #3 } { #4 }
   }
   {
    \yannick_interval:NNnnn \mathopen \mathclose { #2 } { #3 } { #4 }
   }
 }

\cs_new_protected:Nn \yannick_interval:NNnnn
 {
  \str_case:nn { #4 }
   {
    {oo}{#1#3\c_yannick_left_open_tl #5 #2#3\c_yannick_right_open_tl}
    {co}{#1#3[                       #5 #2#3\c_yannick_right_open_tl}
    {oc}{#1#3\c_yannick_left_open_tl #5 #2#3]}
    {cc}{#1#3[                       #5 #2#3]}
   }
 }
\tl_const:Nn \c_yannick_left_open_tl { ( }
\tl_const:Nn \c_yannick_right_open_tl { ) }
\ExplSyntaxOff

%% file: notations.tex
\DeclareMathOperator{\Tr}{Tr}
\DeclareMathOperator{\erfc}{erfc}
\DeclareMathOperator{\sgn}{sgn}
\DeclareMathOperator{\supp}{supp}

\DeclareMathOperator{\E}{\mathbf{E}}
\DeclareMathOperator{\Prob}{\mathbf{P}}

\DeclareMathOperator{\Spec}{Spec}
\newcommand{\ov}{\overline}
\newcommand{\ii}{\mathrm{i}}

\ifdefined\C
\renewcommand{\C}{\mathbf{C}}
\else
\newcommand{\C}{\mathbf{C}}
\fi
\newcommand{\HC}{\mathbf{H}}
\renewcommand{\SS}{\mathcal{S}}
\newcommand{\un}{\underline}
\newcommand{\vx}{\bm{x}}
\newcommand{\vy}{\bm{y}}

\newcommand{\cN}{\mathcal{N}}
\newcommand{\wt}{\widetilde}
\newcommand{\wh}{\widehat}
\newcommand{\cE}{\mathcal{E}}
\newcommand{\R}{\mathbf{R}}
\newcommand{\N}{\mathbf{N}}
\newcommand{\Z}{\mathbf{Z}}
\newcommand{\sump}{\sideset{}{'}\sum}

\newcommand{\DD}{\mathbf{D}}

\newcommand{\cF}{\mathcal{F}}
\newcommand{\cO}{\mathcal{O}}
\newcommand{\co}{{\scriptstyle\mathcal{O}}}
\newcommand{\cB}{\mathcal{B}}
\newcommand{\cP}{\mathcal{P}}
\newcommand{\cQ}{\mathcal{Q}}
\newcommand{\cT}{\mathcal{T}}

\newcommand{\ed}{\mathfrak{e}}
\newcommand{\sym}{\mathrm{sym}}
\newcommand{\dif}{\operatorname{d}\!{}}
\DeclarePairedDelimiter{\braket}{\langle}{\rangle}%
\DeclarePairedDelimiter{\abs}{\lvert}{\rvert}%
\DeclarePairedDelimiter{\norm}{\lVert}{\rVert}%
\providecommand\given{}
\newcommand\SetSymbol[1][]{\nonscript\:#1\vert\allowbreak\nonscript\:\mathopen{}}
\DeclarePairedDelimiterX{\tuple}[1](){\renewcommand\given{\SetSymbol[\delimsize]}#1}
\DeclarePairedDelimiterX{\set}[1]{\{}{\}}{\renewcommand\given{\SetSymbol[\delimsize]}#1}
\DeclarePairedDelimiterX{\Set}[1]\{\}{\renewcommand\given{\SetSymbol[\delimsize]}#1}
\DeclarePairedDelimiterXPP{\landauO}[1]{\cO}(){}{#1}
\DeclarePairedDelimiterXPP{\landauo}[1]{\co}(){}{#1}
\DeclarePairedDelimiterXPP{\landauok}[1]{\co_k}(){}{#1}
\DeclarePairedDelimiterXPP{\landauOprec}[1]{\cO_\prec}(){}{#1}
\DeclarePairedDelimiterXPP{\Exp}[1]{\E}[]{}{\renewcommand\given{\SetSymbol[\delimsize]}#1}
\DeclareFontFamily{U}{mathx}{\hyphenchar\font45}
\DeclareFontShape{U}{mathx}{m}{n}{
      <5> <6> <7> <8> <9> <10>
      <10.95> <12> <14.4> <17.28> <20.74> <24.88>
      mathx10
      }{}
\DeclareSymbolFont{mathx}{U}{mathx}{m}{n}
\DeclareFontSubstitution{U}{mathx}{m}{n}
\DeclareMathAccent{\widecheck}{0}{mathx}{"71}

%% file: bibliography_setup.tex
\usepackage[giveninits=true,url=false,doi=false,isbn=false,eprint=true,datamodel=mrnumber,sorting=nty,maxcitenames=4,maxbibnames=99,backref=false,block=space,backend=biber,bibstyle=phys]{biblatex}
\AtEveryBibitem{\clearfield{month}}
\AtEveryCitekey{\clearfield{month}}
\renewbibmacro{in:}{}
\DeclareFieldFormat[article]{title}{\emph{#1}}
\DeclareFieldFormat{mrnumber}{\ifhyperref{\href{http://www.ams.org/mathscinet-getitem?mr=#1}{\nolinkurl{MR#1}}}{\nolinkurl{#1}}}
\DeclareFieldFormat{pmid}{\ifhyperref{\href{https://www.ncbi.nlm.nih.gov/pubmed/#1}{\nolinkurl{PMID#1}}}{\nolinkurl{#1}}}
\renewbibmacro*{doi+eprint+url}{%
  \iftoggle{bbx:doi}{\printfield{doi}}{}
  \newunit\newblock%
  \printfield{mrnumber}%
  \newunit\newblock%
  \printfield{pmid}%
  \newunit\newblock%
  \printfield{eprint}%
  \iftoggle{bbx:url}{\usebibmacro{url+urldate}}{}}